\documentclass[10pt]{amsart} 
\usepackage{amssymb}
\usepackage{amscd}
\newcounter{TmpEnumi}
\numberwithin{equation}{section}
\usepackage[colorlinks,linkcolor=black,citecolor=blue]{hyperref}
\usepackage[all]{xy}
\def\today{\number\day\space\ifcase\month\or   January\or February\or
   March\or April\or May\or June\or   July\or August\or September\or
   October\or November\or December\fi\   \number\year}

\theoremstyle{definition}
\newtheorem{thm}{Theorem}[section]
\newtheorem{lem}[thm]{Lemma}
\newtheorem{prp}[thm]{Proposition}
\newtheorem{dfn}[thm]{Definition}
\newtheorem{cor}[thm]{Corollary}

\newtheorem{rmk}[thm]{Remark}
\newtheorem{ntn}[thm]{Notation}
\newtheorem{exa}[thm]{Example}
\newtheorem{pbm}[thm]{Problem}

\newtheorem{qst}[thm]{Question}

\newtheorem{cns}[thm]{Construction}

\newcommand{\beq}{\begin{equation}}
\newcommand{\eeq}{\end{equation}}
\newcommand{\beqa}{\begin{eqnarray*}}
\newcommand{\eeqa}{\end{eqnarray*}}
\newcommand{\bal}{\begin{align*}}
\newcommand{\eal}{\end{align*}}
\newcommand{\bi}{\begin{itemize}}
\newcommand{\ei}{\end{itemize}}
\newcommand{\be}{\begin{enumerate}}
\newcommand{\ee}{\end{enumerate}}

\newcommand{\af}{\alpha}
\newcommand{\bt}{\beta}
\newcommand{\gm}{\gamma}
\newcommand{\dt}{\delta}
\newcommand{\ep}{\varepsilon}
\newcommand{\zt}{\zeta}
\newcommand{\et}{\eta}
\newcommand{\ch}{\chi}
\newcommand{\io}{\iota}
\newcommand{\te}{\theta}
\newcommand{\ld}{\lambda}
\newcommand{\sm}{\sigma}
\newcommand{\kp}{\kappa}
\newcommand{\ph}{\varphi}
\newcommand{\ps}{\psi}
\newcommand{\rh}{\rho}
\newcommand{\om}{\omega}
\newcommand{\ta}{\tau}

\newcommand{\Gm}{\Gamma}

\newcommand{\Ph}{\Phi}
\newcommand{\Ps}{\Psi}

\newcommand{\Q}{{\mathbb{Q}}}
\newcommand{\Z}{{\mathbb{Z}}}
\newcommand{\R}{{\mathbb{R}}}
\newcommand{\N}{{\mathbb{N}}}
\newcommand{\Nz}{{\mathbb{Z}}_{\geq 0}}

\newcommand{\OI}{{\mathcal{O}}_{\infty}}
\newcommand{\OT}{{\mathcal{O}}_{2}}

\newcommand{\btt}{{\widetilde{\bt}}}

\newcommand{\Lt}{{\mathtt{Lt}}}

\newcommand{\id}{{\operatorname{id}}}
\newcommand{\ev}{{\operatorname{ev}}}

\newcommand{\dist}{{\operatorname{dist}}}

\newcommand{\spec}{{\operatorname{sp}}}

\newcommand{\diag}{{\operatorname{diag}}}

\newcommand{\rank}{{\operatorname{rank}}}

\newcommand{\card}{{\operatorname{card}}}
\newcommand{\Aut}{{\operatorname{Aut}}}

\newcommand{\Ad}{{\operatorname{Ad}}}

\newcommand{\T}{{\operatorname{T}}}
\newcommand{\QT}{{\operatorname{QT}}}
\newcommand{\rc}{{\operatorname{rc}}}

\newcommand{\dirlim}{\varinjlim}
\newcommand{\invlim}{\varprojlim}

\newcommand{\andeqn}{\qquad {\mbox{and}} \qquad}

\newcommand{\Wolog}{Without loss of generality}

\newcommand{\ifo}{if and only if}

\newcommand{\ca}{C*-algebra}

\newcommand{\uca}{unital C*-algebra}

\newcommand{\hm}{homomorphism}

\newcommand{\uhm}{unital homomorphism}

\newcommand{\fd}{finite dimensional}

\newcommand{\tst}{tracial state}

\newcommand{\pj}{projection}
\newcommand{\mops}{mutually orthogonal \pj s}
\newcommand{\nzp}{nonzero projection}
\newcommand{\mvnt}{Murray-von Neumann equivalent}

\newcommand{\ct}{continuous}
\newcommand{\cfn}{continuous function}

\newcommand{\chs}{compact Hausdorff space}

\newcommand{\trpc}{tracial Rokhlin property with comparison}
\newcommand{\ucp}{unital completely positive}

\renewcommand{\S}{\subseteq}
\newcommand{\ov}{\overline}
\newcommand{\SM}{\setminus}
\newcommand{\I}{\infty}

\title[Examples and
  Nonexistence Theorems]{Compact
 Group Actions with the Tracial Rokhlin Property II: Examples and
  Nonexistence Theorems}

\author{Javad Mohammadkarimi and N. Christopher Phillips}

\date{6~May 2025}

\address{School of Mathematical Sciences,
 Key Laboratory of MEA (Ministry of Education)
 \& Shanghai Key Laboratory of PMMP,
 East China Normal University, Shanghai 200241, China
\and
 Department of Mathematics, University  of Oregon,
       Eugene OR 97403-1222, USA.}

\thanks{The work of the first author was partially supported by
the Research Center for Operator Algebras
and by the Science and Technology Commission
of Shanghai Municipality (No. 22DZ2229014)
at East China Normal University.
The work of the second author was partially supported by the
Simons Foundation Collaboration Grant for Mathematicians \#587103
and by the US National
Science Foundation under Grants DMS-2055771 and DMS-2400332.}

\begin{document}

\begin{abstract}
In a previous paper, we introduced the
restricted tracial Rokhlin property with comparison,
a ``tracial'' analog of the Rokhlin property for actions
of second countable compact groups
on infinite dimensional simple separable unital C*-algebras.
In this paper, we give three classes of examples of actions
of compact groups which have this property but do not have
the Rokhlin property, or even
finite Rokhlin dimension with commuting towers.
One class consists of infinite tensor products
of finite group actions with the tracial Rokhlin property,
giving actions of the product of the groups involved.
The second class consists of actions of the circle group on
simple unital AT~algebras.
The construction of the third class starts with an action of
the circle on the Cuntz algebra ${\mathcal{O}}_{\infty}$
which has the restricted tracial Rokhlin property with comparison;
by contrast, it is known that there is no action of this group
on ${\mathcal{O}}_{\infty}$ which has finite Rokhlin dimension
with commuting towers.
We can then tensor this action with the trivial action on any
unital purely infinite simple separable nuclear C*-algebra.
One also gets such actions
on certain purely infinite simple separable nuclear C*-algebras
by tensoring the AT~examples with the trivial action on
${\mathcal{O}}_{\infty}$; these are different.

We also discuss other tracial Rokhlin properties for
actions of compact groups, and prove that there is
no direct limit action of the circle group on a simple AF~algebra
which even has the weakest of these properties.
\end{abstract}

\maketitle

\tableofcontents

\section{Introduction}\label{Sec_1X21_Intro}

\indent
The tracial Rokhlin property for actions of finite groups on
C*-algebras was introduced in \cite{phill23}
for the purpose of proving that every simple higher dimensional
noncommutative torus is an AT~algebra (done in \cite{PhtRp2}).
The restricted tracial Rokhlin property with comparison
seems to be the right generalization
to actions of compact groups on simple C*-algebras.
It was introduced in~\cite{MhkPh1},
where the authors, among other things, proved that
for actions with this property, many properties transfer
from the original C*-algebra to the fixed point algebra
and the crossed product.
The main purpose of this paper is to provide examples
of actions of compact groups that have this property
but do not have finite Rokhlin dimension with commuting towers
as defined in~\cite{Gar_rokhlin_2017}.
In particular, we give:
\begin{itemize}
\item
An action of a totally disconnected
infinite compact group on a UHF~algebra.
\item
An action of the circle group on a simple unital AT~algebra.
\item
An action of the circle group on~$\OI$.
\end{itemize}
Each of these examples is one of a more general class of examples
of the same general type, about which we say more below.
These examples demonstrate the differences between
finite Rokhlin dimension with commuting towers and the
restricted tracial Rokhlin property with comparison.
Our third example is an action of $S^1$ on~$\OI$,
and, by Theorem~4.6 of~\cite{HrsPh1},
or by Corollary 4.23 of~\cite{Gar_rokhlin_2017}
(using quite different methods),
there is {\emph{no}} action of $S^1$
on $\OI$ which has finite Rokhlin dimension with commuting towers.
In fact, we get an action
that has the restricted tracial Rokhlin property with comparison
on every unital purely infinite simple separable nuclear \ca.

As with finite groups, Rokhlin actions of compact groups are rare,
especially if the group is connected.
For example, by Theorem 3.3(3) of~\cite{GardKKcirc},
if $A$ is unital, $K_0 (A)$ is finitely generated,
and $A$ admits an action of the circle with the Rokhlin property,
then $K_0 (A) \cong K_1 (A)$.
Our examples suggest that actions with the
restricted tracial Rokhlin property with comparison are much more common.

We also give a number of open problems.

The paper is organized as follows.
In the rest of this section, we present some notation.
Section~\ref{S_2795_N_TRP} contains the
definition of the tracial Rokhlin property with comparison
for compact groups (from~\cite{MhkPh1}),
as well as several other versions of this
property that we use in this paper.
We then relate the finite group case of our definition to the
tracial Rokhlin property for finite groups as originally defined
in Definition~1.2 of~\cite{phill23}.
(They are close, but not the same.
They agree on finite \ca{s} with strict comparison.)
In Section~\ref{S_2795_mod_TRP}, we discuss another variant,
the modified tracial Rokhlin property.
It looks promising, and we prove later that some of our examples
have this property, but we have not been able to show
that it implies many permanence properties.

In Section \ref{Sec_3749_Exam_TRPZ2} we consider the action $\af$
action of a totally disconnected
group obtained as the infinite tensor product of actions~$\af^{(n)}$
of finite groups on stably finite simple separable unital \ca{s}.
We show that if $\af^{(n)}$ has the tracial Rokhlin property for
finite groups, and for every~$n$ the tensor product of the first
$n$~algebras has strict comparison,
then $\af$ has the restricted tracial Rokhlin property with comparison.
As a special case, we give an action
of a totally disconnected infinite compact group on
a UHF~algebra which has the tracial Rokhlin property with comparison
and the strong modified tracial Rokhlin property,
but does not have the Rokhlin property,
or even finite Rokhlin dimension with commuting towers.
In the next section, we construct an action of $S^1$
on a simple AT~algebra which has the same properties,
except that we only prove the modified tracial Rokhlin property.
In Section~\ref{Sec_2114_OI} we construct an action of $S^1$
on $\OI$ which has the \trpc{}
but not finite Rokhlin dimension with commuting towers.
Tensoring with the trivial actions on other simple \ca{s},
we obtain examples of actions with the restricted \trpc{}
on arbitrary unital purely infinite simple separable nuclear \ca{s}.
These actions are different from those gotten by
tensoring the actions of Section~\ref{Sec_1908_Exam_TRPS1}
with the trivial action on~$\OI$,
although these tensor products also have the restricted \trpc.

In Section~\ref{Sec_1919_NonE} we give an easy nonexistence result
for direct limit actions of $S^1$ on simple AF~algebras with even the
weakest form of the tracial Rokhlin property we consider.

This paper is the continuation of the first part \cite{MhkPh1}
and constitutes a part of the first author's Ph.D.\  dissertation.

The first author would like to thank the second author
and the University of Oregon for their hospitality
for a long term visit during which a substantial amount of the
work for this paper, and its predecessor~\cite{MhkPh1},
was done.

In the rest of this section, we collect some notation that we need.

The \ca{} of $n \times n$ matrices will be denoted by $M_{n}$.
If $C$ is a \ca,
we write $C_{+}$ for the set of positive elements in~$C$.
We denote the circle group by $S^{1}$, and identify
it with the set of complex numbers of absolute value~$1$.
An action $\alpha \colon G \to \Aut (A)$ of group $G$ on a \ca~$A$
is assumed \ct{} unless stated otherwise.
Also, we denote by $A^{\alpha}$ the fixed point
subalgebra of $A$ under $\alpha$.

We take $\N = \{ 1, 2, \ldots \}$,
and we abbreviate $\N \cup \{ 0 \}$ to $\Nz$.

\begin{ntn}\label{N_1X07_Lt}
If $G$ is a locally compact group, we denote by
${\mathtt{Lt}} \colon G \to \Aut (C_0 (G))$
the action of $G$ on $C_0 (G)$ induced by the action of $G$ on itself
by left translation.
\end{ntn}

Since we sometimes use Cuntz comparison with respect to subalgebras,
if $a, b \in M_n (A)_{+}$, we write $a \precsim_A b$ to mean that
$a$ is Cuntz subequivalent to~$b$ with respect to~$A$.

\section{The tracial Rokhlin property with
 comparison and a naive tracial Rokhlin property}\label{S_2795_N_TRP}

Our definition of the \trpc{} (Definition~\ref{traR}),
applied to finite groups, is formally stronger than the
tracial Rokhlin property for finite groups
(Definition 1.2 of \cite{phill23}), as discussed after
Definition~\ref{traR}.
In this section we address the differences.
The restricted \trpc{} (also in Definition~\ref{traR}) is what we
normally actually use.
It omits
one of the conditions that appears in the definition for finite groups.
Definition~\ref{D_1920_NTRP} below
(the naive tracial Rokhlin property) is given
only to make discussion easier; it is not intended for general use.
It is what one gets by just copying the definition of the
tracial Rokhlin property for finite groups.
We say at the outset that we know of no examples of actions,
even of infinite compact groups,
which have the naive tracial Rokhlin property but not the \trpc,
although we believe they exist.

Since it plays a major role in our definitions,
we recall for convenient reference the definition
of an $(S, F, \varepsilon )$-approximately equivariant
central multiplicative map.
See Definition 1.3 of~\cite{HrsPh1} or
Definition 1.4 of~\cite{MhkPh1}.

\begin{dfn}\label{phillhir}
Let $G$ be a compact group, and let $A$ and $D$ be unital \ca{s}.
Let $\alpha \colon G \to \Aut (A)$
and $\gamma \colon G \to \Aut ( D)$
be actions of $G$ on $A$ and $D$.
Let $S \S D$
and $F \S A$ be subsets, and let $\varepsilon > 0$.
A unital completely positive map
$\varphi \colon (D, \gamma) \to (A, \alpha)$
(or $\varphi \colon D \to A$ when $\gamma$ and $\alpha$ are understood)
is said to be an {\emph{$(S, F, \varepsilon )$-approximately equivariant
central multiplicative map}} if:
\begin{enumerate}
\item\label{Item_1X07_4}
$\| \varphi (x y) - \varphi (x) \varphi (y) \| < \varepsilon$
for all $x, y \in S$.
\item\label{Item_1X07_5}
$\| \varphi (x) a - a \varphi (x) \| < \varepsilon$
for all $x \in S$ and all $a \in F$.
\item\label{equiapprox}
$\mathrm{sup}_{g \in G} \| \varphi ( \gamma_{g} (x)) -
        \alpha_{g} (\ph (x)) \| < \varepsilon$
for all $x \in S$.
\end{enumerate}
\end{dfn}

\begin{dfn}\label{traR}
Let $A$ be an infinite dimensional simple unital \ca,
and let $\alpha \colon G \to \Aut (A)$ be
an action of a second countable compact group $G$ on~$A$.
The action $\alpha$ has the
\emph{restricted tracial Rokhlin property with comparison}
if for every finite set $F \subseteq A$,
every finite set $S \subseteq C (G)$,
every $\varepsilon > 0$, every $x \in A_{+} \setminus \{ 0 \}$,
and every $y \in (A^{\alpha})_{+} \setminus \{ 0 \}$,
there exist a projection $p \in A^{\alpha}$ and a
unital completely positive map $\varphi \colon C (G) \to p A p$
such that the following hold:
\begin{enumerate}
\item\label{Item_893_FS_equi_cen_multi_approx}
$\varphi$ is an $(F, S, \varepsilon)$-approximately equivariant
central multiplicative map.
\item\label{1_pxcompactsets}
$1 - p \precsim_{A} x$.
\item\label{1_pycompactsets}
$1 - p \precsim_{A^{\alpha}} y$.
\item\label{1_ppcompactsets}
$1 - p \precsim_{A^{\alpha}} p$.
\setcounter{TmpEnumi}{\value{enumi}}
\end{enumerate}
We say that $\alpha$ has the
\emph{tracial Rokhlin property with comparison}
if, also assuming $\| x \| = 1$, one can in addition require:
\begin{enumerate}
\setcounter{enumi}{\value{TmpEnumi}}
\item\label{Item_902_pxp_TRP}
$\| p x p \| > 1 - \varepsilon$.
\end{enumerate}
\end{dfn}

In \cite{MhkPh1}, the restricted \trpc{}
is the hypothesis we actually use for our permanence properties.
In some of the examples here, we prove
the tracial Rokhlin property with comparison.
We do not know whether these two properties are actually different;
see Problem~\ref{Pb_5506_Restr}.

For comparison, we give a reformulation of Definition~\ref{traR}
for finite groups
which more closely resembles Definition 1.2 in \cite{phill23}.

\begin{lem}\label{L_1X16_trpc_GFin}
Let $A$ be an infinite dimensional simple unital \ca,
let $G$ be a finite group,
and let $\af \colon G \to \Aut (A)$ be an action of $G$ on~$A$.
Then $\af$ has the \trpc{} \ifo{}
for every finite set $F \subseteq A$, every $\ep > 0$,
every $x \in A_{+}$ with $\| x \| = 1$,
and every $y \in (A^{\af})_{+} \setminus \{ 0 \}$
there exist a projection $p \in A^{\alpha}$
and mutually orthogonal projections $(p_{g})_{g \in G}$
such that the following hold.
\begin{enumerate}
\item\label{Item_1X16_Inv}
$p = \sum_{g \in G} p_{g}$.
\item\label{Item_1X16_Comm}
$\| p_g a - a p_g \| < \ep$ for all $a \in F$ and all $g \in G$.
\item\label{Item_1X16_Prm}
$\| \af_g (p_h) - p_{g h} \| < \ep$ for all $g, h \in G$.
\item\label{Item_1X16_sub_x}
$1 - p \precsim_A x$.
\item\label{Item_1X16_sub_yy}
$1 - p \precsim_{A^{\alpha}} y$.
\item\label{Item_1X16_sub_1mp}
$1 - p \precsim_{A^{\alpha}} p$.
\item\label{Item_1X16_Mnp}
$\| p x p \| > 1 - \ep$.
\setcounter{TmpEnumi}{\value{enumi}}
\end{enumerate}
\end{lem}

\begin{proof}
Set
$S = \bigl\{ \ch_{ \{ g \} } \colon g \in G \bigr\} \S C (G)$.
It is easily seen that Definition~\ref{traR} is equivalent
(with a change in the value of $\ep$) to the same statement
but in which we always use this choice of~$S$.

Assume the conditions of the lemma.
Let $F \subseteq A$ be finite,
let $\varepsilon > 0$, let $x \in A_{+}$ satisfy $\| x \| = 1$,
and let $y \in (A^{\alpha})_{+} \setminus \{ 0 \}$.
Let $p$ and $(p_{g})_{g \in G}$ be as in the condition of the lemma
for these choices.
Then Conditions (\ref{1_pxcompactsets}), (\ref{1_pycompactsets}),
(\ref{1_ppcompactsets}), and~(\ref{Item_902_pxp_TRP})
in Definition~\ref{traR}
follow immediately.
Define a unital \hm{} $\ph \colon C (G) \to p A p$ by
$\ph (f) = \sum_{g \in G} f (g) p_g$ for $f \in C (G)$.
The following calculations then show that
$\varphi$ is $(F, S, \varepsilon)$-approximately equivariant
central, and prove
(\ref{Item_893_FS_equi_cen_multi_approx}) in Definition~\ref{traR}.
First, for $h \in G$, recalling Notation~\ref{N_1X07_Lt},
and by~(\ref{Item_1X16_Prm}) at the second step, we have
\[
\max_{g \in G} \bigl\| (\ph \circ \Lt_g) ( \ch_{ \{ h \} })
       - (\af_g \circ \ph) ( \ch_{ \{ h \} }) \bigr\|
 = \max_{g \in G} \| p_{g h} - \af_g (p_h) \|
 < \ep.
\]
Second, for $g \in G$ and $a \in F$,
by~(\ref{Item_1X16_Comm}) at the second step, we have
\[
\| \ph (\ch_{ \{ g \} }) a - a \ph (\ch_{ \{ g \} }) \|
 = \| p_g a - a p_g \|
 < \ep.
\]

For the other direction, assume that $\af$ has the \trpc.
Set $n = \card (G)$.
Let $F \subseteq A$ be finite,
let $\varepsilon > 0$, let $x \in A_{+}$ satisfy $\| x \| = 1$,
and let $y \in (A^{\alpha})_{+} \setminus \{ 0 \}$.
\Wolog{} $\| a \| \leq 1$ for all $a \in A$.
Set
\[
\ep_0 = \min \biggl( \frac{1}{2 n}, \, \frac{\ep}{3} \biggr).
\]
Choose $\dt > 0$ so small that $\dt \leq \ep_0$ and
whenever $B$ is a \ca{}
and $b_1, b_2, \ldots, b_n \in B$ are selfadjoint
and satisfy $\| b_j^2 - b_j \| < 3 \dt$ for $j = 1, 2, \ldots, n$
and $\| b_j b_k \| < \dt$
for distinct $j, k \in \{ 1, 2, \ldots, n \}$,
then there are \mops{} $e_j \in B$ for $j = 1, 2, \ldots, n$
such that $\| e_j - b_j \| < \ep_0$ for $j = 1, 2, \ldots, n$.

Apply Definition~\ref{traR} with $\dt$ in place of~$\ep$,
and with $F$, $x$, and~$y$ as given,
getting $p \in A^{\af}$ and $\ph \colon C (G) \to p A p$ as there.
In particular:
\begin{enumerate}
\setcounter{enumi}{\value{TmpEnumi}}
\item\label{It_1X21_26_Sq}
$\| \ph (\ch_{ \{ g \} })^2 - \ph (\ch_{ \{ g \} }) \| < \dt$
for all $g \in G$.
\item\label{It_1X21_26_Z}
$\| \ph (\ch_{ \{ g \} }) \ph (\ch_{ \{ h \} }) \| < \dt$
for all $g, h \in G$ with $g \neq h$.
\item\label{It_1X21_26_Comm}
$\| \ph (\ch_{ \{ g \} }) a - a \ph (\ch_{ \{ g \} }) \| < \dt$
for all $g \in G$ and all $a \in F$.
\item\label{It_1X21_26_Tr}
$\| (\af_g \circ \ph) (\ch_{ \{ h \} }) - \ph (\ch_{ \{ g h \} }) \|
   < \dt$
for all $g, h \in G$.
\end{enumerate}
By (\ref{It_1X21_26_Sq}), (\ref{It_1X21_26_Z}),
and the choice of $\dt$, there are \mops{}
$p_g$ for $g \in G$
such that $\| p_g - \ph (\ch_{ \{ g \} }) \| < \ep_0$.
Using~(\ref{It_1X21_26_Comm}), for $g \in G$ and $a \in F$ we get
\[
\| p_g a - a p_g \|
 \leq \| \ph (\ch_{ \{ g \} }) a - a \ph (\ch_{ \{ g \} }) \|
        + 2 \| p_g - \ph (\ch_{ \{ g \} }) \|
 < \dt + 2 \ep_0 \leq \ep.
\]
This is~(\ref{Item_1X16_Comm}).

Using~(\ref{It_1X21_26_Tr}), for $g, h \in G$ we get
\[
\begin{split}
\| \af_g (p_h) - p_{g h} \|
& \leq \| (\af_g \circ \ph) (\ch_{ \{ h \} })
           - \ph (\ch_{ \{ g h \} }) \|
        + \| p_h - \ph (\ch_{ \{ h \} }) \|
        + \| p_{g h} - \ph (\ch_{ \{ g h \} }) \|
\\
& < \dt + 2 \ep_0
  \leq \ep.
\end{split}
\]
This is~(\ref{Item_1X16_Prm}).
Also, since $\ph (1) = 1$,
\[
\biggl\| p - \sum_{g \in G} p_g \biggr\|
 \leq \sum_{g \in G} \| \ph (\ch_{ \{ g \} }) - p_g \|
 < n \ep_0
 < 1,
\]
so, since $\sum_{g \in G} p_g$ is a \pj{},
we get $\sum_{g \in G} p_g = p$.
This is~(\ref{Item_1X16_Inv}).
Conditions (\ref{Item_1X16_sub_x}), (\ref{Item_1X16_sub_yy}),
(\ref{Item_1X16_sub_1mp}), and~(\ref{Item_1X16_Mnp})
are Conditions (\ref{1_pxcompactsets}), (\ref{1_pycompactsets}),
(\ref{1_ppcompactsets}), and~(\ref{Item_902_pxp_TRP})
in Definition~\ref{traR}.
\end{proof}

\begin{dfn}\label{D_1920_NTRP}
Let $A$ be an infinite dimensional simple separable unital \ca,
let $G$ be a second countable compact group,
and let $\af \colon G \to \Aut (A)$ be an action of $G$ on~$A$.
The action $\af$ has the
{\emph{naive tracial Rokhlin property}}
if for every finite set $F \subseteq A$,
every finite set $S \subseteq C (G)$, every $\ep > 0$,
and every $x \in A_{+}$ with $\| x \| = 1$,
there exist a projection $p \in A^{\alpha}$
and a unital completely positive
contractive map $\ph \colon C (G) \to p A p$
such that the following hold.
\begin{enumerate}
\item\label{It_1920_NTRP_FSE}
$\ph$ is an $(F, S, \ep)$-equivariant central multiplicative map.
\item\label{It_1920_NTRP_subx}
$1 - p \precsim_A x$.
\item\label{It_1920_NTRP_pxp}
$\| p x p \| > 1 - \ep$.
\end{enumerate}
\end{dfn}

We prove that condition
(\ref{1_pycompactsets}) of Definition \ref{traR} is automatic
when the group is finite, regardless of what $A$ is.
The list of conditions in the next proposition
is the same as in Lemma~\ref{L_1X16_trpc_GFin},
except that (\ref{Item_1X16_sub_1mp}) there has been omitted.

\begin{prp}\label{L_1617_IfGFin}
Let $A$ be an infinite dimensional simple separable unital \ca,
let $G$ be a finite group,
and let $\af \colon G \to \Aut (A)$ be an action of $G$ on~$A$
which has the tracial Rokhlin property.
Then for every finite set $F \subseteq A$, every $\ep > 0$,
every $x \in A_{+}$ with $\| x \| = 1$,
and every $y \in (A^{\alpha})_{+} \SM \{ 0 \}$,
there exist a projection $p \in A^{\alpha}$
and mutually orthogonal projections $(p_{g})_{g \in G}$
such that the following hold.
\begin{enumerate}
\item\label{Item_786}
$p = \sum_{g \in G} p_{g}$.
\item\label{Item_161p_comm}
$\| p_g a - a p_g \| < \ep$ for all $a \in F$ and all $g \in G$.
\item\label{Item_1619_CM}
$\| \af_g (p_h) - p_{g h} \| < \ep$ for all $g, h \in G$.
\item\label{Item_1617sub_x}
$1 - p \precsim_A x$.
\item\label{Item_1617sub_y}
$1 - p \precsim_{A^{\alpha}} y$.
\item\label{Item_1617_GFin_pxp}
$\| p x p \| > 1 - \ep$.
\end{enumerate}
\end{prp}

\begin{proof}
Let $F \subseteq A$ be finite and let $\varepsilon > 0$.
Let $x \in A_{+}$ with $\| x \| = 1$ and
$y \in (A^{\alpha})_{+} \SM \{ 0 \}$ be given.
Corollary 1.6 of~\cite{phill23} implies that $A^{\alpha}$ is simple.
Since $A^{\alpha}$ is unital and not \fd,
it follows that $A^{\alpha}$ is not of Type~I.
Therefore $\overline{y A^{\alpha} y}$ is simple and not of Type~I.
By Lemma 2.1 of \cite{philar} there is a positive element
$z \in \overline{y A^{\alpha} y}$
such that $0$ is a limit point of $\spec (z)$.
Define \cfn{s} $h, h_0 \colon [0, 1] \to [0, 1]$ by
\[
h (\ld) = \begin{cases}
   0 & \hspace*{1em} 0 \leq \ld \leq 1 - \frac{\ep}{2}
        \\
   \frac{2}{\ep} (\ld - 1) + 1
           & \hspace*{1em} 1 - \frac{\ep}{2} \leq \ld \leq 1
\end{cases}
\]
and
\[
h_0 (\ld) = \begin{cases}
   \left( 1 - \frac{\ep}{2} \right)^{-1} \ld
           & \hspace*{1em} 0 \leq \ld \leq 1 - \frac{\ep}{2}
        \\
   1 & \hspace*{1em} 1 - \frac{\ep}{2} \leq \ld \leq 1.
\end{cases}
\]
Then $\| x - h_0 (x) \| \leq \frac{\ep}{2}$.
Also, $h (x) \neq 0$ since $\| x \| = 1$.
Use Lemma 2.6 of \cite{philar} to choose a nonzero positive
element $x_0 \in \ov{h (x) A h (x)}$ such that $x_0 \precsim_A z$.
We may require $\| x_0 \| = 1$.

Now apply Lemma 1.17 of \cite{phill23} to $\alpha$
with $x_0$ in place of~$x$, with $\frac{\ep}{2}$ in place of~$\ep$,
and with $F$ as given.
We obtain mutually orthogonal projections $p_{g} \in A$ for $g \in G$
such that, with $p = \sum_{g \in G} p_{g}$
(so that (\ref{Item_786}) holds),
$p$ is $\af$-invariant;
with $x_0$ in place of~$x$ and $\frac{\ep}{2}$ in place of~$\ep$,
Conditions (\ref{Item_161p_comm}), (\ref{Item_1619_CM})
and~(\ref{Item_1617_GFin_pxp}) are satisfied;
and $1 - p$ is \mvnt{} to a \pj{} in $\ov{x_0 A x_0}$.
In particular, we have (\ref{Item_786}), (\ref{Item_161p_comm}),
and~(\ref{Item_1619_CM}) as stated.
Also,
$1 - p \precsim_A x_0 \precsim_A x$, which is (\ref{Item_1617sub_x}).
Moreover,
\[
1 - p \precsim_A x_0 \precsim_A z,
\qquad
1 - p, z \in A^{\alpha},
\andeqn
0 \in \ov{ \spec (z) \SM \{ 0 \} }.
\]
Since the tracial Rokhlin property implies the weak tracial Rokhlin
property, it follows from Lemma 3.7 of~\cite{radifinite}
that $1 - p \precsim_{A^{\alpha}} z$.
Since $z \in \ov{y A^{\alpha} y}$,
we get $1 - p \precsim_{A^{\alpha}} y$,
which is~(\ref{Item_1617sub_y}).

It remains to prove~(\ref{Item_1617_GFin_pxp}).
Since $h_0 (x) h (x) = h (x)$, we have
\[
x_0 = h_0 (x)^{1 / 2} x_0 h_0 (x)^{1 / 2} \leq h_0 (x).
\]
So, also using $\| p x_0 p \| > 1 - \frac{\ep}{2}$,
\[
\| p x p \|
 \geq \| p h_0 (x) p \| - \| h_0 (x) - x \|
 \geq \| p x_0 p \| - \frac{\ep}{2}
 > 1 - \frac{\ep}{2} - \frac{\ep}{2}
 = 1 - \ep,
\]
as desired.
\end{proof}

We now give some conditions on actions of finite groups under
which Condition (\ref{1_ppcompactsets}) of
Definition \ref{traR} is automatic.

\begin{prp}\label{P_1X17_StComp}
Let $A$ be a stably finite infinite dimensional
simple separable unital \ca.
Let $\af \colon G \to \Aut (A)$
be an action of a finite group $G$ on $A$
which has the tracial Rokhlin property.
If $\rc (A) < 1$ then $\af$ has the \trpc.
\end{prp}

In particular, if $A$ has strict comparison then $\af$ has the \trpc.

\begin{proof}[Proof of Proposition~\ref{P_1X17_StComp}]
The conclusion is similar to but stronger than
that of Proposition~\ref{L_1617_IfGFin}.
We describe the necessary changes and additions to the proof
of that proposition.

We verify the conditions of Lemma~\ref{L_1X16_trpc_GFin}.

By Proposition 2.5 in \cite{MhkPh1},
we may assume that $\af$ does not have the Rokhlin property.
Let $F \subseteq A$ be finite and let $\varepsilon > 0$.
Let $x \in A_{+}$ with $\| x \| = 1$ and
$y \in (A^{\alpha})_{+} \SM \{ 0 \}$ be given.
As in the proof of Proposition~\ref{L_1617_IfGFin},
$A^{\alpha}$ is simple and not of Type~I,
and there is a positive element $z \in \overline{y A^{\alpha} y}$
such that $0$ is a limit point of $\spec (z)$.

Choose $n \in \N$ such that
\begin{equation}\label{Eq_1X29_n_rc}
n > \frac{2}{1 - \rc (A)}.
\end{equation}
The algebra $A^{\af}$ has Property~(SP) by Lemma~1.13 of~\cite{phill23}.
Lemma 1.10 in \cite{phill23},  provides $n + 1$ nonzero \mops{} in~$A$.
Let $e$ be one of them.
Again as in the proof of Proposition~\ref{L_1617_IfGFin},
there is $d \in (e A^{\af} e)_{+} \SM \{ 0 \}$
such that $0$ is a limit point of $\spec (d)$.
With $h$ as in that proof,
apply Lemma 2.6 in \cite{philar} to find $x_0 \in \ov{h (x) A h (x)}$
such that
\[
\| x_0 \| = 1,
\qquad
x_0 \precsim_A d,
\andeqn
x_0 \precsim_A y_0.
\]

Now, apply Lemma 1.17 of~\cite{phill23} to $\alpha$ with the same
choices as in the proof of Proposition~\ref{L_1617_IfGFin},
getting, as in the second half of the proof,
mutually orthogonal projections $p_{g} \in A$ for $g \in G$
such that, with $p = \sum_{g \in G} p_{g}$,
the conclusion of Proposition~\ref{L_1617_IfGFin} holds,
and $1 - p \precsim_{A} x_0$.

It remains only to show that $1 - p \precsim_{A^{\af}} p$.
We know that $1 - p \precsim_A x_0 \precsim_A d$.
By Lemma 3.7 of~\cite{radifinite},
we have $1 - p \precsim_{A^{\alpha}} d$.
So $1 - p \precsim_{A^{\alpha}} e$.
Let $\ta \in \QT (A^{\af})$.
Then $\ta (1 - p) \leq \ta (e)$, so $\ta (p) \geq \ta (1 - e)$.
The construction of $e$ ensures that
$\ta (e) \leq  \frac{1}{n + 1} < \frac{1}{n}$.
Therefore
\begin{equation}\label{Eq_1X22_Ineq}
\ta (1 - p) + 1 - \frac{2}{n} <  1 - \frac{1}{n} < \ta (p).
\end{equation}
Using Theorem 4.1 of \cite{radifinite} at the first step
and~(\ref{Eq_1X29_n_rc}) at the second,
we have $\rc (A^{\af}) \leq \rc (A) < 1 - \frac{2}{n}$.
Since~(\ref{Eq_1X22_Ineq}) holds for all $\ta \in \QT (A^{\af})$,
we get $1 - p \precsim_{A^{\af}} p$, as desired.
\end{proof}

\begin{cor}\label{C_1X29_Fin_rc}
Let $A$ be a stably finite infinite dimensional
simple separable unital \ca{} such that $\rc (A)$ is finite.
Let $\af \colon G \to \Aut (A)$
be an action of a finite group $G$ on $A$
which has the tracial Rokhlin property.
Then there exists $n \in \N$ such that the action
$g \mapsto \id_{M_n} \otimes \af_g$ on $M_n \otimes A$ has the \trpc.
\end{cor}

\begin{proof}
Choose $n \in \N$ such that $n > \rc (A)$.
Then $\rc (M_n \otimes A) = \frac{1}{n} \rc (A) < 1$.
Also, $g \mapsto \id_{M_n} \otimes \af_g$
has the tracial Rokhlin property by Lemma~3.9 of~\cite{phill23}.
Apply Proposition~\ref{L_1617_IfGFin}.
\end{proof}

To make the argument work in general, one needs to apply to $A^{\af}$
a positive answer to the following question.

\begin{qst}\label{Pb_1X20_ExSm}
Let $B$ be a stably finite simple separable \uca.
Does there exist $z \in B_{+} \setminus \{ 0 \}$ such that
whenever a \pj{} $q \in B$ satisfies $q \precsim_B z$,
then $q \precsim_B 1 - q$?
\end{qst}

This question seems hard, but the answer may well be negative.

We now prove that if $G$ is finite and $A$ is purely infinite simple,
then condition (\ref{1_ppcompactsets}) of
Definition \ref{traR} is automatic.

\begin{prp}\label{P_1X17_PI}
Let $A$ be an infinite dimensional simple separable unital \ca.
Let $\af \colon G \to \Aut (A)$
be an action of a finite group $G$ on $A$
which has the tracial Rokhlin property.
If $A$ is purely infinite then $\af$ has the \trpc.
\end{prp}

\begin{proof}
We verify the condition of Lemma~\ref{L_1X16_trpc_GFin}.
So let $F \subseteq A$ be finite,
let $\varepsilon > 0$, let $x \in A_{+}$ satisfy $\| x \| = 1$,
and let $y \in (A^{\alpha})_{+} \setminus \{ 0 \}$.
\Wolog{} $\ep < 1$.
Apply Lemma 1.17 of~\cite{phill23} with $F$, $\ep$, and $x$ as given,
getting \mops{} $p_g \in A$ for $g \in G$ such that,
in Lemma~\ref{L_1X16_trpc_GFin} and with $p = \sum_{g \in G} p_g$,
Conditions (\ref{Item_1X16_Inv}), (\ref{Item_1X16_Comm}),
(\ref{Item_1X16_Prm}), (\ref{Item_1X16_sub_x}),
and~(\ref{Item_1X16_Mnp}) are satisfied.
The algebra $C^* (G, A, \af)$
is simple by Corollary~1.6 of~\cite{phill23}.
So $A^{\af}$ is simple by Theorem 3.5 of \cite{MhkPh1} .
The action $\af$ is pointwise outer by Lemma~1.5 of~\cite{phill23},
so Theorem~3 of~\cite{Je95}
implies that $C^* (G, A, \af)$ is purely infinite.
Corollary 3.9 of \cite{MhkPh1}
now implies that $A^{\af}$ is stably isomorphic to $C^* (G, A, \af)$,
so $A^{\af}$ is purely infinite.
Since $y \neq 0$,
the relation $1 - p \precsim_{A^{\af}} y$ is automatic;
this is Condition (\ref{Item_1X16_sub_yy})
of Lemma~\ref{L_1X16_trpc_GFin}.
Since $p \neq 0$ (from $\| p x p \| > 1 - \ep$),
the relation $1 - p \precsim_{A^{\af}} p$ is automatic;
this is Condition (\ref{Item_1X16_sub_1mp})
of Lemma~\ref{L_1X16_trpc_GFin}.
\end{proof}

\section{The modified tracial Rokhlin property}\label{S_2795_mod_TRP}

The extra conditions in Definition~\ref{traR} seem somewhat
unsatisfactory, partly because there are two of them.
We seem to need Condition~(\ref{1_pycompactsets})
($1 - p \precsim_{A^{\alpha}} y$)
in order to prove preservation of tracial rank when it is zero or one.
(See the proof of Theorem 4.7 in \cite{MhkPh1}.)
We seem to need Condition~(\ref{1_ppcompactsets})
($1 - p \precsim_{A^{\alpha}} p$)
in order to prove that the crossed product is simple.
(See the proof of Proposition 3.6 in \cite{MhkPh1}.)
This has led us to consider other variants.
In this section, we discuss the most promising of these,
which we call the modified tracial Rokhlin property
(again, a name intended for use only in this paper).
There are actually two versions.
We will explain what we can prove with them.
In a very special case (Proposition~\ref{D_1619_TRP_and_s} below),
they are automatic for
actions of finite groups with the tracial Rokhlin property,
but the proof is somewhat long and is omitted.
The example we construct in Section~\ref{Sec_3749_Exam_TRPZ2} has
the strong modified tracial Rokhlin property,
and the example in Section~\ref{Sec_1908_Exam_TRPS1}
has the modified tracial Rokhlin property
but probably not the strong modified tracial Rokhlin property.

The difference in the definitions we give below
is in Definition \ref{moditra}(\ref{Item_1X07_30})
and Definition \ref{D_1824_AddToTRP_Mod}(\ref{Im_1824_ATRP_s_Comm_Mod}).
For the strong modified tracial Rokhlin property,
the partial isometry $s$ is required to approximately commute
with the elements of a given finite subset of~$A$,
but for the modified tracial Rokhlin property,
$s$ is only required to approximately commute
with the elements of a given finite subset of~$A^{\af}$.
This definition thus requires two finite sets instead of just one.

\begin{dfn}\label{moditra}
Let $A$ be an infinite dimensional simple separable unital \ca,
let $G$ be a second countable compact group,
and let $\af \colon G \to \Aut (A)$ be an action of $G$ on~$A$.
The action $\af$ has the
{\emph{modified tracial Rokhlin property}}
if for every finite set $F_1 \subseteq A$,
every finite set $F_2 \subseteq A^{\af}$,
every finite set $S \subseteq C (G)$, every $\ep > 0$,
and every $x \in A_{+}$ with $\| x \| = 1$,
there exist a projection $p \in A^{\alpha}$,
a partial isometry $s \in A^{\alpha}$,
and a unital completely positive
contractive map $\ph \colon C (G) \to p A p$,
such that the following hold.
\begin{enumerate}
\item\label{Item_1X07_27}
$\varphi$ is an $(F_1, S, \varepsilon)$-approximately equivariant
central multiplicative map.
\item\label{Item_1X07_28}
$1 - p \precsim_{A} x$.
\item\label{Item_1X07_29}
$s^{*} s = 1 - p$ and $s s^{*} \leq p$.
\item\label{Item_1X07_30}
$\| s a - a s \| < \varepsilon$ for all $a \in F_2$.
\item\label{Item_1X07_31}
$\| p x p \| > 1 - \varepsilon$.
\end{enumerate}
\end{dfn}

\begin{dfn}\label{D_1824_AddToTRP_Mod}
Let $A$ be an infinite dimensional simple separable unital \ca,
let $G$ be a second countable compact group,
and let $\af \colon G \to \Aut (A)$ be an action of $G$ on~$A$.
The action $\af$ has the
{\emph{strong modified tracial Rokhlin property}}
if for every finite set $F \subseteq A$,
every finite set $S \subseteq C (G)$, every $\ep > 0$,
and every $x \in A_{+}$ with $\| x \| = 1$,
there exist a partial isometry $s \in A^{\alpha}$,
a projection $p \in A^{\alpha}$,
and a unital completely positive
contractive map $\ph \colon C (G) \to p A p$,
such that the following hold.
\begin{enumerate}
\item\label{Item_1824_ATRP_Mod_CM}
$\ph$ is an $(F, S, \ep)$-approximately
equivariant central multiplicative map.
\item\label{Item_1824_ATRP_sub_x_Mod}
$1 - p \precsim_A x$.
\item\label{Item_1824_ATRP_SubP_Mod}
$s^* s = 1 - p$ and $s s^* \leq p$.
\item\label{Im_1824_ATRP_s_Comm_Mod}
$\| s a - a s \| < \ep$ for all $a \in F$.
\item\label{Item_1824_ATRP_pxp_Mod}
$\| p x p \| > 1 - \ep$.
\end{enumerate}
\end{dfn}

The modified tracial Rokhlin property
implies simplicity of the fixed point algebra and crossed product.
For simplicity of the fixed point algebra,
we don't need to know that the element $s$ in
Definition \ref{moditra} approximately commutes with anything.

The following are analogs of results in Section~3 of~\cite{MhkPh1}.
Their proofs are similar, and we omit them.

\begin{thm}\label{thm:simple2}
Let $A$ be a simple separable infinite dimensional unital \ca,
let $G$ be a second countable compact group, and let
$\alpha \colon G \to \Aut (A)$ be an action
which has the modified tracial Rokhlin property.
Then $A^{\alpha}$ is simple.
\end{thm}

\begin{prp}\label{satpropnoncpts}
Let $A$ be an infinite dimensional simple separable unital \ca.
Let $\alpha \colon G \to \Aut (A)$
be an action of a compact group $G$ on $A$ which
has the modified tracial Rokhlin property.
Then $\alpha$ is saturated.
\end{prp}

\begin{thm}\label{simplecross}
Let $A$ be an infinite dimensional simple separable unital \ca,
and let $\alpha \colon G \to \Aut (A)$
be an action of a second countable compact group $G$ on~$A$
which has the modified tracial Rokhlin property.
Then $C^{*} (G, A, \alpha)$ is simple.
\end{thm}

Crossed products by actions with the modified tracial Rokhlin property
probably also preserve Property~(SP) and pure infiniteness.
We were not able to prove that even
crossed products by actions with the
strong modified tracial Rokhlin property preserve tracial rank zero.

\begin{prp}\label{D_1619_TRP_and_s}
Let $A$ be a UHF~algebra,
let $G$ be a finite group,
and let $\af \colon G \to \Aut (A)$ be an action of $G$ on~$A$
which has the tracial Rokhlin property.
Then $\af$ has the strong modified tracial Rokhlin property.
\end{prp}

The proof is somewhat long and technical.
Since we don't (yet) have a serious use for
the strong modified tracial Rokhlin property
or the modified tracial Rokhlin property, we omit the proof.

\section{Actions of infinite products of finite
 groups on infinite tensor products}\label{Sec_3749_Exam_TRPZ2}

In this section we show that, up to a technicality,
the infinite tensor product of actions of finite groups on
stably finite simple unital separable infinite dimensional \ca{s} with
the tracial Rokhlin property
has the tracial Rokhlin property with compassion when the
underlying algebra has strict comparison.
As a special case, we give an action
of a totally disconnected infinite compact group on
a UHF~algebra which has the tracial Rokhlin property with comparison
and the strong modified tracial Rokhlin property,
but does not have the Rokhlin property,
or even finite Rokhlin dimension with commuting towers.

We start with some general definitions and lemmas that will be used
in Section~\ref{Sec_2114_OI} as well.
The following lemma is well known, and is given in this form for
convenient reference.
One might like to write
$\af = \af_1 \otimes \af_2 \otimes \cdots \otimes \af_n$,
but this notation is already taken for the diagonal action
on a tensor product obtained from an action of the same group on
all the tensor factors.

The lemma holds for any functorial tensor product.

\begin{lem}\label{L_5420_FinTP}
Let $n \in \N$.
Let $A_1, A_2, \ldots, A_n$ be \ca{s},
let $G_1, G_2, \ldots, G_n$ be topological groups,
and for $k = 1, 2, \ldots, n$ let
$\af^{(k)} \colon G_k \to \Aut (A_k)$ be a \ct{} action.
Set
\[
A = A_1 \otimes_{\min} A_2 \otimes_{\min} \cdots \otimes_{\min} A_n
\andeqn
G = G_1 \times G_2 \times \ldots \times G_n.
\]
Then there is a unique \ct{} action $\af \colon G \to \Aut (A)$
such that, whenever $g_k \in G_k$ and $a_k \in A_k$
for $k = 1, 2, \ldots, n$, we have
\begin{equation}\label{Eq_5421_Finite}
\af_{(g_1, g_2, \ldots, g_n)} (a_1 \otimes a_2 \otimes \cdots \otimes a_n)
 = \af^{(1)}_{g_1} (a_1) \otimes \af^{(2)}_{g_2} (a_2)
      \otimes \cdots \otimes \af^{(n)}_{g_n} (a_n).
\end{equation}
\end{lem}

\begin{proof}
The proof is immediate.
\end{proof}

\begin{lem}\label{L_5420_Inf_TP}
Let $A_1, A_2, \ldots$ be \uca{s},
let $G_1, G_2, \ldots$ be topological groups,
and for $k = 1, 2, \ldots$ let
$\af^{(k)} \colon G_k \to \Aut (A_k)$ be a \ct{} action.
Using minimal tensor products, set
\[
A = \bigotimes_{n = 1}^{\I} A_n
  = \dirlim_n A_1 \otimes_{\min} A_2 \otimes_{\min} \cdots \otimes_{\min} A_n
\andeqn
G = \prod_{n = 1}^{\I} G_n.
\]
Then there is a unique \ct{} action $\af \colon G \to \Aut (A)$
such that, whenever $n \in \N$, $g_k \in G_k$ for $k \in \N$,
and $a_k \in A_k$
for $k = 1, 2, \ldots, n$, we have
\[
\af_{(g_1, g_2, \ldots)}
     (a_1 \otimes a_2 \otimes \cdots \otimes a_n \otimes 1)
 = \af^{(1)}_{g_1} (a_1) \otimes \af^{(2)}_{g_2} (a_2)
      \otimes \cdots \otimes \af^{(n)}_{g_n} (a_n) \otimes 1.
\]
\end{lem}

\begin{proof}
Lemma~\ref{L_5420_FinTP} gives actions
\[
\bt^{(n)} \colon G_1 \times G_2 \times \ldots \times G_n
   \to \Aut \bigl( A_1 \otimes_{\min} A_2 \otimes_{\min}
       \cdots \otimes_{\min} A_n \bigr)
\]
as in~(\ref{Eq_5421_Finite}), which extend to actions of~$G$
via the projection map $G \to \prod_{k = 1}^{n} G_k$.
Take the direct limit of these actions.
\end{proof}

\begin{dfn}\label{D_5420_Inf_tp}
The action $\af \colon G \to \Aut (A)$ of Lemma~\ref{L_5420_FinTP}
is called the (minimal) tensor product action of~$G$ on~$A$,
and the action $\af \colon G \to \Aut (A)$ of Lemma~\ref{L_5420_Inf_TP}
is called the infinite (minimal) tensor product action of~$G$ on~$A$.
\end{dfn}

The next lemma is well known, but we do not know a reference.

\begin{lem}\label{L_5420_Fin_fix}
Assume the hypotheses and notation of Lemma~\ref{L_5420_FinTP},
and further assume that the groups $G_1, G_2, \ldots, G_n$
are compact.
Then the fixed point algebra $A^{\af}$ is given by
\[
A^{\af} = A_1^{\af^{(1)}} \otimes_{\min} A_2^{\af^{(2)}}
    \otimes_{\min} \cdots \otimes_{\min} A_n^{\af^{(n)}}.
\]
\end{lem}

\begin{proof}
Since the minimal tensor product preserves injectivity of \hm{s},
$\bigotimes_{k = 1}^{n} A_k^{\af^{(k)}}$ (minimal tensor product)
is in fact a subalgebra of~$A$.
It is obvious that $\bigotimes_{k = 1}^{n} A_k^{\af^{(k)}} \subseteq A^{\af}$.

For the reverse inclusion, for $k = 1, 2, \ldots, n$
let $E_k \colon A_k \to A_k$ be the map given by averaging over~$G_k$:
with $\mu_k$ being normalized Haar measure on~$G$,
$E_k (a) = \int_{G_k} \af^{(k)}_g (a) \, d \mu_k (a)$.
Similarly let $E \colon A \to A$ be the map given by averaging over~$G$.
It suffices to prove that
$E (a) \in \bigotimes_{k = 1}^{n} A_k^{\af^{(k)}}$ for $a \in A$,
and it suffices to do this for $a$ in a set $T \S A$ which
spans a dense subspace.
We take $T$ to be the set of all elementary tensors,
and observe that if $a_k \in A_k$ for $k = 1, 2, \ldots, n$, then
\[
E (a_1 \otimes a_2 \otimes \cdots \otimes a_n)
 = E_1 (a_1) \otimes E_2 (a_2) \otimes \cdots \otimes E_n (a_n)
 \in \bigotimes_{k = 1}^{n} A_k^{\af^{(k)}}.
\]
This completes the proof.
\end{proof}

\begin{lem}\label{L_5420_Inf_fix}
Assume the hypotheses and notation of Lemma~\ref{L_5420_Inf_TP},
and further assume that the groups $G_1, G_2, \ldots$
are compact.
Then the fixed point algebra $A^{\af}$ is given by the infinite
minimal tensor product:
\[
A^{\af} = \bigotimes_{n = 1}^{\I} A_n^{\af^{(n)}}.
\]
\end{lem}

\begin{proof}
Since the tensor products and direct limits preserve injectivity of \hm{s},
$\bigotimes_{n = 1}^{\I} A_n^{\af^{(n)}}$ is in fact a subalgebra of~$A$.
It is obvious that
$\bigotimes_{n = 1}^{\I} A_n^{\af^{(n)}} \subseteq A^{\af}$.

For the reverse inclusion,
let $E \colon A \to A$ be the map given by averaging over~$G$.
It suffices to prove that
$E (a) \in \bigotimes_{n = 1}^{\I} A_n^{\af^{(n)}}$ for $a \in A$,
and it suffices to do this for $a$ in a dense subset $S \S A$.
Take
\[
S = \bigl\{ a_1 \otimes a_2 \otimes \cdots \otimes a_n \otimes 1 \colon
    {\mbox{$n \in \N$ and $a_k \in A_k$ for $k = 1, 2, \ldots, n$}} \bigr\},
\]
and apply Lemma~\ref{L_5420_Fin_fix}.
\end{proof}

The following lemma seems worth stating separately,
and will be used again in Section~\ref{Sec_2114_OI}.

\begin{lem}\label{L_5501_Comp_1}
Let $A$ and $B$ be simple \uca{s}.
Let $x \in (A \otimes_{\min} B)_{+} \setminus \{ 0 \}$.
Then there is $c \in A_{+}$ with $\| c \| = 1$
such that $c \otimes 1_B \precsim_{A \otimes_{\min} B} x$.
\end{lem}

\begin{proof}
By Kirchberg's slice lemma
(Lemma 4.1.9 of \cite{RorSt2002}), there exist $a \in A_{+} \setminus \{0\}$,
$b \in B_{+} \setminus \{0\}$, and $z \in A \otimes_{\min} B$
such that $z z^{*} = a \otimes b$ and
$z^{*} z \in \overline{x (A \otimes_{\min} B) x}$.
In particular, $a \otimes b \precsim_{A \otimes_{\min} B} x$.
Since $B$ is simple and unital, there are $s_1, s_2, \ldots, s_n \in B$
such that $\sum_{j = 1}^{n} s_{j} b s_{j}^{*} = 1$.
Use Lemma 2.4 of \cite{philar} to choose orthogonal
\[
c_1, c_2, \ldots, c_m
 \in \bigl( {\ov{a A a}} \bigr)_{+} \setminus \{ 0 \}
\]
such that
\[
c_1 \sim_{A} c_2 \sim_{A} \cdots \sim_{A} c_m.
\]
Set $c = c_1$.
Thus, the direct sum $c^{\oplus n}$ of $n$~copies of~$c$
satisfies $c^{\oplus n} \precsim_{A} a$.
\Wolog{} $\| c \| = 1$.
Then
\begin{equation}\label{d1x}
\begin{split}
c \otimes 1_B
& = \sum_{j = 1}^{n} (1_A \otimes s_j) (c \otimes b) ( 1_A \otimes s_j)^{*}
\\
& \precsim_{A \otimes_{\min} B} (c \otimes b)^{\oplus n}
  \sim_{A \otimes_{\min} B} c^{\oplus n} \otimes b
  \precsim_{A \otimes_{\min} B} a \otimes b
  \precsim_{A \otimes_{\min} B} x,
\end{split}
\end{equation}
as desired.
\end{proof}

\begin{prp}\label{tensortrpacts}
Let $n \in \N$.
Let $A_1, A_2, \ldots, A_n$ be simple infinite dimensional \uca{s},
let $G_1, G_2, \ldots, G_n$ be finite groups,
and for $k = 1, 2, \ldots, n$ let
$\af^{(k)} \colon G_k \to \Aut (A_k)$ be an action
with the tracial Rokhlin property.
Set
\[
A = A_1 \otimes_{\min} A_2 \otimes_{\min} \cdots \otimes_{\min} A_n
\andeqn
G = G_1 \times G_2 \times \ldots \times G_n.
\]
Let $\alpha \colon G \to \mathrm{Aut} (A)$
be the minimal tensor product action (Definition~\ref{D_5420_Inf_tp}).
If $A$ is stably finite, then $\af$ has the tracial Rokhlin property.
\end{prp}

\begin{proof}
By induction, it is enough to consider the case $n = 2$.
For simplicity, call the actions instead
$\af \colon G \to \Aut (A)$ and $\bt \colon H \to \Aut (B)$,
and let $\gm \colon G \times H \to \Aut (A \otimes_{\min} B)$
be the tensor product action.

By Lemma 1.16 of~\cite{phill23}, it suffices to prove that
for every finite set $E \subseteq A \otimes_{\min} B$,
every $\varepsilon > 0$, and
every $x \in (A \otimes_{\min} B)_{+}$  with $\|x\| = 1$,
there is a family $(t_{g, h})_{(g, h) \in G \times H}$
of orthogonal projections in $A \otimes_{\min} B$ such that, with
$t = \sum_{(g, h) \in G \times H} t_{g, h}$, the following hold:
\begin{enumerate}
\item\label{commut}
$\|t_{g, h} c - c t_{g, h}\| < \varepsilon$
for all $c \in E$, all $g \in G$, and all $h \in H$.
\item\label{shift}
$\| \alpha_{(g_{1}, h_{1})}(t_{g_{2}, h_{2}}) -
   t_{g_{1} g_{2}, h_{1} h_{2}}\| < \varepsilon$
for all $g_{1}, g_{2} \in G$ and $h_{1}, h_{2}\in H$.
\item\label{eror}
$1 - t \precsim_{A \otimes_{\min} B} x$.
\setcounter{TmpEnumi}{\value{enumi}}
\end{enumerate}
Using approximation, scaling, and linear combinations,
we may assume that there exist $a_{1}, a_2, \ldots, a_{m} \in A$ and
$b_{1}, b_2, \ldots, b_{m} \in B$ such that
$E = \{ a_{j} \otimes b_{j} \colon 1 \leq j \leq m\}$ and
$\| a_{j} \|, \| b_{j} \| \leq 1$ for $j = 1, 2, \ldots, m$.

By Lemma~2.4 of~\cite{philar}, there are
$x_{1}, x_{2} \in (A \otimes_{\min} B)_{+} \setminus \{0\}$
such that $x_{1} x_{2} = x_{2} x_{1} = 0$ and
$x_{1} + x_{2} \precsim_{A \otimes_{\min} B} x$.
By Lemma~\ref{L_5501_Comp_1}, there are $c \in A_{+} \setminus \{ 0 \}$
and $d \in B_{+} \setminus \{0\}$ such that
\[
c \otimes 1_B \precsim_{A \otimes_{\min} B} x_{1},
\qquad
1_A \otimes d \precsim_{A \otimes_{\min} B} x_{2},
\andeqn
\| c \| = \| d \| = 1.
\]
Apply Definition 1.2 of~\cite{phill23}
to the action $\alpha$ with $\{ a_{1}, \ldots, a_{m}\}$ in place of $F$,
with $\varepsilon /2$ in place of~$\ep$, and with $c$ in place of~$x$.
We obtain a family $(p_{g})_{g \in G}$ of nonzero orthogonal
projections in $A$
such that, with $p = \sum_{g \in G} p_{g}$, the following hold:
\begin{enumerate}
\setcounter{enumi}{\value{TmpEnumi}}
\item\label{comiouttens_1}
$\|p_{g} a_{j} - a_{j} p_{g}\| < \varepsilon /2$ for
$j = 1, 2, \ldots, m$ and all $g \in G$.
\item\label{shifttens_1}
$\| \alpha_{g_1} (p_{g_2}) - p_{g_1 g_2}\| < \varepsilon /2$
for all $g_1, g_2 \in G$.
\item\label{erortens_1}
$1 - p \precsim_{A} c$.
\setcounter{TmpEnumi}{\value{enumi}}
\end{enumerate}
Similarly, applying Definition 1.2 of~\cite{phill23}
to the action $\beta$ with $\{ b_{1}, \ldots, b_{m}\}$
in place of $F$,
with $\varepsilon /2$ in place of~$\ep$, and with $d$ in place of~$x$,
we obtain a family $(q_{g})_{g \in H}$ of nonzero orthogonal
projections in $B$
such that, with $q = \sum_{g \in H} q_{g}$, the following hold:
\begin{enumerate}
\setcounter{enumi}{\value{TmpEnumi}}
\item\label{comiouttens_2}
$\|q_{h} b_{j} - b_{j} q_{h}\| < \varepsilon / 2$ for
$j = 1, 2, \ldots, m$ and all $h \in H$.
\item\label{shifttens_2}
$\| \alpha_{h_1}(q_{h_2}) - q_{h_1 h_2}\| < \varepsilon /2$
for all $h_1, h_2 \in H$.
\item\label{erortens_2}
$1 - q \precsim_{B} d$.
\setcounter{TmpEnumi}{\value{enumi}}
\end{enumerate}

For $g \in G$ and $h \in H$, set $t_{g, h} = p_{g} \otimes q_{h}$.
Then set $t =\sum_{(g, h) \in G \times H} t_{g, h}$.
Easy calculations show that conditions (\ref{commut}) and~(\ref{shift}) hold.
For~(\ref{eror}), we have:
\[
\begin{split}
1 - t
& = 1 \otimes 1 - p \otimes q
  = 1 \otimes 1 - p \otimes 1 + (p \otimes 1) ( 1 \otimes 1 - 1 \otimes q)
\\
& \precsim_{A \otimes_{\min} B} c \otimes 1 \oplus d \otimes 1
  \precsim_{A \otimes_{\min} B} x_{1} \oplus x_{2}
  \precsim_{A \otimes_{\min} B} x.
\end{split}
\]
This completes the proof.
\end{proof}

\begin{thm}\label{TensProdActFinStr}
Let $A_1, A_2, \ldots$ be simple \uca{s},
let $G_1, G_2, \ldots$ be finite groups,
and for $k = 1, 2, \ldots$ let
$\af^{(k)} \colon G_k \to \Aut (A_k)$ be an action
with the tracial Rokhlin property.
Assume that the infinite minimal tensor product
$A = \bigotimes_{n = 1}^{\I} A_n$ is stably finite,
and that for every $n \in \N$ the
algebra $\bigotimes_{k = 1}^{n} A_k$ has strict comparison.
Then the infinite minimal tensor product action
of $G = \prod_{n = 1}^{\I} G_n$ on $A$
(Definition~\ref{D_5420_Inf_tp})
has the restricted tracial Rokhlin property with comparison.
\end{thm}

The strict comparison part of the hypotheses is automatic
if $A_1$ is ${\mathcal{Z}}$-stable.

\begin{proof}[Proof of Theorem~\ref{TensProdActFinStr}]
We verify the conditions of Definition~\ref{traR}.
So let $F \subseteq A$ and $S \subseteq C (G)$
be finite sets, let $\varepsilon > 0$,
let $x \in A_{+} \setminus \{0 \}$,
and let $y \in A^{\alpha}_{+} \setminus \{0 \}$.
Without loss of generality we can assume
$\| a \| \leq 1$ for all $a \in F$, $\| f \| \leq 1$ for all $f \in S$,
$\| x \| = \| y \| = 1$, and $\varepsilon < \frac{1}{2}$.
Set $\ep_0 = \ep / 5$.
Using for the choice of $y_0$ the identification of $A^{\af}$ in
Lemma~\ref{L_5420_Inf_fix}, and for the choice of $S_0$
the projection map $G \to \prod_{m = 1}^{N} G_m$ to identify
$C \left( \prod_{m = 1}^{N} G_m \right)$ as a subalgebra of $C (G)$,
choose $N \in \N$ so large that there are finite subsets
\[
F_{0} \subseteq A_{1} \otimes_{\min} A_2 \otimes_{\min}
   \cdots \otimes_{\min} A_{N},
\qquad
S_{0} \S C \left( \prod_{m = 1}^{N} G_m \right) \subseteq C (G),
\]
and also
\[
x_0 \in \bigl( A_{1} \otimes_{\min} A_2 \otimes_{\min}
   \cdots \otimes_{\min} A_{N} \bigr)_{+}
\]
and
\[
y_0 \in \bigl( A_{1}^{\af^{(1)}} \otimes_{\min} A_2^{\af^{(2)}} \otimes_{\min}
   \cdots \otimes_{\min} A_{N}^{\af^{(N)}} \bigr)_{+},
\]
such that the conditions below hold.
To state them, set
\[
B = A_{1} \otimes_{\min} A_2 \otimes_{\min}
   \cdots \otimes_{\min} A_{N}
\andeqn
H = G_1 \times G_2 \times \ldots \times G_N.
\]
Further set
\[
C = \bigotimes_{n = N + 1}^{\I} A_n
\andeqn
K = \prod_{n = N + 1}^{\I} G_n,
\]
so that
\[
A = B \otimes_{\min} C,
\qquad
G = H \times K,
\andeqn
C (G) = C (H) \otimes C (K).
\]
Then the conditions are:
\begin{enumerate}
\item\label{I_5421_InfTPAct_FF0}
For every $a \in F$ there is $b \in F_0$
such that $\| b \otimes 1_C - a \| < \ep_0$.
\item\label{I_5421_InfTPAct_F0N1}
$\| b \| \leq 1$ for all $b \in F_0$.
\item\label{I_5421_InfTPAct_SS0}
For every $f \in S$ there is $c \in S_0$
such that $\| c \otimes 1_{C (K)} - f \| < \ep_0$.
\item\label{I_5421_InfTPAct_S0N1}
$\| c \| \leq 1$ for all $c \in S_0$.
\item\label{I_5421_InfTPAct_x0x}
$\| x_0 \otimes 1_C - x \| < \ep_0$.
\item\label{I_5421_InfTPAct_y0y}
$\| y_0 \otimes 1_C - y \| < \ep_0$.
\setcounter{TmpEnumi}{\value{enumi}}
\end{enumerate}
Since $\ep_0 < 1$,
we have $(x_0 - \ep_0)_{+} \neq 0$ and $(y_0 - \ep_0)_{+} \neq 0$.
Let $\bt \colon H \to \Aut (B)$ be the tensor product action
(Definition~\ref{D_5420_Inf_tp}).
By Proposition~\ref{tensortrpacts}, this action
has the tracial Rokhlin property.
Since $B$ has strict comparison,
we may apply Proposition~\ref{P_1X17_StComp} and Definition~\ref{traR}.
We get a projection $p_0 \in B^{\bt}$ and a
unital completely positive map $\ph_0 \colon C (H) \to p_0 B p_0$
such that the following hold:
\begin{enumerate}
\setcounter{enumi}{\value{TmpEnumi}}
\item\label{Item_893_FS_equi_cen_multi_approx_Fn}
$\ph_0$ is an $(F_0, S_0, \ep_0)$-approximately equivariant
central multiplicative map.
\item\label{1_pxcompactsets_Fn}
$1 - p_0 \precsim_{B} (x_0 - \ep_0)_{+}$.
\item\label{1_pycompactsets_Fn}
$1 - p_0 \precsim_{B^{\bt}} (y_0 - \ep_0)_{+}$.
\item\label{1_ppcompactsets_Fn}
$1 - p_0 \precsim_{B^{\bt}} p_0$.
\setcounter{TmpEnumi}{\value{enumi}}
\end{enumerate}
Define $p = p_0 \otimes 1_C \in A^{\af}$.
Let $\mu$ be normalized Haar measure on~$K$.
There is a conditional
expectation $P \colon C (G) \to C (H)$ given by
$P (f) (h) = \int_{K} f (h, k) d \mu (k)$ for $f \in C (G)$ and $h \in H$.
Define $\ph \colon C (G) \to p A p$ by
$\ph (f) = (\ph_0 \circ P) (f) \otimes 1_C$ for $f \in C (G)$.

We claim that $p$ and $\ph$ satisfy the conditions of Definition~\ref{traR}.
Let $g \mapsto \Lt^G_g$ be the translation action of $G$ on $C (G)$,
and let $h \mapsto \Lt^H_h$ be the translation action of $H$ on $C (H)$.
We first observe the following basic formulas, of which the first
is used in the proofs of others:
\begin{enumerate}
\setcounter{enumi}{\value{TmpEnumi}}
\item\label{I_5421_Pc}
$P (c \otimes 1_{C (K)}) = c$ for $c \in C (H)$.
\item\label{I_5421_PLt}
$(P \circ \Lt^G_{(h, k)}) (c \otimes 1_{C (K)}) = \Lt^H_h (c)$
for $c \in C (H)$, $h \in H$, and $k \in K$.
\item\label{I_5421_phph0}
$\ph (c \otimes 1_{C (K)}) = \ph_0 (c) \otimes 1_C$
for $c \in C (H)$.
\item\label{I_5421_afbt}
$\af_{(h, k)} (b \otimes 1_C) = \bt_h (b) \otimes 1_C$
for $h \in H$ and $k \in K$.
\end{enumerate}

It is clear that $\ph$ is unital and completely positive.
For the approximation conditions, let $f, f_1, f_2 \in S$,
let $a \in F$, and let $g \in G$.
Write $g = (h, k)$ with $h \in H$ and $k \in K$.
Choose $c, c_1, c_2 \in S_0$ for  $f, f_1, f_2$
following~(\ref{I_5421_InfTPAct_SS0}),
and choose $b \in F_0$ as in~(\ref{I_5421_InfTPAct_FF0}).
In the following,
we use (\ref{I_5421_InfTPAct_F0N1}) and~(\ref{I_5421_InfTPAct_S0N1})
in estimates on differences of products without comment.
Then, using (\ref{I_5421_PLt}), (\ref{I_5421_phph0}),
and~(\ref{I_5421_afbt}) at the second step and
using~(\ref{Item_893_FS_equi_cen_multi_approx_Fn}) at the third step,
\[
\begin{split}
& \bigl\| (\ph \circ \Lt^G_g) (f) - (\af_g \circ \ph) (f) \bigr\|
\\
& \hspace*{3em} {\mbox{}}
 \leq 2 \| c \otimes 1_{C (K)} - f \|
    + \bigl\| (\ph \circ \Lt^G_g) ( c \otimes 1_{C (K)})
         - (\af_g \circ \ph) ( c \otimes 1_{C (K)}) \bigr\|
\\
& \hspace*{3em} {\mbox{}}
 = 2 \| c \otimes 1_{C (K)} - f \|
    + \bigl\| (\ph_0 \circ \Lt^H_h) (c) \otimes 1_C
         - (\bt_h \circ \ph_0) (c) \otimes 1_C \bigr\|
\\
& < 2 \ep_0 + \ep_0
  < \ep.
\end{split}
\]
Similarly, using~(\ref{I_5421_phph0}) at the second step
and~(\ref{Item_893_FS_equi_cen_multi_approx_Fn}) at the third step,
\[
\begin{split}
& \| \ph (f_1 f_2) - \ph (f_1) \ph (f_2) \|
\\
& \hspace*{3em} {\mbox{}}
 \leq 2 \| c_1 \otimes 1_{C (K)} - f_1 \|
     + 2 \| c_2 \otimes 1_{C (K)} - f_2 \|
\\
& \hspace*{6em} {\mbox{}}
    + \bigl\| \ph \bigl( (c_1 \otimes 1_{C (K)})
                  (c_2 \otimes 1_{C (K)}) \bigr)
         - \ph (c_1 \otimes 1_{C (K)}) \ph (c_2 \otimes 1_{C (K)}) \bigr\|
\\
& \hspace*{3em} {\mbox{}}
 = 2 \| c_1 \otimes 1_{C (K)} - f_1 \|
     + 2 \| c_2 \otimes 1_{C (K)} - f_2 \|
     + \| \ph_0 (c_1 c_2) - \ph_0 (c_1) \ph_0 (c_2) \|
\\
& < 2 \ep_0 + 2 \ep_0 + \ep_0
  = \ep.
\end{split}
\]
Finally, again using~(\ref{I_5421_phph0}) at the second step
and~(\ref{Item_893_FS_equi_cen_multi_approx_Fn}) at the third step,

\[
\begin{split}
\| \ph (f) a - a \ph (f) \|
& \leq 2 \| c \otimes 1_{C (K)} - f \| + 2 \| b \otimes 1_C - a \|
\\
& \hspace*{3em} {\mbox{}}
        + \bigl\| \ph (c \otimes 1_{C (K)}) (b \otimes 1_C)
          - (b \otimes 1_C) \ph (c \otimes 1_{C (K)}) \bigr\|
\\
& = 2 \| c \otimes 1_{C (K)} - f \| + 2 \| b \otimes 1_C - a \|
        + \| \ph_0 (c) b - b \ph_0 (c) \|
\\
& < 2 \ep_0 + 2 \ep_0 + \ep_0
  = \ep.
\end{split}
\]
This completes the verification of
Definition \ref{traR}(\ref{Item_893_FS_equi_cen_multi_approx}).

For Definition \ref{traR}(\ref{1_pxcompactsets}),
use~(\ref{1_pxcompactsets_Fn}) at the second step
and~(\ref{I_5421_InfTPAct_x0x}) at the fourth step to get
\[
1 - p
 = (1 - p_0) \otimes 1_C
 \precsim_A (x_0 - \ep_0)_{+} \otimes 1_C
 = (x_0 \otimes 1_C - \ep_0)_{+}
 \precsim_A x.
\]
Similarly using (\ref{1_pycompactsets_Fn})
and~(\ref{I_5421_InfTPAct_y0y}), we get $1 - p \precsim_{A^{\af}} y$,
which is Definition \ref{traR}(\ref{1_pycompactsets}).
Finally, for Definition \ref{traR}(\ref{1_ppcompactsets}),
in $A^{\af}$, using Lemma~\ref{L_5420_Fin_fix},
and using~(\ref{1_ppcompactsets_Fn}) at the second step, we have
\[
1 - p = (1 - p_0) \otimes 1_{C^{\gm}}
   \precsim_{A^{\af}} p_0 \otimes 1_{C^{\gm}}
   = p.
\]
This completes the proof.
\end{proof}

The rest of this section is the example of an action
of a totally disconnected infinite compact group on a UHF~algebra.
We abbreviate $\Z / n \Z$ to $\Z_n$; the $p$-adic integers will
not appear in this paper.
The group is $G = \prod_{n = 1}^{\infty} \mathbb{Z}_{2}$,
and the action is the infinite tensor product of copies of the same
action of $\Z_2$ on the $3^{\infty}$~UHF algebra.
We give the example in Example~\ref{Cn_1X17_Z2Inf},
and prove its properties in several results afterwards.

\begin{exa}\label{Cn_1X17_Z2Inf}
We start with a slight reformulation
of Example 10.4.8 of~\cite{lecturephill}.
For $k \in \N$, set $r (k) = \frac{1}{2} (3^{k} - 1)$.
Define $w_{k} \in {\operatorname{U}} (M_{3^{k}})$
to be the block unitary
\begin{equation}\label{Eq_1X17_wDfn}
w_{k}
 = \left( \begin{array}{ccc} 0 & 1_{M_{r (k)}} & 0 \\
   1_{M_{r (k)}} & 0 & 0 \\
   0 & 0 & 1_{\mathbb{C}} \end{array} \right) \in M_{3^{k}}.
\end{equation}
Set $B = \bigotimes_{k = 1}^{\infty} M_{3^{k}}$,
which is the $3^{\infty}$~UHF.
Define
\[
\nu = \bigotimes_{k = 1}^{\infty} \Ad (w_{k}) \in \Aut (B),
\]
which is an automorphism of order~$2$.
Let $\gamma \colon \Z_2 \to \Aut (B)$
be the product type action action generated by~$\nu$.

Define $G = \prod_{n = 1}^{\I} \Z_2$
and $A = \bigotimes_{n = 1}^{\infty} B$.
Let $\af \colon G \to \Aut (A)$ be the infinite tensor product action
as in Definition~\ref{D_5420_Inf_tp}.
Then $\af$ has the restricted \trpc{}
by Theorem~\ref{TensProdActFinStr}.
\end{exa}

\begin{prp}\label{P_1X17_NoDimR}
The action $\af \colon G \to \Aut (A)$
of Example~\ref{Cn_1X17_Z2Inf}
does not have finite Rokhlin dimension with commuting towers.
\end{prp}

\begin{proof}
Suppose $\af$ has finite Rokhlin dimension with commuting towers.
Then Proposition~3.10 of~\cite{Gar_rokhlin_2017} implies that
the action on $A$ of the first factor of~$G$,
called $H_1$ in Notation~\ref{N_1X17_Parts},
also has finite Rokhlin dimension with commuting towers.
However, $H_1 \cong \Z_2$, $A$ is the $3^{\I}$~UHF algebra and, by
Corollary 4.8(2) of~\cite{HrsPh1},
there is no action of $\Z_2$ on the $3^{\I}$~UHF algebra
which has finite Rokhlin dimension with commuting towers.
\end{proof}

In the remaining part of this section, we show that the action in
Example~\ref{Cn_1X17_Z2Inf} has the
tracial Rokhlin property with comparison
(Definition~\ref{traR}; not just the
restricted tracial Rokhlin property with comparison)
and the strong modified tracial Rokhlin property
(Definition~\ref{D_1824_AddToTRP_Mod}).
We set up some useful notation.

\begin{ntn}\label{N_1X17_Parts}
Given the notation in Example~\ref{Cn_1X17_Z2Inf}, make
the following further definitions.
For $n \in \N$ set $B_n = B$,
so that $A = \bigotimes_{m = 1}^{\infty} B_m$,
and set $A_n = \bigotimes_{m = 1}^{n} B_m$, so that $A = \dirlim_n A_n$.
For $n, k \in \N$ set $C_{n, k} = M_{3^{k}}$,
so that $B_n = \bigotimes_{k = 1}^{\infty} C_{n, k}$,
and set $B_{n, l} = \bigotimes_{k = 1}^{l} C_{n, k}$,
so that $B_n = \dirlim_k B_{n, k}$.
Further set $A_{n, l} = \bigotimes_{k = 1}^{l} B_{n, l}$.
We identify $A_{n}$ and $A_{n, l}$ with their images in~$A$,
and $B_{n, k}$ with its image in $B_n$.

Treat $G$ similarly: for $n \in \N$ set $H_n = \Z_2$,
so that $G = \prod_{m = 1}^{\I} H_m$,
and set $G_n = \prod_{m = 1}^{n} H_m$, so that $G = \invlim_n G_n$.
This gives
\[
C (G_n) = \bigotimes_{m = 1}^{n} C (H_m),
\andeqn
C (G) = \dirlim_n C (G_n) = \bigotimes_{m = 1}^{\I} C (H_m).
\]
We identify $C (G_n)$ with its image in $C (G)$.
\end{ntn}

As an informal overview, write
\[
A = \bigotimes_{m = 1}^{\infty}
   \left( \bigotimes_{k = 1}^{\infty} M_{3^k} \right)
  = \bigotimes_{m = 1}^{\infty}
   \left( \bigotimes_{k = 1}^{\infty} C_{m, k} \right).
\]
Then:
\begin{itemize}
\item
$C_{n, l}$ uses the $(n, l)$ tensor factor.
\item
$B_n$ uses the $(n, k)$ tensor factors for $k \in \N$.
\item
$B_{n, l}$ uses the $(n, k)$ tensor factors for $k = 1, 2, \ldots, l$.
\item
$A_n$ uses the $(m, k)$ tensor factors
for $m = 1, 2, \ldots, n$ and $k \in \N$.
\item
$A_{n, l}$ uses the $(m, k)$ tensor factors
for $m = 1, 2, \ldots, n$ and $k = 1, 2, \ldots, l$.
\end{itemize}

\begin{lem}\label{L_1X17_Tr_1mp}
Let $n \in \N$, let $A_1, A_2, \ldots, A_n$ be unital \ca{s},
and for $m = 1, 2, \ldots, n$ let $e_m \in A_m$ be a \pj{}
and let $\ta_m$ be a \tst{} on~$A_m$.
Let $A = A_1 \otimes A_2 \otimes \cdots \otimes A_n$
(minimal tensor product), and set
\[
e = e_1 \otimes e_2 \otimes \cdots \otimes e_n \in A
\andeqn
\ta = \ta_1 \otimes \ta_2 \otimes \cdots \otimes \ta_n \in \T (A).
\]
Then
\[
\ta (1 - e) \leq \sum_{m = 1}^{n} \ta_m (1 - e_m).
\]
\end{lem}

\begin{proof}
For $m = 1, 2, \ldots, n$ set $\ld_m = \ta_m (e_m) \in [0, 1]$.
We need to show that
\begin{equation}\label{Eq_1X17_SumProd}
1 - \prod_{m = 1}^{n} \ld_m \leq \sum_{m = 1}^{n} (1 - \ld_m).
\end{equation}
We do this by induction on~$n$.
The case $n = 1$ is immediate.
For $n = 2$, the relation~(\ref{Eq_1X17_SumProd}) becomes
\[
1 - \ld_1 \ld_2 \leq 2 - \ld_1 - \ld_2.
\]
This is equivalent to $(1 - \ld_1) (1 - \ld_2) \geq 0$,
so the case $n = 2$ holds.

Assume now~(\ref{Eq_1X17_SumProd}) holds for some $n \geq 2$,
and $\ld_1, \ld_2, \ldots \ld_{n + 1} \in [0, 1]$.
Set $\mu = \prod_{m = 1}^{n} \ld_m$.
Then $\mu \in [0, 1]$.
Using the case $n = 2$ at the second step
and the induction hypothesis at the third step, we get
\[
1 - \prod_{m = 1}^{n + 1} \ld_m
= 1 - \mu \ld_{n + 1}
\leq (1 - \ld_{n + 1}) + (1 - \mu)
\leq \sum_{m = 1}^{n + 1} (1 - \ld_m).
\]
This completes the proof.
\end{proof}

\begin{lem}\label{L_1X17_1Step}
Let the notation be as in Example~\ref{Cn_1X17_Z2Inf}.
Let $k \in \N$.
Then there are isomorphisms
\[
(M_{3^{k}})^{\Ad (w_{k})} \cong M_{r (k)} \oplus M_{r (k) + 1}
\]
and
\[
(M_{3^{k}} \otimes M_{3^{k + 1}})^{\Ad (w_{k}) \otimes \Ad (w_{k + 1})}
\cong M_{r (k^2 + k)} \oplus M_{r (k^2 + k) + 1}.
\]
The first isomorphism sends the \pj{s}
\[
e_0 = \left( \begin{array}{ccc} 1_{M_{r (k)}} & 0 & 0 \\
0 & 0 & 0 \\
0 & 0 & 0 \end{array} \right)
\andeqn
e_1 = \left( \begin{array}{ccc} 0 & 0 & 0 \\
0 & 1_{M_{r (k)}} & 0 \\
0 & 0 & 0 \end{array} \right).
\]
(using the same block matrix decomposition as in~(\ref{Eq_1X17_wDfn}))
to a \pj{} of rank $r (k)$ in $M_{r (k) + 1}$ and
to the identity of $M_{r (k)}$ respectively.
The map
\[
\rh \colon (M_{3^{k}})^{\Ad (w_{k})}
\to
(M_{3^{k}} \otimes M_{3^{k + 1}})^{\Ad (w_{k}) \otimes \Ad (w_{k + 1})}
\]
induced by $a \mapsto a \otimes 1$ induces maps
$\rh_{i, j} \colon M_{r (k) + i} \to M_{r (k^2 + k) + j}$
for $i, j \in \{ 0, 1 \}$,
and the corresponding partial embedding multiplicities $m_k (i, j)$ are
given by
\[
m_k (0, 0) = m_k (1, 1) = r (k + 1) + 1
\andeqn
m_k (0, 1) = m_k (1, 0) = r (k + 1).
\]
\end{lem}

\begin{proof}
For any $k \in \N$,
it is easy to check that $w_{k}$ is unitarily equivalent to
\[
v_{k} = \mathrm{diag} (1, 1, \ldots, 1, -1, -1, \ldots, -1)
\in M_{3^{k}},
\]
in which the diagonal entry $1$ occurs $r (k) + 1$
times and the diagonal entry $-1$ occurs $r (k)$ times.
Therefore we can prove the lemma with $v_k$ and $v_{k + 1}$
in place of $w_k$ and $w_{k + 1}$.
With this change, for example, the map
$\nu \colon M_{r (k)} \oplus M_{r (k) + 1} \to M_{3^{k}}$,
given by $\nu (a_0, a_1) = \diag (a_1, a_0)$,
is easily seen to be an isomorphism
from $M_{r (k)} \oplus M_{r (k) + 1}$ to $(M_{3^{k}})^{\Ad (w_{k})}$.
The rest of the proof is a computation with diagonal matrices
and the dimensions of their eigenspaces, and is omitted.
\end{proof}

\begin{thm}\label{T_1X17_Has_trpc}
The action $\af \colon G \to \Aut (A)$
of Example~\ref{Cn_1X17_Z2Inf}
has the tracial Rokhlin property with comparison
(Definition~\ref{traR})
and the strong modified tracial Rokhlin property
(Definition~\ref{D_1824_AddToTRP_Mod}),
using the same choices of $p \in A$ and $\ph \colon G \to p A p$.
\end{thm}

\begin{proof}
Let $F \subseteq A$ and $S \subseteq C (G)$
be finite sets, let $\varepsilon > 0$,
let $x \in A_{+} \setminus \{0 \}$ with $\big\| x \big\| = 1$,
and let $y \in A^{\alpha}_{+} \setminus \{0 \}$.
Without loss of generality we can assume
$\| a \| \leq 1$ for all $a \in F$, $\| f \| \leq 1$ for
all $f \in S$, and $\varepsilon < 1$.
According to Definition \ref{traR}, we need to find
a projection $p \in A^{\alpha}$ and a \ucp{}
$\varphi \colon C (G) \to p A p$ such that the following hold.
\begin{enumerate}
\item\label{Item_3954_approx_FS}
$\varphi$ is an $(F, S, \varepsilon)$-approximately equivariant
central multiplicative map.
\item\label{Item_3958_1mp_sub_x}
$1 - p \precsim_{A} x$.
\item\label{Item_3956_1mp_sub_y}
$1 - p \precsim_{A^{\alpha}} y$.
\item\label{Item_1mp_fixed_sub_p}
$1 - p \precsim_{A^{\alpha}} p$.
\item\label{Item_pxp_no0}
$\| p x p \| > 1 - \varepsilon$.
\setcounter{TmpEnumi}{\value{enumi}}
\end{enumerate}
According to Definition \ref{moditra}, we also need to find a partial
isometry $s \in A^{\alpha}$ such that the following hold.
\begin{enumerate}
\setcounter{enumi}{\value{TmpEnumi}}
\item\label{Item_3971_s_p_1mp}
$s^{*} s = 1 - p$ and $s s^{*} \leq p$.
\item\label{Item_3973_comm_saas}
$\| s a - a s \| < \varepsilon$ for all $a \in F$.
\end{enumerate}

We can ignore~(\ref{Item_pxp_no0}).
We can also ignore~(\ref{Item_1mp_fixed_sub_p}),
since it follows from~(\ref{Item_3971_s_p_1mp}).

Since $A$ is a UHF~algebra,
there is a nonzero \pj{} $q_1 \in \ov{x A x}$,
and the unique \tst~$\ta$ on~$A$ satisfies $\ta (q_1) > 0$.
Set $\dt_1 = \ta (q_1)$.

The action $\gm$ has the tracial Rokhlin property by Example 10.4.8
of \cite{lecturephill}.
Therefore $C^* (\Z_2, B, \gm)$ is simple by
Corollary 1.6 of~\cite{phill23},
so $B^{\gm}$ is simple by Theorem 3.5 of \cite{MhkPh1}.
Since $\gm$ is a direct limit action, $B^{\gm}$ is an AF~algebra.
It is easy to check that $A^{\af}$ can be identified with
$\bigotimes_{m = 1}^{\infty} B_m^{\gm}$.
It follows that $A^{\af}$ is an AF~algebra, which is
simple because it is an infinite tensor product of simple \ca{s}.
Therefore there is a nonzero \pj{} $q_2 \in \ov{y A^{\af} y}$,
and the number $\dt_2 = \inf_{\ta \in T (A^{\af})} \ta (q_2)$
satisfies $\dt_2 > 0$.

Following Notation~\ref{N_1X17_Parts},
$\bigcup_{n = 1}^{\I} A_n$ is dense in~$A$ and, for every $n \in \N$,
$\bigcup_{l = 1}^{\I} B_{n, l}$ is dense in $B_n$.
Therefore there are $N_1, L_0 \in \N$
and a finite subset $F_0 \subseteq A_{N_1, L_0} \subseteq A$
such that for every
$a \in F$ there is $b \in F_0$ with $\| a - b \| < \frac{\ep}{4}$,
and also $\| b \| \leq 1$ for all $b \in F_0$.
Similarly, there are $N_2 \in \N$
and a finite subset $S_0 \S C (G_{N_2}) \subseteq C (G)$
such that for every
$f \in F$ there is $c \in F_0$ with $\| f - c \| < \frac{\ep}{4}$
and also $\| c \| \leq 1$ for all $c \in S_0$.
Set $N = \max (N_1, N_2)$, and choose $L \in \N$ so large that
\[
L \geq L_0
\andeqn
\frac{2 N}{3^{L}} < \min \left( \dt_1, \dt_2, \frac{1}{2} \right).
\]

Let $e_0, e_1 \in M_{3^{L + 1}}$ be as in Lemma~\ref{L_1X17_1Step},
with $k = L + 1$.
For $m = 1, 2, \ldots, N$ define the \pj{s}
\[
e_0^{(m)} = 1_{B_{m, L}} \otimes e_0, \, \,
e_1^{(m)}= 1_{B_{m, L}} \otimes e_1
\in B_{m, L} \otimes M_{3^{L + 1}}
= B_{m, L + 1}
\S B_m.
\]
Identify $H_m = \Z_2 = \{ 0, 1 \}$ with addition modulo~$2$,
and for
$h = (h_1, h_2, \ldots, h_{N}) \in G_{N} = \prod_{m = 1}^{N} H_m$,
set
\[
e_h = e_{h_1}^{(1)} \otimes e_{h_2}^{(2)}
\otimes \cdots \otimes e_{h_{N}}^{(N)}.
\]
These are \mops.
Define $p = \sum_{h \in G_{N}} e_h \in A_{n, L + 1} \S A$.
As the proof of Theorem \ref{TensProdActFinStr}, there is a unital \hm{}
$\ph_0 \colon C (G_{N}) \to p A_{n, L + 1} p \S p A p$
given by $\ph_0 (f) = \sum_{h \in G_{N}} f (h) e_h$
for $f \in C (G_{N})$.

Set $K = \prod_{m = N + 1}^{\I} H_m$, so that $G = G_{N} \times K$.
Let $\mu$ be normalized Haar measure on~$K$.
Then there is a conditional expectation
$P \colon C (G) \to C (G_{N})$,
given by $P (f) (h) = \int_K f (h, g) \, d \mu (g)$
for $f \in C (G)$ and $h \in G_{N}$.
By the same reason as the
proof of Theorem \ref{TensProdActFinStr}, $\ph = \ph_0 \circ P \colon C (G) \to p A p$
is an equivariant \ucp{} map.

Again, as the proof of Theorem \ref{TensProdActFinStr}, condition
(\ref{Item_3954_approx_FS}) holds.

For $m = 1, 2, \ldots, n$, set $p_m = e_0^{(m)} + e_0^{(m)} \in A_m$.
Then $p_m$ is the image in $A_m$
of a \pj{} $z_m \in C_{m, L + 1} = M_{3^{L + 1}}$
such that $1 - z_m$ has rank~$1$.
Therefore the unique \tst{} $\ta_m$ on $A_m$ satisfies
$\ta (1 - p_m) = 3^{- (L + 1)}$.
Since $p = p_1 \otimes p_2 \otimes \cdots \otimes p_N$,
by Lemma~\ref{L_1X17_Tr_1mp} we have
\[
\ta (1 - p) \leq \frac{N}{3^{L + 1}} < \dt_1 = \ta (q_1).
\]
Since UHF~algebras have strict comparison,
we get $1 - p \precsim_A q_1 \precsim_A x$,
which is~(\ref{Item_3958_1mp_sub_x}).

For the remaining conditions,
for convenience set $T = (L + 1) (L + 2)$.
Let $e_0$, $e_1$, and $\rh$ be as in Lemma~\ref{L_1X17_1Step},
with $k = L + 1$, and let the components of $\rh$
in its codomain be
\[
\rh_j \colon M_{r (L + 1)} \oplus M_{r (L + 1) + 1} \to M_{r (T) + j}
\]
for $j = 0, 1$.
Using the ranks and partial embedding multiplicities
given in Lemma~\ref{L_1X17_1Step}, we see that
$\rh_0 (1 - e_0 - e_1) \in M_{r (T)}$ has rank $r (L + 2)$
and $\rh_1 (1 - e_0 - e_1) \in M_{r (T) + 1}$ has rank $r (L + 2) + 1$.
The normalized traces of these are
\begin{equation}\label{Eq_1X21_NTraces}
\frac{r (L + 2)}{r (T)} < \frac{2}{3^{L + 1}}
\andeqn
\frac{r (L + 2) + 1}{r (T) + 1} < \frac{2}{3^{L + 1}}.
\end{equation}

Set
\[
D = (M_{r (L + 1)} \oplus M_{r (L + 1) + 1})^{\otimes N}
\andeqn
E = (M_{r (T)} \oplus M_{r (T) + 1})^{\otimes N},
\]
and consider $\rh^{\otimes N} \colon D \to E$.
We can write
\[
E = \bigoplus_{j \in \{ 0, 1 \}^N}
M_{r (T) + j_1} \otimes M_{r (T) + j_2} \otimes
\cdots \otimes M_{r (T) + j_N}.
\]
Call the $j$~tensor factor $E_j$.
For $j = (j_1, j_2, \ldots, j_N) \in \{ 0, 1 \}^N$, let
$d_j$ be the image in $E_j$ of the corresponding summand of
$\rh^{\otimes N} \bigl( (e_0 + e_1)^{\otimes N} \bigr)$.
Thus
\[
d_j = \rh_{j_1} (e_0 + e_1) \otimes \rh_{j_2} (e_0 + e_1) \otimes
\cdots \otimes \rh_{j_N} (e_0 + e_1).
\]
By Lemma~\ref{L_1X17_Tr_1mp} and~(\ref{Eq_1X21_NTraces}),
$1 - d_j$ has normalized trace less than $2 N \cdot 3^{- L - 1}$.

Use Lemma~\ref{L_1X17_1Step} to identify $D$ and $E$
with the subalgebras
\[
\left( \bigotimes_{m = 1}^N C_{m, L + 1} \right)^{\af}
\andeqn
\left( \bigotimes_{m = 1}^N
C_{m, L + 1} \otimes C_{m, L + 2} \right)^{\af},
\]
in such a way that $(e_0 + e_1)^{\otimes N}$ is identified with~$p$.
Under this identification, $E$ commutes exactly with all elements
of $A_{N, L}$, hence with all elements of~$F_0$.

Since $2 N \cdot 3^{- L - 1} < \frac{1}{2}$,
we conclude that $1 - d_j \precsim_{E_j} d_j$.
Therefore $1 - p \precsim_E p$, that is,
there is $s \in E$ such that $s^{*} s = 1 - p$ and $s s^{*} \leq p$.
We have $s \in A^{\alpha}$ since $E \S A^{\alpha}$.
Also, $s$ commutes exactly with all elements of~$F_0$,
so $\| a s - s a \| < \frac{\ep}{2}$ for all $a \in F$.
We have verified Conditions (\ref{Item_3971_s_p_1mp})
and~(\ref{Item_3973_comm_saas}).

Since $2 N \cdot 3^{- L - 1} < \dt_2$, for every $\sm \in \T (E)$
we have $\sm (1 - p) < \dt_2$.
Every \tst{} $\ta \in \T (A^{\af})$ restricts to a \tst{} on~$E$,
so $\ta (1 - p) < \dt_2$ for all $\ta \in \T (A^{\af})$.
Since simple AF~algebras have strict comparison,
we get $1 - p \precsim_{A^{\af}} q_2$.
Condition~(\ref{Item_3956_1mp_sub_y}) follows.
\end{proof}

\section{An action of $S^1$ on a simple AT~algebra}\label{Sec_1908_Exam_TRPS1}

The purpose of this section is to construct a direct limit action
of the group~$S^1$
on a simple unital AT~algebra which has
the tracial Rokhlin property with comparison
but does not have finite Rokhlin dimension with commuting towers.

The general construction, with unspecified partial embedding
multiplicities (which, for properties we want,
need to be chosen appropriately),
is presented in Construction~\ref{Cn_1824_S1An}.
For the purpose of readability, the properties asserted there,
as well as others needed later, are proved in a series of lemmas.

The algebra $A$ in our construction will be a direct limit of
algebras $A_n$ isomorphic to
$C (S^1, M_{N r_0 (n)}) \oplus C (S^1, M_{r_1 (n)})$.
Up to equivariant isomorphism and exterior equivalence,
the action of $\zt \in S^1$ on $C (S^1, M_{N r_0 (n)})$ is rotation
by $\zt^N$ and its action on $C (S^1, M_{r_1 (n)})$
is rotation by $\zt$.
It is technically convenient to present the first summand in a
different way; the description above
is explicit in Lemma~\ref{L_1824_R_T}.
The action has the tracial Rokhlin property with comparison provided
the image of the summand $C (S^1, M_{N r_0 (n)}) \S A_n$ in $A$
can be made ``arbitrarily small in trace'' by choosing $n$ large enough.
Actions obtained using different values of~$N$ are not conjugate.

The algebras $B_n$ and~$B$ in parts
(\ref{Item_1824_S1_Bn0Bn1}), (\ref{Item_1824_S1_FixPtSys}),
(\ref{Item_1824_S1_FixLim}), and~(\ref{Item_1915_S1_qn})
of Construction~\ref{Cn_1824_S1An}
are a convenient description of the fixed point algebras of $A_n$
and~$A$; see Lemma~\ref{L_1916_FPisB}.

We say here a little more about the motivation for the construction
and possible extensions.
If $G$ is finite, one can construct a direct limit action of $G$ on an
AF~algebra $\dirlim_n A_n$
by taking $A_n = M_{r_0 (n)} \oplus C (G, M_{r_1 (n)})$.
The action on $C (G, M_{r_1 (n)})$ is essentially
translation by group elements.
The partial map from $C (G, M_{r_1 (n)})$ to $M_{r_0 (n + 1)}$
is the direct sum of the evaluations at the points of~$G$.
The action on $M_{r_0 (n + 1)}$ is inner, and must permute the
images of the maps from $C (G, M_{r_1 (n)})$ appropriately;
this leads to a slightly messy inductive construction of inner actions
of $G$ on the algebras $M_{r_0 (n)}$ and inner perturbations
of the translation actions on $C (G, M_{r_1 (n)})$.

When $G$ is not finite, point evaluations can no longer be used,
since equivariance forces one to use all of them or none of them.
The algebra $C (S^1, M_{N r_0 (n)})$
with the action of rotation by $\zt^N$
is the codomain for a usable substitute for point evaluations.
Something similar to the inductive construction of perturbations
from above is needed, but the messiness can be mostly hidden
by instead using the algebra $R$ as in
Construction \ref{Cn_1824_S1An}(\ref{Item_1824_S1_R}).

Construction~\ref{Cn_1824_S1An} can be generalized in several ways.
One can replace $C (S^1, M_{r_1 (n)})$ with $C (X, M_{r_1 (n)})$
for a compact space~$X$ with a free action of $S^1$.
To ensure simplicity, one will need to incorporate additional
partial maps
in the direct system, which can be roughly described as point
evaluations at points of $X / S^1$.
One can increase the complexity of the K-theory and the departure
from the Rokhlin property by taking
\[
A_n = C (S^1, M_{N_2 N_1 r_0 (n)})
   \oplus C (S^1, M_{N_1 r_1 (n)}) \oplus C (S^1, M_{r_2 (n)})
\]
with actions exterior equivalent to rotations by
$\zt^{N_1 N_2}$, $\zt^{N_1}$, and $\zt$.
One can use more summands, even letting the number of them
approach infinity as $n \to \I$.
One can also replace $S^1$ with $(S^1)^m$.
However, it is not clear how to construct an analogous action
with $S^1$ replaced by a nonabelian connected compact Lie group.

We introduce some notation specifically for this section.

\begin{dfn}\label{D_1908_eue}
Let $G$ be a group, let $A$ and $B$ be \ca{s}, with $B$ unital,
let $\af \colon G \to \Aut (A)$ and $\bt \colon G \to \Aut (B)$
be actions of $G$ on $A$ and~$B$,
and let $\ph, \ps \colon A \to B$ be equivariant \hm{s}.
We say that $\ph$ and $\ps$ are
{\emph{equivariantly unitarily equivalent}}, written $\ph \sim \ps$,
if there is a $\bt$-invariant unitary $u \in B$
such that $u \ph (a) u^* = \ps (a)$ for all $a \in A$.
\end{dfn}

\begin{ntn}\label{N_1908_Ampl}
Let $A$ and $B$ be \ca{s}, and let $\ps \colon A \to B$ be a \hm.
We let $\ps^{(k)} \colon A \to M_k \otimes B$ be the map
$a \mapsto 1_{M_k} \otimes \ps (a)$,
and we define
\[
\ps_n = \id_{M_n} \otimes \ps \colon M_n \otimes A \to M_n \otimes B
\]
and
\[
\ps^{(k)}_n = \id_{M_n} \otimes \ps^{(k)} \colon
   M_n \otimes A \to M_{k n} \otimes B.
\]
\end{ntn}

In particular,
the ``amplification map'' from $M_n (A)$ to $M_{k n} (A)$,
given by $a \mapsto 1_{M_k} \otimes a$, is denoted by $(\id_A)^{(k)}_n$.

\begin{cns}\label{Cn_1824_S1An}
We choose and fix $N \in \N$ with $N \geq 2$, $\te \in \R \SM \Q$,
$r (0) = (r_0 (0), r_1 (0)) \in \N^2$,
and, for $n \in \Nz$ and $j, k \in \{ 0, 1 \}$,
numbers $l_{j, k} (n) \in \N$.
We suppress them in the notation for the objects we construct,
and, in later results, we will impose additional restrictions on them.

We then define the following \ca{s}, maps, and actions of~$S^1$.
\begin{enumerate}
\item\label{Item_1824_S1_CS1}
Define $\bt \colon S^1 \to \Aut (C (S^1))$
by $\bt_{\zt} (f) (z) = f (\zt^{-1} z)$ for $\zt, z \in S^1$.
Further, for $n \in \N$, identify $C (S^1, M_n)$ and $M_n (C (S^1))$
with $M_n \otimes C (S^1)$ in the obvious way, and
let $\bt_n \colon S^1 \to \Aut (C (S^1, M_n))$ be given by
$\bt_{\zt, n} = \id_{M_n} \otimes \bt_{\zt}$ for $\zt \in S^1$.
(The order of subscripts in $\bt_{\zt, n}$ is chosen to be
consistent with Notation~\ref{N_1908_Ampl}).
To simplify notation,
for $\ld \in \R$ we define
\[
{\widetilde{\bt}}_{\ld} = \bt_{\exp (2 \pi i \ld)}
\andeqn
{\widetilde{\bt}}_{\ld, n} = \bt_{\exp (2 \pi i \ld), \, n}.
\]
\item\label{Item_1824_S1_s_om}
Define
\[
\om = \exp (2 \pi i / N)
\andeqn
s = \left( \begin{matrix}
  0     &  1     &  0     & \cdots &  0     &  0        \\
  0     &  0     &  1     & \cdots &  0     &  0        \\
 \vdots & \vdots & \ddots & \ddots & \vdots & \vdots    \\
 \vdots & \vdots & \ddots & \ddots &  1     &  0        \\
  0     &  0     & \cdots & \cdots &  0     &  1        \\
  1     &  0     & \cdots & \cdots &  0     &  0
\end{matrix} \right)
\in M_N.
\]
\item\label{Item_1824_S1_R}
Define
\[
R = \bigl\{ f \in C (S^1, M_N) \colon
 {\mbox{$f (\om z) = s f (z) s^*$ for all $z \in S^1$}} \bigr\}.
\]
Then $R$ is invariant under the action $\bt_N$ of $S^1$
on $C (S^1, M_N)$ above.
(See Lemma~\ref{L_1824_R_inv} below.)
We define $\gm \colon S^1 \to \Aut (R)$ to be the restriction
of this action.
Further, for $n \in \N$,
let $\gm_n \colon S^1 \to \Aut (M_n \otimes R)$ be the action
$\gm_{\zt, n} = \id_{M_n} \otimes \gm_{\zt}$ for $\zt \in S^1$.
Finally, for $\ld \in \R$ we define
\[
{\widetilde{\gm}}_{\ld} = \gm_{\exp (2 \pi i \ld)}
\andeqn
{\widetilde{\gm}}_{\ld, n} = \gm_{\exp (2 \pi i \ld), \, n}.
\]
\item\label{Item_1824_S1_Crs}
Let $\io \colon R \to C (S^1, M_N)$ be the inclusion.
Define $\xi \colon C (S^1) \to R$ by
\[
\xi (f) (z)
 = {\operatorname{diag}} \bigl( f (z), \, f (\om z),
    \, \ldots, \, f (\om^{N - 1} z) \bigr)
\]
for $f \in C (S^1)$ and $z \in S^1$.
\item\label{Item_1824_S1_ll}
For $n \in \N$ write
\[
l (n) = \left( \begin{matrix}
l_{0, 0} (n)   & l_{0, 1} (n)    \\
N l_{1, 0} (n) & 2N l_{1, 1} (n)
\end{matrix} \right).
\]
For $n \in \N$ inductively define,
starting with $r (0) = (r_0 (0), r_1 (0)) \in \N^2$ as at
the beginning of the construction,
\begin{equation}\label{Eq_1901_rn_dfm}
r (n + 1) = l (n) r (n)
\end{equation}
(usual matrix multiplication).
\item\label{Item_1824_S1_An}
For $n \in \Nz$ set
\[
A_{n, 0} = M_{r_0 (n)} (R),
\quad
A_{n, 1} = M_{r_1 (n)} \bigl( C (S^1) \bigr),
\quad {\mbox{and}} \quad
A_n = A_{n, 0} \oplus A_{n, 1}.
\]
Define an action $\af^{(n)} \colon S^1 \to \Aut (A)$
(notation not in line with Notation~\ref{N_1908_Ampl}) by,
now following Notation~\ref{N_1908_Ampl},
$\af^{(n)}_{\zt} = \gm_{\zt, r_0 (n)} \oplus \bt_{\zt, r_1 (n)}$.
\item\label{Item_1824_S1_nu_p}
For $n \in \Nz$ and $j, k \in \{ 0, 1 \}$, define maps
\[
\nu_{n, 0, 0} \colon A_{n, 0} \to M_{l_{0, 0} (n) r_0 (n)} (R),
\qquad
\nu_{n, 0, 1} \colon A_{n, 1} \to M_{l_{0, 1} (n) r_1 (n)} (R),
\]
\[
\nu_{n, 1, 0} \colon A_{n, 0}
 \to M_{N l_{1, 0} (n) r_0 (n)} \bigl( C (S^1) \bigr),
\]
and
\[
\nu_{n, 1, 1} \colon A_{n, 1}
 \to M_{2 N l_{1, 1} (n) r_1 (n)} \bigl( C (S^1) \bigr)
\]
as follows.
Recalling Notation~\ref{N_1908_Ampl}, set
\[
\nu_{n, 0, 0} = (\id_R)_{r_0 (n)}^{l_{0, 0} (n)},
\qquad
\nu_{n, 0, 1} = \xi_{r_1 (n)}^{l_{0, 1} (n)},
\qquad
\nu_{n, 1, 0} = \io_{r_0 (n)}^{l_{1, 0} (n)},
\]
and
\[
\begin{split}
\nu_{n, 1, 1}
& = \diag \left( \btt_{0, \, r_1 (n)}^{l_{1, 1} (n)},
        \, \btt_{1 / N, \, r_1 (n)}^{l_{1, 1} (n)},
        \, \btt_{2 / N, \, r_1 (n)}^{l_{1, 1} (n)},
        \, \ldots,
        \, \btt_{(N - 1) / N, \, r_1 (n)}^{l_{1, 1} (n)}, \right.
\\
& \hspace*{4em} {\mbox{}}
        \left. \, \btt_{\te, \, r_1 (n)}^{l_{1, 1} (n)},
        \, \btt_{\te + 1 / N, \, r_1 (n)}^{l_{1, 1} (n)},
        \, \btt_{\te + 2 / N, \, r_1 (n)}^{l_{1, 1} (n)},
        \, \ldots,
        \, \btt_{\te + (N - 1) / N, \, r_1 (n)}^{l_{1, 1} (n)} \right).
\end{split}
\]
Up to equivariant unitary equivalence, the last one can be written
in the neater form
\[
\begin{split}
& \diag \bigl( \btt_{0},
        \, \btt_{1 / N},
        \, \btt_{2 / N},
        \, \ldots,
        \, \btt_{(N - 1) / N},
\\
& \hspace*{4em} {\mbox{}}
        \, \btt_{\te},
        \, \btt_{\te + 1 / N},
        \, \btt_{\te + 2 / N},
        \, \ldots,
        \, \btt_{\te + (N - 1) / N} \bigr)_{r_1 (n)}^{l_{1, 1} (n)}.
\end{split}
\]
\item\label{Item_1824_S1_nu_A}
Define $\nu_n \colon A_n \to A_{n + 1}$ by
\[
\nu_n (a_0, a_1)
 = \bigl(
  \diag \bigl( \nu_{n, 0, 0} (a_0), \,  \nu_{n, 0, 1} (a_1) \bigr), \,\,
  \diag \bigl( \nu_{n, 1, 0} (a_0), \,  \nu_{n, 1, 1} (a_1) \bigr)
 \bigr)
\]
for $a_0 \in A_{n, 0}$ and $a_1 \in A_{n, 1}$.
For $m, n \in \Nz$ with $m \leq n$ set
\[
\nu_{n, m}
 = \nu_{n - 1} \circ \nu_{n - 2} \circ \cdots \circ \nu_m
 \colon A_m \to A_n.
\]
\item\label{Item_1924_S1_A}
Let $A$ be the direct limit of the system
$\bigl( (A_n)_{n \in \Nz}, \, (\nu_{n, m})_{m \leq n} \bigr)$,
with maps $\nu_{\I, m} \colon A_m \to A$.
Equip $A$ with the direct limit action $\af = \dirlim \af^{(n)}$
of~$S^1$.
This action exists by Lemma~\ref{L_1909_nu_e} below.
\item\label{Item_1824_S1_pn}
For $n \in \Nz$ let $p_n = (0, 1) \in A_{n, 0} \oplus A_{n, 1} = A_n$.
\item\label{Item_1824_S1_Bn0Bn1}
For $n \in \Nz$ set
\[
B_{n, 0} = (M_{r_0 (n)})^N = \bigoplus_{k = 0}^{N - 1} M_{r_0 (n)},
\quad
B_{n, 1} = M_{r_1 (n)},
\quad {\mbox{and}} \quad
B_n = B_{n, 0} \oplus B_{n, 1}.
\]
\item\label{Item_1824_S1_FixPtSys}
Let $\mu \colon {\mathbb{C}} \to {\mathbb{C}}^N$ be
$\mu (\ld) = (\ld, \ld, \ldots, \ld)$
for $\ld \in {\mathbb{C}}$,
and let $\dt \colon {\mathbb{C}}^N \to M_N$ be
$\dt (\ld_0, \ld_1, \ldots, \ld_{N - 1})
 = \diag (\ld_0, \ld_1, \ldots, \ld_{N - 1})$
for $\ld_0, \ld_1, \ldots, \ld_{N - 1} \in {\mathbb{C}}$.
For $n \in \Nz$ and $j, k \in \{ 0, 1 \}$,
and recalling Notation~\ref{N_1908_Ampl}, define maps
\[
\ch_{n, 0, 0} \colon B_{n, 0} \to M_{l_{0, 0} (n) r_0 (n)} ({\mathbb{C}}^N),
\qquad
\ch_{n, 0, 1} \colon B_{n, 1} \to M_{l_{0, 1} (n) r_1 (n)} ({\mathbb{C}}^N),
\]
\[
\ch_{n, 1, 0} \colon B_{n, 0} \to M_{N l_{1, 0} (n) r_0 (n)},
\andeqn
\ch_{n, 1, 1} \colon B_{n, 1} \to M_{2 N l_{1, 1} (n) r_1 (n)}
\]
by
\[
\ch_{n, 0, 0} = (\id_{{\mathbb{C}}^N})_{r_0 (n)}^{l_{0, 0} (n)},
\qquad
\ch_{n, 0, 1} = \mu_{r_1 (n)}^{l_{0, 1} (n)},
\]
\[
\ch_{n, 1, 0} = \dt_{r_0 (n)}^{l_{1, 0} (n)},
\andeqn
\ch_{n, 1, 1} = (\id_{\mathbb{C}})_{r_1 (n)}^{2 N l_{1, 1} (n)}.
\]
\item\label{Item_1824_S1_FixLim}
Define $\ch_n \colon B_n \to B_{n + 1}$ by
\[
\ch_n (a_0, a_1)
 = \bigl(
  \diag \bigl( \ch_{n, 0, 0} (a_0), \,  \ch_{n, 0, 1} (a_1) \bigr), \,\,
  \diag \bigl( \ch_{n, 1, 0} (a_0), \,  \ch_{n, 1, 1} (a_1) \bigr)
 \bigr)
\]
for $a_0 \in B_{n, 0}$ and $a_1 \in B_{n, 1}$.
For $m, n \in \Nz$ with $m \leq n$ set
\[
\ch_{n, m}
 = \ch_{n - 1} \circ \ch_{n - 2} \circ \cdots \circ \ch_m
 \colon B_m \to B_n.
\]
Let $B$ be the direct limit of the system
$\bigl( (B_n)_{n \in \Nz}, \, (\ch_{n, m})_{m \leq n} \bigr)$,
with maps $\ch_{\I, m} \colon B_m \to B$.
\item\label{Item_1915_S1_qn}
For $n \in \Nz$ let $q_n = (0, 1) \in B_{n, 0} \oplus B_{n, 1} = B_n$.
\end{enumerate}

\end{cns}

\begin{lem}\label{L_1824_R_inv}
Let $R \S C (S^1, M_N)$ be as in
Construction \ref{Cn_1824_S1An}(\ref{Item_1824_S1_R}).
Then $R$ is invariant under the action $\bt_N$ of
Construction \ref{Cn_1824_S1An}(\ref{Item_1824_S1_CS1}),
and map $\io \colon R \to C (S^1, M_N)$
of Construction \ref{Cn_1824_S1An}(\ref{Item_1824_S1_Crs})
is equivariant.
\end{lem}

\begin{proof}
The first part is easy to check from the definitions
of $\bt_N$ and $R$.
Since $\io$ is the inclusion, the second part is immediate.
\end{proof}

\begin{lem}\label{L_1824_phe}
The map $\xi \colon C (S^1) \to R$
of Construction \ref{Cn_1824_S1An}(\ref{Item_1824_S1_Crs})
is well defined and equivariant.
\end{lem}

\begin{proof}
The first part is easy to check just by the definition of $\xi$.
For equivariance, it is enough and immediate
to check on the usual generator of $C (S^1)$.
\end{proof}

\begin{lem}\label{L_1909_nu_e}
The maps $\nu_{n, j, k} \colon A_{n, k} \to A_{n, j}$
of Construction \ref{Cn_1824_S1An}(\ref{Item_1824_S1_nu_p})
and $\nu_n \colon A_n \to A_{n + 1}$
of Construction \ref{Cn_1824_S1An}(\ref{Item_1824_S1_nu_A})
are equivariant.
\end{lem}

\begin{proof}
This is immediate from Lemma~\ref{L_1824_R_inv}, Lemma~\ref{L_1824_phe},
and the fact that $\bt_{\zt_1}$ commutes with $\bt_{\zt_2}$
for $\zt_1, \zt_2 \in S^1$.
\end{proof}

We will need the notation $L_{\Ph}$ from 1.3 of~\cite{DNNP}.
For a \chs~$X$, $m \in \N$, and $x \in X$,
let $\ev_x \colon C (X, M_m) \to M_m$ be evaluation at~$x$.
If also $Y$ is a \chs{} $\Ph \colon C (X, M_m) \to C (Y, M_n)$ is a \hm,
then $L_{\Ph}$ assigns to $y \in Y$ the set of all $x \in X$
such that $\ev_x$ occurs as a summand in the representation
$\ev_y \circ \Ph$.
The definition in 1.3 of~\cite{DNNP} is extended from this case
to \hm{s} between direct sums of algebras of this type.
We refer to that paper for details.

\begin{lem}\label{L_1824_R_T}
Let $R \S C (S^1, M_N)$ be as in
Construction \ref{Cn_1824_S1An}(\ref{Item_1824_S1_R}).
Let ${\overline{\gm}} \colon S^1 \to \Aut (C (S^1))$
be the action
${\overline{\gm}}_{\zt} (f) (z) = f (\zt^{-N} z)$ for $\zt, z \in S^1$.
Then there is an isomorphism $\ps \colon R \to C (S^1, M_N)$
satisfying the following conditions.
\begin{enumerate}
\item\label{Item_824_R_T_Rk1}
For every rank one \pj{} $e \in R \S C (S^1, M_N)$,
the \pj{} $\ps (e) \in C (S^1, M_N)$ has rank one.
\item\label{Item_824_R_T_ExtEq}
The action $\zt \mapsto \ps \circ \gm_{\zt} \circ \ps^{-1}$
is exterior equivalent to the action
$\zt \mapsto {\overline{\gm}}_{\zt, N}$.
\item\label{Item_824_R_T_L}
With $L_{\Ph}$ as defined in 1.3 of~\cite{DNNP},
for $z \in S^1$ we have
\[
L_{\ps \circ \xi} (z)
 = \bigl\{ y \in S^1 \colon y^N = z \bigr\}
\andeqn
L_{\io \circ \ps^{-1}} (z) = \{ z^N \}.
\]
\end{enumerate}
\end{lem}

The isomorphism in this lemma is not equivariant when
$C (S^1, M_N)$ is equipped with the action $\zt \mapsto \bt_{\zt, N}$,
or any action exterior equivalent to it.

\begin{proof}[Proof of Lemma~\ref{L_1824_R_T}]
Choose a \ct{} unitary path $\ld \mapsto s_{\ld}$ in $M_N$,
defined for $\ld \in [0, 1]$,
such that $s_0 = 1$ and $s_1 = s$.
For any $\ld \in \R$,
choose $\ld_0 \in [0, 1)$ such that $\ld - \ld_0 \in \Z$.
Taking $n = \ld - \ld_0$, we then define $s_{\ld} = s^n s_{\ld_0}$.
This function is still \ct.
Moreover, for any $m \in \Z$,
\begin{equation}\label{Eq_1912_ldm}
s_{\ld + m}
 = s^{m + n} s_{\ld_0}
 = s^{m} s^{n} s_{\ld_0} = s^{m} s_{\ld}.
\end{equation}

We claim that there is a well defined \hm{}
$\ps \colon R \to C (S^1, M_N)$ such that, whenever $f \in R$
and $\ld \in \R$, we have
\[
\ps (f) (e^{2 \pi i \ld}) = s_{\ld}^* f (e^{2 \pi i \ld / N}) s_{\ld}.
\]
The only issue is whether $\ps (f) (e^{2 \pi i \ld})$ is well defined.
It is sufficient to prove that if $\ld_0 \in [0, 1)$ and
$n = \ld - \ld_0 \in \Z$, then the formulas for
$\ps (f) (e^{2 \pi i \ld})$ and $\ps (f) (e^{2 \pi i \ld_0})$ agree.
To see this, use the definition of $R$ at the second step to get
\[
s_{\ld}^* f (e^{2 \pi i \ld / N}) s_{\ld}
 = s_{\ld_0}^* s^{- n} f (\om^{n} e^{2 \pi i \ld_0 / N}) s^n s_{\ld}
 = s_{\ld_0}^* f (e^{2 \pi i \ld_0 / N}) s_{\ld_0},
\]
as desired.

The construction of $\ps$ makes Part~(\ref{Item_824_R_T_Rk1}) obvious.
Bijectivity is easy just by checking the definition.
We now prove~(\ref{Item_824_R_T_ExtEq}).
For $\zt \in S^1$, choose $\ta \in \R$ such that
$e^{2 \pi i \ta} = \zt$,
and define a function $v_{\zt} \in C (S^1, M_N)$
by $v_{\zt} (e^{2 \pi i \ld}) = s_{\ld}^* s_{\ld - N \ta}$
for $\ld \in \R$.
We claim that $v_{\zt}$ is well defined.
First, we must show that if $m \in \Z$ then
\[
s_{\ld + m}^* s_{\ld + m - N \ta} = s_{\ld}^* s_{\ld - N \ta}.
\]
This follows directly from~(\ref{Eq_1912_ldm}).
Second, we must show that if $e^{2 \pi i \ta_1} = e^{2 \pi i \ta_2}$,
then
\[
s_{\ld}^* s_{\ld - N \ta_1} = s_{\ld}^* s_{\ld - N \ta_2}.
\]
For this, set $m = \ta_1 - \ta_2 \in \Z$, and use~(\ref{Eq_1912_ldm})
and $s^N = 1$ to see that
\[
s_{\ld - N \ta_2} = s_{\ld - N \ta_1 + N m}
 = s^{N m} s_{\ld - N \ta_1} = s_{\ld - N \ta_1}.
\]
The claim is proved.

It is now easy to check that
$(\zt, \ld) \mapsto v_{\zt} (e^{2 \pi i \ld})$ is \ct,
so that $\zt \mapsto v_{\zt}$ is a \cfn{} from $S^1$
to the unitary group of $C (S^1, M_N)$.

We next claim that
$v_{\zt_1 \zt_2} = v_{\zt_1} \bt_{\zt_1^N, N} (v_{\zt_2})$
for $\zt_1, \zt_2 \in S^1$.
To do this, choose $\ta_1, \ta_2 \in \R$
such that $\zt_1 = e^{2 \pi i \ta_1}$ and $\zt_2 = e^{2 \pi i \ta_2}$.
Then $\zt_1 \zt_2 = e^{2 \pi i (\ta_1 + \ta_2)}$.
So for $\ld \in \R$,
\[
\begin{split}
v_{\zt_1} (e^{2 \pi i \ld})
    \bt_{\zt_1^N, N} (v_{\zt_2}) (e^{2 \pi i \ld})
& = v_{\zt_1} (e^{2 \pi i \ld})
       v_{\zt_2} \bigl( e^{2 \pi i (\ld - N \ta_1)} \bigr)
\\
&
 = s_{\ld}^* \cdot s_{\ld - N \ta_1}
   \cdot s_{\ld - N \ta_1}^* \cdot s_{\ld - N \ta_1 - N \ta_2}
 = v_{\zt_1 \zt_2} (e^{2 \pi i \ld}),
\end{split}
\]
proving the claim.

We have shown that $\zt \mapsto v_{\zt}$ is a cocycle for the
action $\zt \mapsto \bt_{\zt^N, N}$ of $S^1$ on $C (S^1, M_N)$.
Therefore the formula
\[
\rh_{\zt} (g) = v_{\zt} \bt_{\zt^N, N} (g) v_{\zt}^*
\]
defines an action of $S^1$ on $C (S^1, M_N)$ which is exterior
equivalent to $\zt \mapsto \bt_{\zt^N, N} = {\overline{\gm}}_{\zt, N}$.

To finish the proof of~(\ref{Item_824_R_T_ExtEq}),
we show that $\ps$ is equivariant
for the action $\rh$ on $C (S^1, M_N)$.
Let $f \in R$, let $\ld \in \R$, let $\zt \in S^1$, and
choose $\ta \in \R$ such that $\zt = e^{2 \pi i \ta}$.
Then
\[
\begin{split}
(\rh_{\zt} \circ \ps) (f) (e^{2 \pi i \ld})
& = s_{\ld}^* s_{\ld - N \ta} \ps (f) (\zt^{- N} e^{2 \pi i \ld})
                       s_{\ld - N \ta}^* s_{\ld}
\\
& = s_{\ld}^* s_{\ld - N \ta} \bigl[ s_{\ld - N \ta}^*
        f \bigl( e^{2 \pi i (\ld - N \ta) / N} \bigr)
         s_{\ld - N \ta} \bigr] s_{\ld - N \ta}^* s_{\ld}
\\
& = s_{\ld}^* f \bigl(\zt^{-1} e^{2 \pi i \ld / N} \bigr) s_{\ld}
  = (\ps \circ \gm_{\zt}) (f) (e^{2 \pi i \ld}).
\end{split}
\]
This completes the proof of~(\ref{Item_824_R_T_ExtEq}).

For~(\ref{Item_824_R_T_L}),
we first observe that if $f \in C (S^1)$ and $\ld \in \R$ then
\[
(\ps \circ \xi) (f) (e^{2 \pi i \ld})
 = s_{\ld}^*  {\operatorname{diag}} \bigl( f (e^{2 \pi i \ld / N}),
     \, f (\om^{- 1} e^{2 \pi i \ld / N}),
    \, \ldots, \, f (\om^{- N + 1} e^{2 \pi i \ld / N}) \bigr) s_{\ld}.
\]
Therefore
\[
L_{\ps \circ \xi} (e^{2 \pi i \ld})
 = \bigl\{ e^{2 \pi i \ld / N}, \, \om^{- 1} e^{2 \pi i \ld / N},
    \, \ldots, \, \om^{- N + 1} e^{2 \pi i \ld / N} \bigr\}.
\]
This is the same as the description in the statement.

For the second formula,
one checks that for $g \in C (S^1, M_N)$ and $\ld \in \R$, we have
\[
\ps^{-1} (g) (e^{2 \pi i \ld})
 = s_{N \ld} g ( e^{2 \pi i N \ld} ) s_{N \ld}^*
 = s_{N \ld} g \bigl( (e^{2 \pi i \ld})^N  \bigr) s_{N \ld}^*.
\]
Since $\io \colon R \to C (S^1, M_N)$ is just the inclusion,
this gives $L_{\io \circ \ps^{-1}} (z) = \{ z^N \}$ for $z \in S^1$,
as desired.
\end{proof}

\begin{lem}\label{L_1916_Simple}
The algebra $A$ of
Construction \ref{Cn_1824_S1An}(\ref{Item_1924_S1_A})
is a simple AT~algebra.
\end{lem}

\begin{proof}
Using Lemma~\ref{L_1824_R_T}, we can rewrite the direct system in
Construction \ref{Cn_1824_S1An}(\ref{Item_1824_S1_nu_A}) as
\[
A = \dirlim_n
 \bigl[ C (S^1, M_{N r_0 (n)}) \oplus C (S^1, M_{r_1 (n)}) \bigr],
\]
with maps ${\widetilde{\nu}}_{n}$, for $n \in \Nz$,
obtained analogously to
Construction \ref{Cn_1824_S1An}(\ref{Item_1824_S1_nu_A}) from
\[
{\widetilde{\nu}}_{n, 0, 0}
 = (\id_{C (S^1)})_{N r_0 (n)}^{l_{0, 0} (n)},
\qquad
{\widetilde{\nu}}_{n, 0, 1}
 = \ps_{l_{0, 1} (n) r_1 (n)} \circ \nu_{n, 0, 1},
\]
\[
{\widetilde{\nu}}_{n, 1, 0}
 = \nu_{n, 1, 0} \circ (\ps_{l_{1, 0} (n) r_0 (n)})^{-1},
\andeqn
{\widetilde{\nu}}_{n, 1, 1} = \nu_{n, 1, 1},
\]
and with
\[
\begin{split}
{\widetilde{\nu}}_{n, m}
& = {\widetilde{\nu}}_{n - 1}
    \circ {\widetilde{\nu}}_{n - 2}
    \circ \cdots \circ {\widetilde{\nu}}_m \colon
\\
&  C (S^1, M_{N r_0 (m)}) \oplus C (S^1, M_{r_1 (m)})
   \to C (S^1, M_{N r_0 (n)}) \oplus C (S^1, M_{r_1 (n)}).
\end{split}
\]

It is now obvious that $A$ is an AT~algebra.
For simplicity, we use Proposition 2.1 of~\cite{DNNP},
with $L_{\Ph}$ as defined in 1.3 of~\cite{DNNP}.
To make the notation easier, we take $X_n = S^1 \times \{ 0, 1 \}$
(rather than $S^1 \amalg S^1$ as in~\cite{DNNP}),
and for $j \in \{ 0, 1 \}$ identify $S^1 \times \{ j \}$
with the primitive ideal space of $C (S^1, M_{r_j (n)})$.
Moreover, since the spaces $X_n$ are all equal,
we write them all as~$X$.

For $z \in S^1$ it is immediate that
\[
L_{{\widetilde{\nu}}_{n, 0, 0}} (z) = \{ z \}
\]
and, recalling $\om = \exp (2 \pi i / N)$ from
Construction \ref{Cn_1824_S1An}(\ref{Item_1824_S1_s_om})
and using
\[
\{ 1, \om^{- 1}, \ldots, \om^{- N + 1} \}
 = \{ 1, \om, \ldots, \om^{N - 1} \},
\]
also
\[
\begin{split}
L_{{\widetilde{\nu}}_{n, 1, 1}} (z)
& = L_{\nu_{n, 1, 1}} (z)
\\
& = \bigl\{ z, \, \om z, \, \ldots, \om^{N - 1} z, \,
       e^{- 2 \pi i \te} z, \, e^{- 2 \pi i \te} \om z,
        \, \ldots, e^{- 2 \pi i \te} \om^{N - 1} z \bigr\}.
\end{split}
\]
Also, using Lemma \ref{L_1824_R_T}(\ref{Item_824_R_T_L}),
\[
L_{{\widetilde{\nu}}_{n, 0, 1}} (z)
 = L_{\ps \circ \xi} (z)
 = \bigl\{ y \in S^1 \colon y^N = z \bigr\}
\]
and
\[
L_{{\widetilde{\nu}}_{n, 1, 0}} (z)
= L_{\io \circ \ps^{-1}} (z) = \{ z^N \}.
\]
Putting these together, we get
\begin{equation}\label{Eq_1916_Lz_zero}
L_{{\widetilde{\nu}}_{n}} (z, 0)
 = \{ (z, 0) \}
  \cup \bigl\{ (y, 1) \colon {\mbox{$y \in S^1$ and $y^N = z$}} \bigr\}
\end{equation}
and
\begin{equation}\label{Eq_1916_Lz1}
\begin{split}
L_{{\widetilde{\nu}}_{n}} (z, 1)
& = \{ (z^N, 0), \,
     (z, 1), \, (\om z, 1), \, \ldots, (\om^{N - 1} z, 1),
\\
& \hspace*{5em} {\mbox{}}
     (e^{- 2 \pi i \te} z, 1), \, (e^{- 2 \pi i \te} \om z, 1),
        \, \ldots, (e^{- 2 \pi i \te} \om^{N - 1} z, 1) \bigr\}.
\end{split}
\end{equation}
One checks that if $C$, $D$, and $E$ are finite direct sums
of homogeneous \uca{s},
and $\Ph \colon C \to D$ and $\Ps \colon D \to E$ are unital \hm{s},
with primitive ideal spaces $X$, $Y$, and~$Z$,
then for $z \in Z$ we have
\begin{equation}\label{Eq_1916_Comp}
L_{\Ps \circ \Ph} (z) = \bigcup_{y \in L_{\Ps} (z)} L_{\Ph} (y).
\end{equation}
The equations (\ref{Eq_1916_Lz_zero}) and~(\ref{Eq_1916_Lz1})
show that $x \in L_{{\widetilde{\nu}}_{n}} (x)$ for any $x \in X$.
So~(\ref{Eq_1916_Comp}) implies that
for any $l, m, n \in \Nz$ with $n > m > l$,
and any $x \in X$,
\begin{equation}\label{Eq_1916_Contain}
L_{{\widetilde{\nu}}_{n, m}} (x) \cup L_{{\widetilde{\nu}}_{m, l}} (x)
 \S L_{{\widetilde{\nu}}_{n, l}} (x).
\end{equation}
It now suffices to prove that for every $l \in \Nz$ and every $\ep > 0$,
there is $n > l$ such that for every $z \in S^1$ and $j \in \{ 0, 1 \}$,
the set $L_{{\widetilde{\nu}}_{n, l}} (z, j)$ is
$\ep$-dense in $S^1 \times \{ 0, 1 \}$.
Given this, simplicity of $A$ can be deduced
from Proposition~2.1 of~\cite{DNNP}, and the proof will be complete.

Choose $n > l + 2$ such that
\[
\bigl\{ e^{- 2 \pi i k \te} \colon k = 0, 1, \ldots, n - l - 2 \bigr\}
\andeqn
\bigl\{ e^{- 2 \pi i k N \te} \colon k = 0, 1, \ldots, n - l - 2 \bigr\}
\]
are both $\ep$-dense in $S^1$.

We claim that if $z \in S^1$ is arbitrary, then
$L_{{\widetilde{\nu}}_{n - 1, \, l}} (z, 1)$
is $\ep$-dense in $S^1 \times \{ 0, 1 \}$.
To prove the claim,
first use (\ref{Eq_1916_Comp}) repeatedly, (\ref{Eq_1916_Contain}),
and the fact that
$(e^{- 2 \pi i \te} y, 1) \in L_{{\widetilde{\nu}}_{m}} (y, 1)$
for all $m \in \N$ and $y \in S^1$ (by (\ref{Eq_1916_Lz1}))
to see that for $z \in S^1$,
\begin{equation}\label{Eq_1927_Cont}
\bigl\{ (e^{- 2 \pi i k \te} z, \, 1) \colon
     k = 0, 1, \ldots, n - l - 2 \bigr\}
 \S L_{{\widetilde{\nu}}_{n - 1, \, l + 1}} (z, 1).
\end{equation}
By~(\ref{Eq_1916_Contain}),
this set is contained in $L_{{\widetilde{\nu}}_{n - 1, \, l}} (z, 1)$,
and, since since multiplication by~$z$ is isometric,
it is $\ep$-dense in $S^1 \times \{ 1 \}$.
Also, by (\ref{Eq_1927_Cont}) and (\ref{Eq_1916_Contain})
(taking $m = l + 1$), and using~(\ref{Eq_1916_Lz1})
to get $(y^N, 0) \in L_{{\widetilde{\nu}}_{l + 1, l}} (y, 1)$
for all  $y \in S^1$,
\[
\bigl\{ (e^{- 2 \pi i k N \te} z^N, \, 0) \colon
     k = 0, 1, \ldots, n - l - 2 \bigr\}
 \S L_{{\widetilde{\nu}}_{n - 1, \, l}} (z, 1),
\]
and this set is $\ep$-dense in $S^1 \times \{ 0 \}$.
The claim follows.

Now, for any $x \in S^1 \times \{ 0, 1 \}$,
the set $ L_{{\widetilde{\nu}}_{n, n - 1}} (x)$ contains at least
one point in $S^1 \times \{ 1 \}$
by (\ref{Eq_1916_Lz_zero}) and~(\ref{Eq_1916_Lz1}).
Using~(\ref{Eq_1916_Contain}) with $m = n - 1$
and the previous claim, it follows that
$L_{{\widetilde{\nu}}_{n, l}} (x)$
is $\ep$-dense in $S^1 \times \{ 0, 1 \}$, as desired.
\end{proof}

\begin{lem}\label{L_1912_B_Smp}
The algebra $B$ of
Construction \ref{Cn_1824_S1An}(\ref{Item_1824_S1_FixLim})
is a simple unital AF~algebra.
\end{lem}

\begin{proof}
That $B$ is a unital AF~algebra is immediate from its definition.
Simplicity follows from the corollary on page 212 of~\cite{Brtl}.
\end{proof}

\begin{lem}\label{L_1912_R_Fix}
Define
\[
c = \frac{1}{\sqrt{N}} \left( \begin{matrix}
  1    &   1         &   1              & \cdots    &   1          \\
  1    & \om         & \om^2            & \cdots    & \om^{N - 1} \\
  1    & \om^2       & \om^4            & \cdots    & \om^{2 (N - 1)} \\
\vdots & \vdots      & \vdots           & \ddots    &  \vdots      \\
  1    & \om^{N - 1} & \om^{2 (N - 1)}  & \cdots    & \om^{(N - 1)^2}
\end{matrix} \right).
\]
Then the formula
\[
\ep_0 (\ld_0, \ld_1, \ldots, \ld_{N - 1})
 = c^* \diag (\ld_0, \ld_1, \ldots, \ld_{N - 1}) c,
\]
for $\ld_0, \ld_1, \ldots, \ld_{N - 1} \in {\mathbb{C}}$,
defines an isomorphism from ${\mathbb{C}}^N$ to $R^{\gm}$.
\end{lem}

\begin{proof}
One checks that $c$ is unitary and
\begin{equation}\label{Eq_1912_wsw}
c s c^*
 = \diag \bigl(1 , \om^{-1}, \om^{-2}, \ldots, \om^{- (N - 1)} \bigr).
\end{equation}

It is immediate that $R^{\gm}$ is the set of constant functions
in $C (S^1, M_N)$ whose constant value commutes with $s$.
Let $D$ be the set of constant functions
in $C (S^1, M_N)$ whose constant value commutes with $c s c^*$.
Then $a \mapsto c^* a c$ is an isomorphism from $D$ to $R^{\gm}$.
Also, by~(\ref{Eq_1912_wsw}),
\[
(\ld_0, \ld_1, \ldots, \ld_{N - 1})
 \mapsto \diag ( \ld_0, \ld_1, \ldots, \ld_{N - 1} )
\]
is an isomorphism from ${\mathbb{C}}^N$ to~$D$.
\end{proof}

\begin{lem}\label{L_1916_FPisB}
There is a family $(\et_n)_{n \in \Nz}$ of isomorphisms
$\et_n \colon B_n \to (A_n)^{\af^{(n)}}$ such that the following hold.
\begin{enumerate}
\item\label{Item_1925_Comm}
$\et_n \circ \ch_{n, m} = \nu_{n, m} \circ \et_m$
whenever $m, n, \in \Nz$ satisfy $m \leq n$.
\item\label{Item_1925_pn}
With $p_n$ as in Construction \ref{Cn_1824_S1An}(\ref{Item_1824_S1_pn})
and $q_n$ as in Construction \ref{Cn_1824_S1An}(\ref{Item_1915_S1_qn}),
we have $\et_n (q_n) = p_n$.
\item\label{Item_1925_BBnj}
For all $n \in \Nz$ and $j \in \{ 0, 1 \}$, we have
$\et_n (B_{n, j}) = (A_{n, j})^{\af^{(n)}}$.
\item\label{Item_1925_Iso_A}
The family $(\et_n)_{n \in \Nz}$ induces an isomorphism
$\et_{\I} \colon B \to A^{\af}$.
\end{enumerate}
\end{lem}

We warn that the subscript in $\et_n$ does not have the meaning
taken from Notation~\ref{N_1908_Ampl}.

\begin{proof}[Proof of Lemma~\ref{L_1916_FPisB}]
Since $\af$ is a direct limit action,
the inclusions $(A_n)^{\af^{(n)}} \to A_n$ induce an isomorphism
$A^{\af} \to \dirlim_n (A_n)^{\af^{(n)}}$.
Therefore (\ref{Item_1925_Iso_A}) follows from the rest of the
statement of the lemma.

Lemma~\ref{L_1912_R_Fix} implies that
$(\ep_0)_{r_0 (n)}
 \colon M_{r_0 (n)} ({\mathbb{C}}^N) \to (A_{n, 0})^{\af^{(n)}}$
is an isomorphism.
It is immediate that the embedding
$\ep_1 \colon {\mathbb{C}} \to C (S^1)^{\bt}$ as constant functions
is an isomorphism.
Therefore
\[
\et_n^{(0)} = (\ep_0)_{r_0 (n)} \oplus (\ep_1)_{r_1 (n)} \colon
  B_{n, 0} \oplus B_{n, 1} \to (A_{n})^{\af^{(n)}}
\]
is an isomorphism.
Clearly (\ref{Item_1925_pn}) and~(\ref{Item_1925_BBnj})
hold with $\et_n^{(0)}$ in place of~$\et_n$.

Let $\io^{S^1}$ and $\xi^{S^1}$ be the restriction
and corestriction of $\io$ and $\xi$
to the corresponding fixed point algebras,
and similarly define $\nu_n^{S^1}$, $\nu_{n, j, k}^{S^1}$, etc.
Then the inverses of the maps $\et_n^{(0)}$ implement an isomorphism
from the direct system $\bigl( (A_n)^{\af^{(n)}} \bigr)_{n \in \Nz}$
to the direct system $( B_n )_{n \in \Nz}$, with
the maps $\sm_n \colon B_n \to B_{n + 1}$ taken to be
\[
\begin{split}
\sm_n (a_0, a_1)
& = \Bigl( \diag \bigl( (\id_{B_{n, 0}})^{l_{0, 0} (n)} (a_0),
\\
& \hspace*{6em} {\mbox{}}
   \, \bigl( ( (\ep_0)_{l_{0, 1} (n) r_1 (n)})^{-1}
       \circ (\xi^{S^1})_{r_1 (n)}^{l_{0, 1} (n)}
       \circ (\ep_1)_{r_1 (n)} \bigr) (a_1) \bigr),
\\
& \hspace*{2em} {\mbox{}}
 \diag \bigl(  \bigl( ( (\ep_1)_{l_{1, 0} (n) r_1 (n)} )^{-1}
             \circ (\io^{S^1})_{r_0 (n)}^{l_{1, 0} (n)}
             \circ (\ep_0)_{r_1 (n)} \bigr) (a_0),
\\
& \hspace*{6em} {\mbox{}}
        \, (\id_{B_{n, 1}})^{2 N l_{1, 1} (n)} (a_1) \bigr) \Bigr).
\end{split}
\]

The map $(\ep_0)^{-1} \circ \xi^{S^1} \circ \ep_1 \colon {\mathbb{C}} \to {\mathbb{C}}^N$
is a unital homomorphism.
There is only one such unital homomorphism, so this map is equal
to the map $\mu$
in Construction \ref{Cn_1824_S1An}(\ref{Item_1824_S1_FixPtSys}).
The map $(\ep_1)^{-1} \circ \io^{S^1} \circ \ep_0 \colon {\mathbb{C}}^N \to M_N$
is an injective unital homomorphism.
Therefore it must be unitarily equivalent to the map $\dt$
in Construction \ref{Cn_1824_S1An}(\ref{Item_1824_S1_FixPtSys}).
It now follows from the definitions
(see Construction \ref{Cn_1824_S1An}(\ref{Item_1824_S1_FixLim}))
that for $n \in \Nz$
there is a unitary $v_n \in B_{n + 1}$
such that $\ch_n (b) = v_n \sm_n (b) v_n^*$ for all $b \in B_n$.
Inductively define unitaries $w_n \in (A_n)^{\af^{(n)}}$
by $w_0 = 1$ and, given~$w_n$,
setting $w_{n + 1} = \nu_n^{S^1} (w_n) \et_{n + 1}^{(0)} (v_n)^*$.
Then define $\et_n (b) = w_n \et_{n}^{(0)} (b) w_n^*$ for $b \in B_n$.
The conditions (\ref{Item_1925_pn}) and~(\ref{Item_1925_BBnj})
hold as stated because they hold for the maps $\et_{n}^{(0)}$.
For $n \in \Nz$,
using $\nu_n^{S^1} \circ \et_{n}^{(0)} = \et_{n + 1}^{(0)} \circ \sm_n$,
one gets $\nu_n^{S^1} \circ \et_{n} = \et_{n + 1} \circ \ch_n$.
This implies~(\ref{Item_1925_Comm}).
\end{proof}

\begin{lem}\label{L_1917_HasTRP}
In Construction~\ref{Cn_1824_S1An}, assume that $r_0 (0) \leq r_1 (0)$
and that for all $n \in \Nz$
we have $l_{1, 0} (n) \geq l_{0, 0} (n)$
and $l_{1, 1} (n) \geq l_{0, 1} (n)$.
Further assume that
\[
\lim_{n \to \I} \frac{l_{0, 1} (n)}{l_{0, 0} (n)} = \I
\andeqn
\lim_{n \to \I} \frac{l_{1, 1} (n)}{l_{1, 0} (n)} = \I.
\]
Then the action $\af$
of Construction \ref{Cn_1824_S1An}(\ref{Item_1824_S1_nu_A})
has the tracial Rokhlin property with comparison.
\end{lem}

The hypotheses are overkill.
They are chosen to make the proof easy.

\begin{proof}[Proof of Lemma~\ref{L_1917_HasTRP}]
Since $A$ is stably finite,
by Lemma 1.15 in \cite{phill23}, we may disregard
condition~(\ref{Item_902_pxp_TRP})
in Definition~\ref{traR}.

We first claim that $0 < r_0 (n) \leq r_1 (n)$ for all $n \in \N$.
This is true for $n = 0$ by hypothesis.
For any other value of~$n$, using
$r_0 (n - 1) > 0$ and $r_1 (n - 1) > 0$, we have
\[
\begin{split}
r_0 (n)
& = l_{0, 0} (n - 1) r_0 (n - 1) + l_{0, 1} (n - 1) r_1 (n - 1)
\\
&
 \leq l_{1, 0} (n - 1) r_0 (n - 1) + l_{1, 1} (n - 1) r_1 (n - 1)
 = r_1 (n).
\end{split}
\]
Also, since $l_{0, 0} (n - 1) > 0$, the first step of this
calculation implies that $r_0 (n) > 0$.

Next, we claim that for every $n \in \Nz$
and every tracial state $\ta$ on the algebra $B_{n + 1}$
of Construction \ref{Cn_1824_S1An}(\ref{Item_1824_S1_Bn0Bn1}),
with $\ch_n$ as in
Construction \ref{Cn_1824_S1An}(\ref{Item_1824_S1_FixLim})
and $q_n$ as in
Construction \ref{Cn_1824_S1An}(\ref{Item_1915_S1_qn}),
we have
\begin{equation}\label{Eq_1916_TrBd}
\ta (1 - \ch_n (q_n))
 \leq \max \left( \frac{l_{0, 0} (n)}{l_{0, 1} (n)},
   \, \frac{l_{1, 0} (n)}{l_{1, 1} (n)} \right).
\end{equation}
To see this, we first look at the partial embedding multiplicities
in Construction \ref{Cn_1824_S1An}(\ref{Item_1824_S1_FixPtSys})
to see that the rank of $1 - \ch_n (q_n)$ in each of the first
$N$ summands (all equal to $M_{r_0 (n + 1)}$) is $l_{0, 0} (n) r_0 (n)$,
and the rank of $1 - \ch_n (q_n)$ in the last summand
(equal to $M_{r_1 (n + 1)}$) is $l_{1, 0} (n) r_0 (n)$.
Now
\begin{equation}\label{Eq_1916_ZeroCrd}
\frac{l_{0, 0} (n) r_0 (n)}{r_0 (n + 1)}
  = \frac{l_{0, 0} (n) r_0 (n)}{l_{0, 0} (n) r_0 (n)
         + l_{0, 1} (n) r_1 (n)}
  \leq \frac{l_{0, 0} (n) r_0 (n)}{l_{0, 1} (n) r_1 (n)}
  \leq \frac{l_{0, 0} (n)}{l_{0, 1} (n)},
\end{equation}
and
\begin{equation}\label{Eq_1916_1stCrd}
\frac{l_{1, 0} (n) r_0 (n)}{r_1 (n + 1)}
  = \frac{l_{1, 0} (n) r_0 (n)}{
   N l_{1, 0} (n) r_0 (n) + 2 N l_{1, 1} (n) r_1 (n)}
  \leq \frac{l_{1, 0} (n) r_0 (n)}{l_{1, 1} (n) r_1 (n)}
  \leq \frac{l_{1, 0} (n)}{l_{1, 1} (n)}.
\end{equation}
The number $\ta (1 - \ch_n (q_n))$ is a convex combination of
the numbers in (\ref{Eq_1916_ZeroCrd}) and~(\ref{Eq_1916_1stCrd}).
The claim follows.

Now let $F \subseteq A$ and $S \subseteq C (S^1)$ be finite sets,
let $\ep > 0$, let $x \in A_{+} \setminus \{ 0 \}$,
and let $y \in A_{+}^{\af} \setminus \{ 0 \}$.
We may assume that $\| x \| \leq 1$ and $\| y \| \leq 1$,
and that $\| f \| \leq 1$ for all $f \in S$.
Set
\begin{equation}\label{Eq_1917_ep0}
\ep_0 = \frac{1}{2}
 \min \left( \inf_{\ta \in T (A)} \ta (x), \,
     \inf_{\ta \in T (A^{\af})} \ta (y), \, \frac{1}{2} \right).
\end{equation}
Choose $n \in \Nz$ so large that there is a finite subset
$F_0 \subseteq A_n$
with $\dist (a, \, \nu_{\I, n} (F_0) ) < \frac{\ep}{3}$
for all $a \in F$,
and also so large that
\begin{equation}\label{Eq_1917_Ch_n}
\min \left( \frac{l_{0, 1} (n)}{l_{0, 0} (n)},
    \, \frac{l_{1, 1} (n)}{l_{1, 0} (n)} \right)
  > \frac{1}{\ep_0}.
\end{equation}

Let $p \in A$ be $p = \nu_{\I, n} (p_n)$.
Define
\[
\ph_0 \colon
 C (S^1) \to p_n A_n p_n
  = M_{r_1 (n)} \otimes C (S^1)
\]
by $\ph_0 (g) = (0, \, 1 \otimes g)$ for $g \in C (S^1)$.
Define $\ph = \nu_{\I, n} \circ \ph_0 \colon C (S^1) \to p A p$.
Then $\ph$ is an equivariant unital \hm.
In particular, $\ph$ is exactly multiplicative on~$S$.
Further, let $a \in F$ and $f \in S$.
Choose $b \in F_0$
such that $\| a - \nu_{\I, n} (b) \| < \frac{\ep}{3}$.
Then, using $\| f \| \leq 1$ and the fact that $\ph_0 (f)$ commutes
with all elements of~$A_n$, we have
\[
\| \ph (f) a - a \ph (f) \|
  \leq 2 \| a - \nu_{\I, n} (b) \| < \ep.
\]
Part~(\ref{Item_893_FS_equi_cen_multi_approx})
of the definition is verified.

For the remaining three conditions,
let $\ta$ be any \tst{} on either $A$ or $A^{\af}$.
Let $(\et_n)_{n \in \Nz}$ be as in Lemma~\ref{L_1916_FPisB}.
Then $\ta \circ \nu_{\I, n + 1} \circ \et_{n + 1}$
is a \tst{} on~$B_{n + 1}$.
Combining this with (\ref{Eq_1916_TrBd}), (\ref{Eq_1917_Ch_n}),
(\ref{Eq_1917_ep0}),
and $(\nu_{\I, n + 1} \circ \et_{n + 1} \circ \ch_n) (q_n) = p$, we get
$\ta (1 - p) < \ta (x) \leq d_{\ta} (x)$ for all $\ta \in T (A)$
and $\ta (1 - p) < \ta (y) \leq d_{\ta} (y)$
for all $\ta \in T (A^{\alpha})$.
Since simple unital AF~algebras and
simple unital AT~algebras have strict comparison,
it follows that $1 - p \precsim_A x$
and $1 - p \precsim_{A^{\alpha}} y$.
Since $\ep_0 < \frac{1}{2}$,
similar reasoning gives $1 - p \precsim_{A^{\alpha}} p$.
\end{proof}

We use equivariant K-theory to show that,
with suitable choices, $\af$ does not have
finite Rokhlin dimension with commuting towers.

\begin{rmk}\label{R_5503_TnsOI}
In the work from here through Theorem~\ref{T_1919_Ex},
the order on the K-theory plays no role.
For simplicity of notation, we write the lemmas and proofs
for $\af \colon S^1 \to \Aut (A)$,
but they apply equally well to the action
$\zt \mapsto \id_{\OI} \otimes \af_{\zt}$ of $S^1$ on $\OI \otimes A$.
\end{rmk}

Recall equivariant K-theory from Definition 2.8.1 of~\cite{Phl1}.
For a unital \ca~$A$ with an action $\af \colon G \to \Aut (A)$
of a compact group~$G$, it is the Grothendieck group of the
equivariant isomorphism classes of equivariant finitely generated
projective right modules $E$ over~$A$,
with ``equivariant'' meaning that the module is equipped
with an action of $G$
such that $g \cdot (\xi a) = (g \cdot \xi) \af_g (a)$
for $g \in G$, $\xi \in E$, and $a \in A$.
Further recall the representation ring $R (G)$ of a compact group
from the introduction of~\cite{Sgl1} or Definition 2.1.3 of~\cite{Phl1}
(it is $K_0^G ({\mathbb{C}})$, or the Grothendieck group of the
isomorphism classes of \fd{} representations of~$G$),
its augmentation ideal $I (G)$ from the example before
Proposition~3.8 of~\cite{Sgl1} (where it is called $I_G$)
or the discussion after Definition 2.1.3 of~\cite{Phl1}
(it is the kernel of the dimension map from $R (G)$ to $\Z$),
and, for a \ca{} $A$ with an action $\af \colon G \to \Aut (A)$,
the $R (G)$-module structure on $K_*^G (A)$
from Theorem 2.8.3 and Definition 2.7.8 of~\cite{Phl1}.
In particular, for $G = S^1$,
if we let $\sm \in {\widehat{S^1}}$ be the identity map $S^1 \to S^1$,
then $R (S^1) = \Z [ \sm, \sm^{-1}]$,
the Laurent polynomial ring in one variable over~$\Z$.
(See Example~(ii) at the beginning of Section~3 of~\cite{Sgl1}.)
Moreover, $I (S^1)$ is the ideal generated by $\sm - 1$.

Recall the action $\bt \colon S^1 \to \Aut (C (S^1))$
from Construction \ref{Cn_1824_S1An}(\ref{Item_1824_S1_CS1}),
given by $\bt_{\zt} (f) (z) = f (\zt^{-1} z)$ for $\zt, z \in S^1$,
and ${\overline{\gm}} \colon S^1 \to \Aut (C (S^1))$
from Lemma~\ref{L_1824_R_T},
given by ${\overline{\gm}}_{\zt} (f) (z) = f (\zt^{-N} z)$.
We denote the equivariant K-theory for these actions by
$K_*^{S^1, \bt} (C (S^1))$ and $K_*^{S^1, {\overline{\gm}}} (C (S^1))$,
and similarly for other actions when ambiguity is possible.

We won't actually use the following computation of the
equivariant $K_{1}$-groups, but it is included to give a more
complete description of our example.
As in Section~\ref{Sec_3749_Exam_TRPZ2},
we abbreviate $\Z / n \Z$ to $\Z_n$.

\begin{lem}\label{L_1918_K1}
We have $K_1^{S^1, \bt} (C (S^1)) = 0$ and
(with $R \S C (S^1, M_N)$ as in
Construction \ref{Cn_1824_S1An}(\ref{Item_1824_S1_R}))
$K_1^{S^1} (R) = 0$.
\end{lem}

\begin{proof}
By Theorem 2.8.3(7) of~\cite{Phl1}, we have
\[
K_{1}^{S^{1}, \bt} (C (S^{1}))
 \cong K_{1} \bigr( C^{*} (S^{1}, \, C (S^{1}), \, \bt) \bigr).
\]
Since
\[
C^{*} (S^{1}, \, C (S^{1}), \, \bt) \cong K (L^{2} (S^{1})),
\]
we conclude $K_{1}^{S^{1}, \bt} (C (S^{1})) = 0$.

For $K_{1}^{S^{1}} (R)$, since exterior equivalent actions
of a compact group~$G$ give isomorphic $R (G)$-modules $K_*^G (A)$
(Theorem 2.8.3(5) of~\cite{Phl1}),
by Lemma \ref{L_1824_R_T}(\ref{Item_824_R_T_ExtEq})
it is sufficient to prove this
for the action $\zt \mapsto {\overline{\gm}}_{\zt, N}$.
By stability of equivariant K-theory (Theorem 2.8.3(4) of~\cite{Phl1}),
it suffices to prove this
for the action ${\overline{\gm}}$ of $S^{1}$ on $C (S^{1})$.
By~\cite{Jlg} (or Theorem 2.8.3(7) of~\cite{Phl1}), we have
$K_{1}^{S^{1}, {\overline{\gm}}} (C (S^{1}))
  \cong K_{1}
     \bigl( C^{*} (S^{1}, \, C (S^{1}), \, {\overline{\gm}}) \bigr)$.
Corollary 2.10 of~\cite{Grn}, with $G = S^{1}$ and $H = \Z_N$,
tells us that
\[
C^{*} (S^{1}, \, C (S^{1}), \, {\overline{\gm}})
 \cong K (L^{2} (S^{1})) \otimes C^* (\Z_N),
\]
which has trivial $K_{1}$-group.
\end{proof}

\begin{lem}\label{L_1918_EqKbt}
There is an $R (S^1)$-module isomorphism
$K_0^{S^1, \bt} (C (S^1)) \cong \Z$,
with the $R (S^1)$-module structure coming from the
isomorphism $\Z \cong R (S^1) / I (S^1)$,
and such that the class in $K_0^{S^1, \bt} (C (S^1))$ of the rank one
free module is sent to $1 \in \Z$.
\end{lem}

\begin{proof}
In Proposition 2.9.4 of~\cite{Phl1}, take $A = {\mathbb{C}}$, $G = S^1$,
and $H = \{ 1 \}$,
and refer to the description of the map
in the proof of that proposition.
\end{proof}

\begin{lem}\label{L_1918_EqKgmt}
There is an $R (S^1)$-module isomorphism
$K_0^{S^1, {\overline{\gm}}} (C (S^1)) \cong R ( \Z_N)$,
with the $R (S^1)$-module structure coming from the
surjective restriction map $R (S^1) \to R ( \Z_N)$.
Moreover, the classes in $K_0^{S^1, {\overline{\gm}}} (C (S^1))$
of the equivariant finitely generated projective right $C (S^1)$-modules
with underlying nonequivariant module $C (S^1)$
correspond exactly to the elements of
$(\Z_N)^{\wedge} \S R ( \Z_N)$.
\end{lem}

\begin{proof}
In Proposition 2.9.4 of~\cite{Phl1}, take
\[
A = {\mathbb{C}},
\qquad
G = S^1,
\andeqn
H = \{ 1, \om, \ldots, \om^{N - 1} \} \cong \Z_N.
\]
With these choices, $C (G \times_H {\mathbb{C}})$
is the set of $\om$-periodic
functions on $S^1$,
with the action of $\zt \in S^1$ being rotation by~$\zt$.
With ${\overline{\gm}} \colon S^1 \to \Aut (C (S^1))$
as in Lemma~\ref{L_1824_R_T}, this algebra is equivariantly isomorphic
to $(C (S^1), {\overline{\gm}})$ in an obvious way.
{}From Proposition 2.9.4 of~\cite{Phl1}, we get
$K_0^{S^1, {\overline{\gm}}} (C (S^1)) \cong R ( \Z_N)$.
Using the description of the map in the proof of the proposition,
the map sends the class of a ${\overline{\gm}}$-equivariant
finitely generated projective right $C (S^1)$-module~$E$
to the class, as a representation space of~$H$,
of its pushforward under the evaluation map $f \mapsto f (1)$.
If $E$ is nonequivariantly isomorphic to $C (S^1)$,
this pushforward is nonequivariantly isomorphic to~${\mathbb{C}}$.
The only classes in $R (\Z_N)$
with underlying vector space~${\mathbb{C}}$
are those in $(\Z_N)^{\wedge}$.
To check that an element $\ta \in (\Z_N)^{\wedge}$ actually
arises this way, choose $l \in \Z$ such that $\ta (\om) = \om^l$.
Then use the action of $S^1$ on $C (S^1)$ given by
$(\zt \cdot f) (z) = \zt^l f (\zt^{- N} z)$ for $\zt, z \in S^1$
and $f \in C (S^1)$.
One readily checks that this makes $C (S^1)$ a
${\overline{\gm}}$-equivariant right $C (S^1)$-module
whose restriction to $\{ 1 \}$ is ${\mathbb{C}}$
with the representation~$\ta$.
\end{proof}

\begin{lem}\label{L_1918_EqKThR}
There is an $R (S^1)$-module isomorphism
$K_0^{S^1} (R) \cong R ( \Z_N)$,
with the $R (S^1)$-module structure coming from the
surjective restriction map $R (S^1) \to R ( \Z_N)$.
Moreover,
for any given rank one invariant \pj{} $e \in R \S C (S^1, M_N)$,
the isomorphism can be chosen to send $[e]$ to $1 \in R ( \Z_N)$.
\end{lem}

\begin{proof}
Using Lemma \ref{L_1824_R_T}(\ref{Item_824_R_T_Rk1}),
it suffices to prove this for $C (S^1, M_N)$ and the action
$\zt \mapsto \rh_{\zt} = \ps \circ \gm_{\zt} \circ \ps^{-1}$ in
Lemma \ref{L_1824_R_T}(\ref{Item_824_R_T_ExtEq})
in place of $R$ and~$\gm$.

So fix a rank one \pj{} $e \in C (S^1, M_N)$ which is invariant
under $\zt \mapsto \ps \circ \gm_{\zt} \circ \ps^{-1}$.
If the group action is ignored, $e$ is \mvnt{} a constant \pj,
so $E = e C (S^1, M_N)$ is nonequivariantly isomorphic to
$C (S^1, {\mathbb{C}}^N)$ as a right $C (S^1, M_N)$-module.
Let ${\overline{\gm}} \colon S^1 \to \Aut (C (S^1))$
be as in Lemma~\ref{L_1824_R_T},
and write ${\overline{\gm}}_N$ for the action
$\zt \mapsto {\overline{\gm}}_{\zt, N}$ on $C (S^1, M_N)$.
By Lemma \ref{L_1824_R_T}(\ref{Item_824_R_T_ExtEq}),
$\rh$ is exterior equivalent to ${\overline{\gm}}_{N}$.
By Proposition 2.7.4 of~\cite{Phl1},
and the formula in the proof for the isomorphism,
there is a (natural) isomorphism from $K_0^{S^1, \rh} (C (S^1))$
to $K_0^{S^1, {\overline{\gm}}_{N}} (C (S^1, M_N))$
which sends the class of $E$
to the class of the same module with a different action of~$G$.
By stability in equivariant K-theory
(Theorem 2.8.3(4) of~\cite{Phl1}), there is an isomorphism
\[
K_0^{S^1, {\overline{\gm}}_{N}} (C (S^1, M_N))
 \to K_0^{S^1, {\overline{\gm}}} (C (S^1))
\]
which maps the class of $E$ to the class of some equivariant module
whose underlying nonequivariant module is $C (S^1)$.
Combining this with Lemma~\ref{L_1918_EqKgmt},
we have an isomorphism from $K_0^{S^1} (R)$ to $R ( \Z_N)$
which sends $[e]$ to some element $\ta \in (\Z_N)^{\wedge}$.
Multiplying by $\ta^{-1}$ gives an isomorphism
from $K_0^{S^1} (R)$ to $R ( \Z_N)$
which sends $[e]$ to $1 \in (\Z_N)^{\wedge}$.
\end{proof}

\begin{lem}\label{L_1918_EqKSys}
Identify $R (S^1) = \Z [ \sm, \sm^{-1}]$
as before Lemma~\ref{L_1918_K1}.
Let $\io$ and $\xi$
be as in Construction \ref{Cn_1824_S1An}(\ref{Item_1824_S1_Crs}).
There are isomorphisms of $R (S^1)$-modules
$K_0^{S^1, \bt} (C (S^1)) \cong \Z$,
via the surjective ring \hm{} which sends
$\sm \in R (S^1)$ to $1 \in \Z$, and
$K_0^{S^1} (R) \cong \Z [ \ov{\sm} ] / \langle \ov{\sm}^N - 1 \rangle$,
via the surjective ring \hm{} which sends
$\sm \in R (S^1)$ to $\ov{\sm} \in \Z$.
In terms of these isomorphisms, the map
$\io_* \colon K_0^{S^1} (R) \to K_0^{S^1, \bt} (C (S^1))$ becomes
the map $\Z [ \ov{\sm} ] / \langle \ov{\sm}^N - 1 \rangle \to \Z$
determined by $\io_* (1) = \io_* (\ov{\sm}) = 1$,
and the map $\xi_* \colon K_0^{S^1, \bt} (C (S^1)) \to K_0^{S^1} (R)$
becomes
the map $\Z \to \Z [ \ov{\sm} ] / \langle \ov{\sm}^N - 1 \rangle$
determined by $\xi_* (1) = 1 + \ov{\sm} + \cdots + \ov{\sm}^{N - 1}$.
\end{lem}

\begin{proof}
Recall that $\sm \in {\widehat{S^1}}$ is the identity map $S^1 \to S^1$.
The map $R (S^1) \to R ( \Z_N)$ is well known to be surjective,
and the image $\ov{\sm}$ of $\sm$ in $R ( \Z_N)$
satisfies $\ov{\sm}^N = 1$ but no lower degree polynomial relations,
so
$R (\Z_N) \cong \Z [ \ov{\sm} ] / \langle \ov{\sm}^N - 1 \rangle$.
Now the isomorphism $K_0^{S^1, \bt} (C (S^1)) \cong \Z$ is
Lemma~\ref{L_1918_EqKbt}
and the isomorphism
\begin{equation}\label{Eq_1918_KGR}
K_0^{S^1} (R)
 \cong \Z [ \ov{\sm} ] / \langle \ov{\sm}^N - 1 \rangle
\end{equation}
is Lemma~\ref{L_1918_EqKThR}.
Fix a rank one invariant \pj{} $e \in R \S C (S^1, M_N)$,
gotten from Lemma~\ref{L_1912_R_Fix}.

By Lemma~\ref{L_1918_EqKThR},
the isomorphism~(\ref{Eq_1918_KGR}) can be chosen to send
the class $[e R]$ of the right module $e R$ to~$1$.
We have $\io_* ([e R]) = [e C (S^1, M_N)]$, the class of some
rank one free module, but, by Lemma~\ref{L_1918_EqKbt}, only
one element of $K_0^{S^1, \bt} (C (S^1))$, namely $1 \in \Z$,
comes from a rank one free module.
So $\io_* (1) = 1$.
Since $\ov{\sm}$
is the class of $e R$ with a different action of~$S^1$,
we get $\io_* (\ov{\sm}) = 1$ for the same reason.

By Lemma~\ref{L_1912_R_Fix}, $\xi (1)$ is a sum of $N$ rank one
$\gm$-invariant \pj{s} in~$R$.
It follows from Lemma~\ref{L_1918_EqKThR} that, under the
isomorphism~(\ref{Eq_1918_KGR}),
each corresponds to some element
of $(\Z_N)^{\wedge} \S R (\Z_N)$,
that is, to some power $\ov{\sm}^k$ with $0 \leq k \leq N - 1$.
So there are $m_0, m_1, \ldots, m_{N - 1} \in \{ 0, 1, \ldots, N - 1 \}$
such that $\xi_* (1) = \sum_{j = 0}^{N - 1} \ov{\sm}^{m_j}$.
Since $\sm \cdot 1 = 1$ in $K_0^{S^1, \bt} (C (S^1)) \cong \Z$,
it follows that $\sm \cdot \xi_* (1) = \xi_* (1)$.
The only possibility is then
$\xi_* (1) = 1 + \ov{\sm} + \cdots + \ov{\sm}^{N - 1}$.
\end{proof}

\begin{lem}\label{L_1918_MapRN}
In Construction \ref{Cn_1824_S1An},
assume that for all $n \in \Nz$ we have
$2 l_{1, 1} (n) l_{0, 0} (n) \neq l_{1, 0} (n) l_{0, 1} (n)$.
Let $A$ and $\af \colon S^1 \to \Aut (A)$ be as in
Construction \ref{Cn_1824_S1An}(\ref{Item_1924_S1_A}).
Then, following the notation of
Construction \ref{Cn_1824_S1An}(\ref{Item_1824_S1_nu_A})
and Construction \ref{Cn_1824_S1An}(\ref{Item_1924_S1_A}),
for every $n \in \N$:
\begin{enumerate}
\item\label{I_5328_MapRN_K}
The maps
\[
(\nu_n)_* \colon K_0^{S^1} (A_n) \to K_0^{S^1} (A_{n + 1})
\andeqn
(\nu_{n, \I})_* \colon K_0^{S^1} (A_n) \to K_0^{S^1} (A)
\]
are injective.
\item\label{I_5328_MapRN_ModI}
The induced maps
\[
K_0^{S^1} (A_n) / I (S^1) K_0^{S^1} (A_n)
  \to K_0^{S^1} (A_{n + 1}) / I (S^1) K_0^{S^1} (A_{n + 1})
\]
and
\[
K_0^{S^1} (A_n) / I (S^1) K_0^{S^1} (A_n)
  \to K_0^{S^1} (A) / I (S^1) K_0^{S^1} (A)
\]
are injective.
\end{enumerate}
\end{lem}

The injectivity conclusion in~(\ref{I_5328_MapRN_K}) can fail if
$2 l_{1, 1} (n) l_{0, 0} (n) = l_{1, 0} (n) l_{0, 1} (n)$,
but it seems likely that at least some of the consequences
we derive from this lemma still hold.
We don't know the details of what happens in this case.

Part~(\ref{I_5328_MapRN_ModI})
will only be needed in Section~\ref{Sec_2114_OI}.

\begin{proof}[Proof of Lemma~\ref{L_1918_MapRN}]
For both parts,
since equivariant K-theory commutes with direct limits
(Theorem 2.8.3(6) of~\cite{Phl1}),
it is enough to prove injectivity of
\[
(\nu_n)_* \colon K_0^{S^1} (A_n) \to K_0^{S^1} (A_{n + 1})
\]
and
\[
K_0^{S^1} (A_n) / I (S^1) K_0^{S^1} (A_n)
  \to K_0^{S^1} (A_{n + 1}) / I (S^1) K_0^{S^1} (A_{n + 1}).
\]

Identify $K_0^{S^1, \bt} (C (S^1))$ and $K_0^{S^1} (R)$,
as well as $\io_*$ and $\xi_*$, as in Lemma~\ref{L_1918_EqKSys}.
Also, observe that, with these identifications,
for all $t \in \R$, using the notation of
Construction \ref{Cn_1824_S1An}(\ref{Item_1824_S1_CS1}), the map
$( {\widetilde{\bt}}_t )_* \colon
  K_0^{S^1} (C (S^1)) \to K_0^{S^1} (C (S^1))$
becomes $\id_{\Z}$.
Therefore injectivity of $(\nu_n)_*$
is the same as injectivity of the map
\[
\Ph \colon \Z [ \ov{\sm} ] / \langle \ov{\sm}^N - 1 \rangle \oplus \Z
       \to \Z [ \ov{\sm} ] / \langle \ov{\sm}^N - 1 \rangle \oplus \Z
\]
which for $m, m_0, m_1, \ldots, m_{N - 1} \in \Z$ is given by
\begin{equation}\label{Eq_1918_Ph}
\begin{split}
& \Ph \left( \sum_{k = 0}^{N - 1} m_k \ov{\sm}^{k}, \, m \right)
\\
& \hspace*{2em} {\mbox{}}
 = \left( \sum_{k = 0}^{N - 1}
          \bigl( l_{0, 0} (n) m_k + l_{0, 1} (n) m \bigr) \ov{\sm}^k, \,
    l_{1, 0} (n) \sum_{k = 0}^{N - 1} m_k + 2 N l_{1, 1} (n) m \right).
\end{split}
\end{equation}

Suppose the right hand side of (\ref{Eq_1918_Ph}) is zero.
For $k = 0, 1, \ldots, N - 1$ we then have
$l_{0, 0} (n) m_k + l_{0, 1} (n) m = 0$.
We have $l_{0, 0} (n) \neq 0$ by the choices at the
beginning of Construction~\ref{Cn_1824_S1An}.
Therefore
\[
m_0 = m_1 = \cdots = m_{N - 1} = - \frac{l_{0, 1} (n) m}{l_{0, 0} (n)}.
\]
Putting this in the second coordinate gives
\begin{equation}\label{Eq_5328_Rpt}
0 = 2 N l_{1, 1} (n) m
   - \frac{l_{1, 0} (n) l_{0, 1} (n) m N}{l_{0, 0} (n)}.
\end{equation}
Since $N \neq 0$, this says
$2 l_{1, 1} (n) l_{0, 0} (n) = l_{1, 0} (n) l_{0, 1} (n)$
or $m = 0$.
The hypotheses rule out the first, so $m = 0$,
whence also $m_k = 0$ for $k = 0, 1, \ldots, N - 1$.
Thus $(\nu_n)_*$ is injective.

We have $K_0^{S^1} (R) / I (S^1) K_0^{S^1} (R) \cong \Z$,
with, in the identifications above,
$\Z [ \ov{\sm} ] / \langle \ov{\sm}^N - 1 \rangle \to \Z$
being
$\sum_{k = 0}^{N - 1} m_k \ov{\sm}^{k} \mapsto \sum_{k = 0}^{N - 1} m_k$.
Therefore injectivity of
\[
K_0^{S^1} (A_n) / I (S^1) K_0^{S^1} (A_n)
  \to K_0^{S^1} (A_{n + 1}) / I (S^1) K_0^{S^1} (A_{n + 1})
\]
is equivalent to, in the notation used in~(\ref{Eq_1918_Ph}),
\begin{equation}\label{Eq_5328_SumZ}
\sum_{k = 0}^{N - 1}
          \bigl( l_{0, 0} (n) m_k + l_{0, 1} (n) m \bigr) = 0
\andeqn
l_{1, 0} (n) \sum_{k = 0}^{N - 1} m_k + 2 N l_{1, 1} (n) m = 0
\end{equation}
implying
\[
\sum_{k = 0}^{N - 1} m_k = 0
\andeqn
m = 0.
\]
So assume~(\ref{Eq_5328_SumZ}).
The first part implies
\[
\sum_{k = 0}^{N - 1} m_k
 = - \frac{N l_{0, 1} (n) m}{l_{0, 0} (n)}.
\]
Substituting this in the second part, we get~(\ref{Eq_5328_Rpt}) again,
so, as above, $m = 0$.
The first part of~(\ref{Eq_5328_SumZ}),
combined with $l_{0, 0} (n) \neq 0$,
now implies $\sum_{k = 0}^{N - 1} m_k = 0$, as desired.
\end{proof}

\begin{cor}\label{C_1918_NoFinRD}
Under the hypotheses of Lemma~\ref{L_1918_MapRN},
the action $\af$ does not have
finite Rokhlin dimension with commuting towers.
\end{cor}

\begin{proof}
Following the notation
of Construction \ref{Cn_1824_S1An}(\ref{Item_1824_S1_An}),
recall that $A_0 = A_{0, 0} \oplus A_{0, 1}$
with $A_{0, 0} =  M_{r_0 (0)} (R)$.
Obviously the inclusion $A_{0, 0} \to A_0$
is injective on equivariant K-theory.
So Lemma~\ref{L_1918_EqKSys} implies that there is a submodule
of $K_0^{S^1} (A)$ which is isomorphic to $K_0^{S^1} (R)$.

Suppose $\af$ has finite Rokhlin dimension with commuting towers.
By Corollary 4.5 of~\cite{Gar_rokhlin_2017},
there is $n \in \N$ such that $I (S^1)^n K_0^{S^1} (A) = 0$.
Lemma \ref{L_1918_MapRN}(\ref{I_5328_MapRN_K})
implies that $I (S^1)^n R (\Z_N) = 0$.
Lemma~\ref{L_1918_EqKgmt} then says that
\[
I (S^1)^n K_0^{S^1, {\overline{\gm}}} (C (S^1)) = 0.
\]
Since the underlying action of $S^1$ on $S^1$ is not free,
this contradicts Theorem 1.1.1 of~\cite{Phl1}.
\end{proof}

It is also not hard to prove directly that
$I (S^1)^n R (\Z_N) \neq 0$ for all $n \in \N$.
With the notation of Lemma~\ref{L_1918_EqKSys},
there is a \hm{} $h \colon R ( \Z_N) \to {\mathbb{C}}$
such that $h ( \ov{\sm} ) = \exp (2 \pi i / N)$.
Since $\sm - 1 \in I (S^1)$ and the map $R (S^1) \to R ( \Z_N)$
sends $\sm$ to $\ov{\sm}$, it is enough to show that
$h ( (\ov{\sm} - 1)^n) \neq 0$ for all $n \in \N$.
But ${\mathbb{C}}$ is a field and $h (\ov{\sm} - 1 ) \neq 0$.

\begin{thm}\label{T_1919_Ex}
In Construction~\ref{Cn_1824_S1An}, assume the following:
\begin{enumerate}
\item\label{I_5328_l0011_r0}
$r_0 (0) \leq r_1 (0)$.
\item\label{I_5328_l0011_ineq}
$l_{1, 0} (n) \geq l_{0, 0} (n)$
and $l_{1, 1} (n) \geq l_{0, 1} (n)$ for all $n \in \Nz$.
\item\label{I_5328_l0011_limsI}
$\lim_{n \to \I} \frac{l_{0, 1} (n)}{l_{0, 0} (n)} = \I$
and $\lim_{n \to \I} \frac{l_{1, 1} (n)}{l_{1, 0} (n)} = \I$.
\item\label{I_5328_l0011_nq}
$2 l_{1, 1} (n) l_{0, 0} (n) \neq l_{1, 0} (n) l_{0, 1} (n)$
for all $n \in \Nz$.
\end{enumerate}
Then $A$ is a simple unital AT~algebra and the action $\af$
of Construction \ref{Cn_1824_S1An}(\ref{Item_1824_S1_nu_A})
has the tracial Rokhlin property with comparison.
However, neither $\af$
nor the action
$\zt \mapsto \id_{\OI} \otimes \af_{\zt}$ of $S^1$ on $\OI \otimes A$
has finite Rokhlin dimension with commuting towers.
\end{thm}

\begin{proof}
That $A$ is a simple unital AT~algebra is Lemma~\ref{L_1916_Simple}.
Assuming (\ref{I_5328_l0011_r0}), (\ref{I_5328_l0011_ineq}),
(\ref{I_5328_l0011_limsI}), and~(\ref{I_5328_l0011_nq}),
both Lemma~\ref{L_1917_HasTRP}
and Corollary~\ref{C_1918_NoFinRD} apply to~$\af$,
and Corollary~\ref{C_1918_NoFinRD} applies to
$\zt \mapsto \id_{\OI} \otimes \af_{\zt}$ by Remark~\ref{R_5503_TnsOI}.
\end{proof}

\begin{exa}\label{E_5328_Ex}
In Construction~\ref{Cn_1824_S1An}, choose $r_0 (0) = r_1 (0) = 1$
and for $n \in \Nz$ set
\[
l_{0, 0} (n) = l_{1, 0} (n) = 1
\andeqn
l_{0, 1} (n) = l_{1, 1} (n) = 2 n + 1.
\]
These choices satisfy the conditions in Theorem~\ref{T_1919_Ex}.
So $A$ is a simple unital AT~algebra, the action $\af$
of Construction \ref{Cn_1824_S1An}(\ref{Item_1824_S1_nu_A})
has the tracial Rokhlin property with comparison,
but $\af$ does not have finite Rokhlin dimension with commuting towers.
\end{exa}

We now show that different choices of~$N$
give actions which are not conjugate,
even if the underlying C*-algebras are isomorphic.
For this purpose, we introduce a primitive numerical invariant
of $R (S^1)$-modules.
It is intended only for use in this paper.

\begin{dfn}\label{D_5328_ao}
Let $E$ be a $\Z [ \sm, \sm^{-1} ]$-module,
and let $x \in E \setminus \{ 0 \}$.
We define $\mathrm{ao} (x)$ (the ``annihilator order of~$x$'') by
\[
\mathrm{ao} (x)
 = \inf \bigl( \bigl\{ m \in \N \colon (\sm^m - 1) x = 0 \bigr\} \bigr).
\]
We further define
\[
\mathrm{mao} (E)
 = \sup \bigl(
   \bigl\{ \mathrm{ao} (x) \colon x \in E \setminus \{ 0 \} \bigr\} \bigr).
\]
\end{dfn}

Even in our very limited situation, the set
$\bigl\{ \mathrm{ao} (x) \colon x \in E \setminus \{ 0 \} \bigr\}$ is
probably more interesting than its maximum.

The next lemma gives some basic properties.

\begin{lem}\label{L_5328_mao}
The assignment $E \mapsto \mathrm{mao} (E)$ on $\Z [ \sm, \sm^{-1} ]$-modules
has the following properties.
\begin{enumerate}
\item\label{I_5328_mao_Sb}
Let $E$ be a $\Z [ \sm, \sm^{-1} ]$-module and let
$F \S E$ be a submodule.
Let $x \in F$.
Then $\mathrm{ao} (x)$ is the same whether calculated with respect to~$E$ or to~$F$.
\item\label{I_5328_mao_Sum}
$\mathrm{mao} (E_1 \oplus E_2) =
\max \bigl( \mathrm{mao} (E_1), \, \mathrm{mao} (E_2) \bigr)$.
\item\label{I_5328_mao_DLim}
If $(E_n)_{n \in \Nz}$ is a direct system of $\Z [ \sm, \sm^{-1} ]$-moduless
with injective maps,
and we set $E = \dirlim_{n} E_n$,
then $\mathrm{mao} (E) = \sup_{n \in \Nz} \mathrm{mao} (E_n)$.
\item\label{I_5328_mao_RZN}
Let $N \in \N$ and make $R (\Z_N)$ into a module over
$\Z [ \sm, \sm^{-1} ] = R (S^1)$
via the surjective restriction map $R (S^1) \to R (\Z_N)$.
Then, with $\ov{\sm}$ being the image of $\sm$ in $R (\Z_N)$,
we have $R (\Z_N) \cong \Z [ \ov{\sm} ] / \langle \ov{\sm}^N - 1 \rangle$
and $\mathrm{mao} (R (\Z_N)) = N$.
\end{enumerate}
\end{lem}

One can check that
\[
\bigl\{ \mathrm{ao} (x) \colon x \in R (\Z_N) \setminus \{ 0 \} \bigr\}
 = \bigl\{ m \in \N \colon m | N \bigr\},
\]
but we do not need this.

\begin{proof}[Proof of Lemma~\ref{L_5328_mao}]
Part~(\ref{I_5328_mao_Sb}) is trivial.
Part~(\ref{I_5328_mao_Sum}) follows from
the obvious fact that if $x_1 \in E_1$ and $x_2 \in E_2$, then
$\mathrm{ao} (x_1, x_2)
 = \max \bigl( \mathrm{ao} (x_1), \, \mathrm{ao} (x_2) \bigr)$.
For part~(\ref{I_5328_mao_DLim}),
injectivity of the maps allows us to identify $E_n$ as a submodule of~$E$
for every $n \in \Nz$.
Then $E = \bigcup_{n = 0}^{\I} E_n$,
and the result follows from part~(\ref{I_5328_mao_Sb}).

We prove~(\ref{I_5328_mao_RZN}).
The identification of $R (\Z_N)$ as a $\Z [ \sm, \sm^{-1} ]$-module
is in the proof of Lemma~\ref{L_1918_EqKSys}.
From this identification,
it is immediate that $\mathrm{mao} (R (\Z_N)) \leq N$.
For the reverse inequality, let
$h \colon \Z [ \ov{\sm} ] / \langle \ov{\sm}^N - 1 \rangle \to \mathbb{C}$
be the \hm{} which sends $\ov{\sm}$ to $\exp (2 \pi i / N)$.
If $k \in \{ 1, 2, \ldots, N - 1 \}$, then $h (\ov{\sm}^k) \neq 0$.
This shows that $\mathrm{ao} (1_{R (\Z_N)}) \geq N$.
\end{proof}

\begin{lem}\label{L_5328_maoK0A}
Assume the hypotheses of Lemma~\ref{L_1918_MapRN}.
Then $\mathrm{mao} (K_0^{S^1} (A)) = N$.
\end{lem}

\begin{proof}
As in the proof of Lemma~\ref{L_1918_MapRN}, for every $n \in \Nz$
we have $K_0^{S^1} (A_n) \cong R (\Z_N) \oplus \Z$,
with the $\Z [ \sm, \sm^{-1} ]$-module structure on $R (\Z_N)$
as in Lemma \ref{L_5328_mao}(\ref{I_5328_mao_RZN})
and the $\Z [ \sm, \sm^{-1} ]$-module structure on~$\Z$
coming from its identification with
$\Z [ \sm, \sm^{-1} ] / \langle \ov{\sm} - 1 \rangle$
(the case $N = 1$ of Lemma \ref{L_5328_mao}(\ref{I_5328_mao_RZN})).
So $\mathrm{mao} ( K_0^{S^1} (A_n) ) = N$ by
Lemma \ref{L_5328_mao}(\ref{I_5328_mao_RZN})
and Lemma \ref{L_5328_mao}(\ref{I_5328_mao_Sum}).
Therefore $\mathrm{mao} ( K_0^{S^1} (A) ) = N$
by Lemma \ref{L_1918_MapRN}(\ref{I_5328_MapRN_K})
and Lemma \ref{L_5328_mao}(\ref{I_5328_mao_RZN}).
\end{proof}

\begin{thm}\label{T_5328_Compare}
Let $\af \colon S^1 \to \Aut (A)$ and $\bt \colon S^1 \to \Aut (B)$
be two actions as in Theorem~\ref{T_1919_Ex}, using different choices of~$N$.
Then $A$ and $B$ are not equivariantly isomorphic.
\end{thm}

\begin{proof}
It follows from Lemma~\ref{L_5328_maoK0A}
that $\mathrm{mao} (K_0^{S^1} (A)) \neq \mathrm{mao} (K_0^{S^1} (B))$.
\end{proof}

We now address the modified tracial Rokhlin property.

\begin{lem}\label{P_1929_RComm}
Let the notation be as in Construction~\ref{Cn_1824_S1An}.
Let $n \in \Nz$.
Define
\[
D_1 = M_{l_{0, 0} (n + 1) l_{0, 0} (n) + l_{0, 1} (n + 1) l_{1, 0} (n)},
\qquad
D_2 = M_{l_{0, 1} (n + 1) l_{1, 0} (n)},
\]
\[
D_3 = M_{l_{0, 0} (n + 1) l_{0, 1} (n)
           + 2 N l_{0, 1} (n + 1) l_{1, 1} (n)},
\qquad
D_4 = M_{l_{1, 0} (n + 1) l_{0, 0} (n)
           + 2 N l_{1, 1} (n + 1) l_{1, 0} (n)},
\]
and
\[
D_5 = M_{N l_{1, 0} (n + 1) l_{0, 1} (n)
       + 4 N^2 l_{1, 1} (n + 1) l_{1, 1} (n)},
\]
and define
\begin{equation}\label{Eq_1929_mj}
m_1 = N, \qquad
m_2 = N^2 - N, \qquad
m_3 = N, \qquad
m_4 = N, \andeqn
m_5 = 1.
\end{equation}
Set
\begin{equation}\label{Eq_1929_D}
D = (D_1)^{m_1} \oplus (D_2)^{m_2} \oplus (D_3)^{m_3}
  \oplus (D_4)^{m_4} \oplus (D_5)^{m_5},
\end{equation}
and for $k \in \{ 1, 2, 3, 4, 5 \}$ and $j = 1, 2, \ldots, m_j$,
let $\pi_{k, j} \colon D \to D_k$ be the projection to the
$j$~summand of $D_k$ in the definition of~$D$.
As usual, write the relative commutant of
$\nu_{n + 2, \, n} \bigl( (A_n)^{\af^{(n)}} \bigr)$
in $(A_{n + 2})^{\af^{(n + 2)}}$
as
$\nu_{n + 2, \, n} \bigl( (A_n)^{\af^{(n)}} \bigr)'
    \cap (A_{n + 2})^{\af^{(n + 2)}}$.
Then
\[
\nu_{n + 1} (p_{n + 1})
  \in \nu_{n + 2, \, n} \bigl( (A_n)^{\af^{(n)}} \bigr)'
    \cap (A_{n + 2})^{\af^{(n + 2)}},
\]
and there is an isomorphism
\[
\kp \colon \nu_{n + 2, \, n} \bigl( (A_n)^{\af^{(n)}} \bigr)'
    \cap (A_{n + 2})^{\af^{(n + 2)}}
  \to D
\]
such that for $k \in \{ 1, 2, 3, 4, 5 \}$ and $j = 1, 2, \ldots, m_j$,
we have
\[
\rank \bigl( (\pi_{1, j} \circ \kp \circ \nu_{n + 1}) (p_{n + 1}) \bigr)
 = \rank
    \bigl( (\pi_{2, j} \circ \kp \circ \nu_{n + 1}) (p_{n + 1}) \bigr)
 = l_{0, 1} (n + 1) l_{1, 0} (n),
\]
\[
\rank \bigl( (\pi_{3, j} \circ \kp \circ \nu_{n + 1}) (p_{n + 1}) \bigr)
 = 2 N l_{0, 1} (n + 1) l_{1, 1} (n),
\]
\[
\rank \bigl( (\pi_{4, j \circ \kp} \circ \nu_{n + 1}) (p_{n + 1}) \bigr)
 = 2 N l_{1, 1} (n + 1) l_{1, 0} (n),
\]
and
\[
\rank \bigl( (\pi_{5, j} \circ \kp \circ \nu_{n + 1}) (p_{n + 1}) \bigr)
 =  4 N^2 l_{1, 1} (n + 1) l_{1, 1} (n).
\]
\end{lem}

\begin{proof}
Let the notation be as in parts
(\ref{Item_1824_S1_Bn0Bn1}), (\ref{Item_1824_S1_FixPtSys}),
(\ref{Item_1824_S1_FixLim}), and~(\ref{Item_1915_S1_qn})
of Construction~\ref{Cn_1824_S1An}.
By Lemma~\ref{L_1916_FPisB}, it is enough to prove the lemma
with $B_t$ in place of $(A_t)^{\af^{(n)}}$
for $t = n, \, n + 1, \, n + 1$,
with $\ch_{t}$ and $\ch_{t, s}$ in place of $\nu_{t}$ and $\nu_{t, s}$,
and with $q_{n + 1}$ in place of $p_{n + 1}$.

First, since $q_{n + 1}$ is in the center of $B_{n + 1}$,
it commutes with the range of $\ch_{n}$.
So $\ch_{n + 1} (p_{n + 1}) \in \ch_{n + 2, \, n} (B_n)' \cap B_{n + 2}$
is clear.

We change to more convenient notation
for the structure of the algebras~$B_t$.
Write
\[
B_t = B_{t, 0} \oplus B_{t, 1}
   \oplus \cdots \oplus B_{t, N - 1} \oplus B_{t, N},
\]
with
\[
B_{t, 0} = B_{t, 1} = \cdots = B_{t, N - 1} = M_{r_0 (n)}
\andeqn
B_{t, N} = M_{r_1 (n)}.
\]
Thus $B_{t, 0} \oplus B_{t, 1} \oplus \cdots \oplus B_{t, N - 1}$
is what was formerly called $B_{t, 0}$,
and $B_{t, N}$ is what was formerly called $B_{t, 1}$.
The partial multiplicities in $\ch_t$
of the maps $B_{t, j} \to B_{t + 1, k}$ are
\[
m_t (k, j) = \begin{cases}
   l_{0, 0} (t) & \hspace*{1em} j, k \in \{ 0, 1, \ldots, N - 1 \}
        \\
   l_{0, 1} (t) & \hspace*{1em}
      {\mbox{$j = N$ and $k \in \{ 0, 1, \ldots, N - 1 \}$}}
        \\
   l_{1, 0} (t) & \hspace*{1em}
      {\mbox{$j \in \{ 0, 1, \ldots, N - 1 \}$ and $k = N$}}
       \\
   2 N l_{1, 1} (t) & \hspace*{1em} j - k = N.
\end{cases}
\]
An easy calculation now
shows that the partial multiplicities in $\ch_{n + 2, n}$
of the maps $B_{n, j} \to B_{n + 2, k}$ are
\begin{equation}\label{Eq_1929_Cases}
\widetilde{m} (k, j) = \begin{cases}
& \hspace*{-1em}
l_{0, 0} (n + 1) l_{0, 0} (n) + l_{0, 1} (n + 1) l_{1, 0} (n)
\\
& \hspace*{10em} {\mbox{}}
      {\mbox{$j, k \in \{ 0, 1, \ldots, N - 1 \}$ and $j = k$}}
        \\
& \hspace*{-1em}
l_{0, 1} (n + 1) l_{1, 0} (n)
\\
& \hspace*{10em} {\mbox{}}
      {\mbox{$j, k \in \{ 0, 1, \ldots, N - 1 \}$ and $j \neq k$}}
        \\
& \hspace*{-1em}
l_{0, 0} (n + 1) l_{0, 1} (n)
           + 2 N l_{0, 1} (n + 1) l_{1, 1} (n)
\\
& \hspace*{10em} {\mbox{}}
      {\mbox{$j = N$ and $k \in \{ 0, 1, \ldots, N - 1 \}$}}
        \\
& \hspace*{-1em}
l_{1, 0} (n + 1) l_{0, 0} (n)
           + 2 N l_{1, 1} (n + 1) l_{1, 0} (n)
\\
& \hspace*{10em} {\mbox{}}
      {\mbox{$j \in \{ 0, 1, \ldots, N - 1 \}$ and $k = N$}}
       \\
& \hspace*{-1em}
N l_{1, 0} (n + 1) l_{0, 1} (n)
       + 4 N^2 l_{1, 1} (n + 1) l_{1, 1} (n)
\\
& \hspace*{10em} {\mbox{}} j - k = N.
\end{cases}
\end{equation}
Recall that if
\[
\ph \colon M_{r_1} \oplus M_{r_2} \oplus \cdots \oplus M_{r_s}
 \to M_{l_1 r_1 + l_2 r_2 + \cdots + l_2 r_2}
\]
is unital with partial multiplicities $l_1, l_2, \ldots, l_s$,
then the relative commutant of the range of $\ph$ is isomorphic to
$M_{l_1} \oplus M_{l_2} \oplus \cdots \oplus M_{l_s}$,
with the identity of $M_{l_t}$ being,
by abuse of notation, $\ph (1_{M_{l_t}})$.
Therefore the values of ${\widetilde{m}} (k, j)$ are the matrix sizes
of the summands in $\ch_{n + 2, \, n} (B_n)' \cap B_{n + 2}$.
That the exponents $m_j$ in~(\ref{Eq_1929_D}) are given as
in~(\ref{Eq_1929_mj}) follows simply by counting the number of
times each case in~(\ref{Eq_1929_Cases}) occurs.

The rank of the image of $q_{n + 1}$ in each summand is the
contribution to ${\widetilde{m}} (k, j)$ from maps factoring
through $B_{n + 1, N}$ (in the original notation, $B_{n + 1, 1}$),
as opposed to the other summands.
That these numbers are in the statement of the lemma is again
an easy calculation.
\end{proof}

\begin{prp}\label{L_1X05_HasMdTRP}
In Construction~\ref{Cn_1824_S1An}, assume that $r_0 (0) \leq r_1 (0)$
and that for all $n \in \Nz$
we have
\[
l_{1, 0} (n) \geq l_{0, 0} (n),
\qquad
l_{0, 1} (n) \geq l_{0, 0} (n),
\]
\[
l_{1, 1} (n) \geq l_{0, 1} (n),
\andeqn
l_{1, 1} (n) \geq l_{1, 0} (n).
\]
Further assume that
\[
\lim_{n \to \I} \frac{l_{0, 1} (n)}{l_{0, 0} (n)} = \I
\andeqn
\lim_{n \to \I} \frac{l_{1, 1} (n)}{l_{1, 0} (n)} = \I.
\]
Then the action $\af$
of Construction \ref{Cn_1824_S1An}(\ref{Item_1824_S1_nu_A})
has the tracial Rokhlin property with comparison
and has the modified tracial Rokhlin property
as in Definition~\ref{moditra}.
Moreover, given finite sets $F \subseteq A$, $F_0 \subseteq A^{\af}$,
and $S \subseteq C (G)$, as well as $\ep > 0$,
$x \in A_{+}$ with $\| x \| = 1$,
and $y \in (A^{\alpha})_{+} \SM \{ 0 \}$,
it is possible to choose a projection $p \in A^{\alpha}$,
a unital completely positive
contractive map $\ph \colon C (G) \to p A p$,
and a partial isometry $s \in A^{\alpha}$,
such that the conditions of both
Definition~\ref{traR} and Definition~\ref{moditra}
are simultaneously satisfied.
\end{prp}

As usual, the hypotheses are overkill.

\begin{proof}[Proof of Proposition~\ref{L_1X05_HasMdTRP}]
Since $A$ is finite, as usual, the argument of
Lemma~1.16 of~\cite{phill23} applies, and shows that it suffices
to verify this without the condition $\| p x p \| > 1 - \ep$,
simultaneously in both definitions.

The method used for
the proof of Lemma~\ref{L_1917_HasTRP} also applies here.
The key new point is that in every summand of
\[
\nu_{n + 2, \, n} \bigl( (A_n)^{\af^{(n)}} \bigr)'
    \cap (A_{n + 2})^{\af^{(n + 2)}}
\]
as described in Lemma~\ref{P_1929_RComm}, the rank of the
component of $\nu_{n + 1} (p_{n + 1})$
is greater than half the corresponding matrix size.
Therefore the rank of
the component of $1 - \nu_{n + 1} (p_{n + 1})$
is less than the rank of the component of $\nu_{n + 1} (p_{n + 1})$.
Rank comparison implies that there exists
\[
s_{0} \in \nu_{n + 2, \, n} \bigl( (A_n)^{\af^{(n)}} \bigr)'
    \cap (A_{n + 2})^{\af^{(n + 2)}}
\]
such that
\[
1 - \nu_{n + 1} (p_{n + 1}) = s_{0} ^{*} s_{0}
\andeqn
s_{0} s_{0} ^{*} \leq \nu_{n + 1} (p_{n + 1}).
\]
Now set $s = \nu_{\infty, n + 1} (s_{0})$.
The rest of the proof is an easy computation.
\end{proof}

\section{Actions of $S^1$ on Kirchberg algebras}\label{Sec_2114_OI}

\indent
The purpose of this section is to construct a action of $S^1$ on $\OI$
which has the tracial Rokhlin property with comparison.
As observed in Lemma~\ref{2610_No_fin_RD} below
(really just Corollary 4.23 of \cite{Gar_rokhlin_2017}),
there is no action of $S^1$
on $\OI$ which has finite Rokhlin dimension with commuting towers.
Although we don't carry this out,
easy modifications should work for ${\mathcal{O}}_n$ in place of~$\OI$,
and it should also be not too hard to generalize the construction
to cover actions of $(S^1)^m$ and $(S^1)^{\Z}$.

Tensoring with our action gives an action on any unital Kirchberg algebra
which has the restricted tracial Rokhlin property with comparison
(Definition~\ref{traR}).
One also can get actions on some unital Kirchberg algebras
by tensoring the actions of Theorem~\ref{T_1919_Ex}
with the trivial action on~$\OI$.
When one gets two actions on the same unital Kirchberg algebra this way,
the actions are not conjugate,
both have the restricted tracial Rokhlin property with comparison,
and neither has finite Rokhlin dimension with commuting towers.

\begin{lem}\label{L_2521_Exist}
There exist an action $\bt \colon S^1 \to \Aut (\OI)$,
a nonzero \pj{} $p \in \OI^{\bt}$, and a unital \hm{}
$\io \colon C (S^1, \OT) \to p \OI p$ such that $\io$ is equivariant
when $C (S^1, \OT)$ is equipped with the action
(following Notation~\ref{N_1X07_Lt})
$\zt \mapsto \Lt_{\zt} \otimes \id_{\OT}$,
and such that $p$ has the following property.
For every $n \in \N$ and every \pj{} $q \in M_n \otimes \OI^{\bt}$,
there is $t \in M_n \otimes \OI^{\bt}$
such that $\| t \| = 1$ and $t^* (e_{1, 1} \otimes p) t = q$.
\end{lem}

\begin{proof}
Write $s_1, s_2, \ldots$ for the standard generating isometries
in~$\OI$.
Also write $w$ for the standard generating unitary in $C (S^1)$,
which is the function $w (z) = z$ for $z \in S^1$.
When convenient, identify $C (S^1, \OT)$ with $C (S^1) \otimes \OT$.

Let $\gm \colon S^1 \to \Aut (\OI)$ be the quasifree action
(in the sense of~\cite{Kts2})
determined by, for $\zt \in S^1$,
\[
\gm_{\zt} (s_j)
 = \begin{cases}
   s_j & \hspace*{1em} j = 1
        \\
   \zt s_j & \hspace*{1em} j = 2
       \\
   \zt^{-1} s_j & \hspace*{1em} j = 3
       \\
   s_j & \hspace*{1em} j = 4, 5, \ldots,
\end{cases}
\]
and let ${\ov{\gm}} \colon S^1 \to \Aut (\OI)$
be the quasifree action $\zt \mapsto \gm_{\zt}^{- 1}$.
Define an action
$\bt^{(0)} \colon S^1 \to \Aut (\OI \otimes \OI \otimes \OI)$ by
$\bt^{(0)}_{\zt} = \gm_{\zt} \otimes {\ov{\gm}}_{\zt} \otimes \id_{\OI}$
for $\zt \in S^1$.

In $\OI$, define
\[
e = 1 - s_1 s_1^*,
\qquad
v_1 = s_2 e,
\andeqn
v_2 = 1 - s_1 s_1^* - s_2 s_2^* + s_2 s_1 s_2^*.
\]
In $\OI \otimes \OI \otimes \OI$, define
\[
p_0 = e \otimes e \otimes e
\andeqn
u = v_1 \otimes v_2^* \otimes e + v_2 \otimes v_1^* \otimes e.
\]
One checks that $u$ is a unitary in $e (\OI \otimes \OI \otimes \OI) e$,
and that
\begin{equation}\label{Eq_2605_Act_u}
\bt^{(0)}_{\zt} (u) = \zt u
\end{equation}
for all $\zt \in S^1$.
Since $[e] = 0$ in $K_0 (\OI)$, there is a unital \hm{}
$\mu \colon \OT \to e \OI e$.
Since $u$ commutes with $e \otimes e \otimes \mu (a)$ for all $a \in \OT$,
there is a \uhm{}
$\io_0 \colon C (S^1) \otimes \OT \to p_0 (\OI \otimes \OI \otimes \OI) p_0$
such that $\io_0 (w \otimes 1) = u$ and
$\io_0 (1 \otimes a) = e \otimes e \otimes \mu (a)$ for $a \in \OT$.
Then $\io_0$ is equivariant by~(\ref{Eq_2605_Act_u}).

Let $\sm \colon \OI \otimes \OI \otimes \OI \to \OI$ be an isomorphism.
Define $\io = \sm \circ \io_0$,
$p = \sm (p_0)$, and $\bt_{\zt} = \rh \circ \bt^{(0)}_{\zt} \circ \rh^{-1}$
for $\zt \in S^1$.
Then $\bt$ an action of $S^1$ on~$\OI$,
$p$ is a \pj{} in $\OI^{\bt}$, and
$\io$ is an equivariant unital \hm{} from $C (S^1, \OT)$ to $p \OI p$.

It remains to prove the last sentence.
For this purpose, it suffices to use
$\OI \otimes \OI \otimes \OI$ in place of $\OI$,
$\bt^{(0)}$ in place of $\bt$, and $p_0$ in place of~$p$.
We may also assume that $q = 1_{M_n} \otimes 1$ for some $n \in \N$.

We first claim that $\OI^{\gm}$ is purely infinite and simple.
One checks from the definition of $\om$-invariance for subsets of~$\Gm$
(Definition 3.3 of~\cite{Kts2}),
with $\om = (0, 1, -1, 0, 0, \ldots)$,
that, in our case, $\Z$ has no nontrivial invariant subsets.
So Proposition~7.4 of~\cite{Kts2} implies that $C^* (S^1, \OI, \gm)$
is purely infinite and simple.
Now the claim follows from Theorem 3.5 in \cite{MhkPh1}.
Since ${\ov{\gm}}$ has the same fixed point algebra,
it also follows that $\OI^{\ov{\gm}}$ is purely infinite and simple.

Next, define an action
${\widetilde{\bt}} \colon
   S^1 \times S^1 \to \Aut (\OI \otimes \OI \otimes \OI)$
by
\[
{\widetilde{\bt}}_{\zt_1, \zt_2}
 = \gm_{\zt_1} \otimes {\ov{\gm}}_{\zt_2} \otimes \id_{\OI}
\]
for $\zt_1, \zt_2 \in S^1$.
Then
\[
(\OI \otimes \OI \otimes \OI)^{\widetilde{\bt}}
 = \OI^{\gm} \otimes \OI^{\ov{\gm}} \otimes \OI,
\]
which is purely infinite and simple.
Hence so is $M_n \otimes (\OI \otimes \OI \otimes \OI)^{\widetilde{\bt}}$.
Clearly $e_{1, 1} \otimes p_0$ is a nonzero \pj{} in
$M_n \otimes (\OI \otimes \OI \otimes \OI)^{\widetilde{\bt}}$.
Therefore there exists an isometry
$t \in M_n \otimes (\OI \otimes \OI \otimes \OI)^{\widetilde{\bt}}$
such that $t^* t = 1$ and $t t^* \leq e_{1, 1} \otimes p_0$.
It is immediate that
$t \in M_n \otimes (\OI \otimes \OI \otimes \OI)^{\bt^{(0)}}$
and $t^* (e_{1, 1} \otimes p_0) t = q$.
\end{proof}

\begin{cns}\label{Cns_2114_OIA}
We define an equivariant direct system,
and a corresponding direct limit action, as follows.
That we actually get an equivariant direct system is proved afterwards,
in Lemma~\ref{L_2605_Cns_OK}.
\begin{enumerate}
\item\label{I_2114_OIA_A}
Let $\bt \colon S^1 \to \Aut (\OI)$, $p \in \OI^{\bt} \subseteq \OI$, and
$\io \colon C (S^1, \OT) \to p \OI p$
be as in Lemma~\ref{L_2521_Exist}.
Let $(e_{j, k})_{j, k \in \Nz}$ be the standard system of
matrix units for $K = K (l^2 (\Nz))$.
For $n \in \Nz$ define $f_n \in K \otimes \OI$ by
\[
f_n = e_{0, 0} \otimes 1 + e_{1, 1} \otimes p + e_{2, 2} \otimes p
                                        + \cdots e_{n, n} \otimes p.
\]
Then define
\[
A_{n, 0} = f_{n} ( K \otimes \OI ) f_{n},
\quad
A_{n, 1} = C (S^1) \otimes \OT = C (S^1, \OT),
\quad {\mbox{and}} \quad
A_{n} = A_0 \oplus A_1.
\]
\item\label{I_2521_OIA_Action0}
Define actions
\[
\af^{(n, 0)} \colon S^{1} \to \Aut (A_{n, 0}),
\quad
\af^{(n, 1)} \colon S^1 \to \Aut (A_{n, 1})
\quad {\mbox{and}} \quad
\af^{(n)} \colon S^1 \to \Aut (A_{n})
\]
as follows.
For $\zeta \in S^1$, take $\af^{(n, 0)}_{\zeta}$ to be the
restriction and corestriction of $\id_K \otimes \bt_{\zt}$
to the subalgebra $f_{n} ( K \otimes \OI ) f_{n}$, and,
using Notation~\ref{N_1X07_Lt} for the first, take
\[
\af^{(n, 1)}_{\zeta} = \Lt_{\zt} \otimes \id_{\OT}
\quad {\mbox{and}} \quad
\af^{(n, 1)}_{\zeta} = \af^{(n, 0)}_{\zeta} \oplus \af^{(n, 1)}_{\zeta}.
\]
\item\label{I_2114_OIA_PreParts}
We give preliminaries for the construction of the maps of the system.
Fix \nzp{s} $q_0, q_1 \in \OT$ such that $q_0 + q_1 = 1$.
For $n \in \Nz$ choose unital \hm{s}
\[
\mu_{n, 0} \colon A_{n, 0} \to q_0 \OT q_0,
\andeqn
\mu_{n, 1} \colon \OT \to q_1 \OT q_1.
\]
(The \hm{} $\mu_{n, 1}$ can be chosen to be independent of~$n$.)
Use Proposition~4.1 of~\cite{Rokhdimtracial} to choose an isomorphism
$\ld_n \colon C (S^1, A_{n, 0}) \to C (S^1, A_{n, 0})$ which is equivariant
when $S^1$ acts on the domain
via $\zt \mapsto \Lt_{\zt} \otimes \af^{(n, 0)}_{\zeta}$
and on the codomain via $\zt \mapsto \Lt_{\zt} \otimes \id_{A_{n, 0}}$.
Let $\kp_n \colon A_{n, 0} \to C (S^1, A_{n, 0})$ be the inclusion of
elements of $A_{n, 0}$ as constant functions.
\item\label{I_2114_OIA_Parts}
For $n \in \Nz$ and $j, k \in \{ 0, 1 \}$ define \hm{s}
$\nu_{n + 1, n}^{(k, j)} \colon A_{n, j} \to A_{n + 1, k}$ as follows.
We let $\nu_{n + 1, n}^{(0, 0)}$ be the inclusion of
$A_{n, 0}$ in $A_{n + 1, 0}$
which comes from the relation $f_n \leq f_{n + 1}$.
For $a \in C (S^1, \OT)$, with $\io$ as in~(\ref{I_2114_OIA_A}),
take $\nu_{n + 1, n}^{(0, 1)} (a) = e_{n + 1, \, n + 1} \otimes \io (a)$.
Set
\[
\nu_{n + 1, n}^{(1, 0)}
 = (\id_{C (S^1)} \otimes \mu_{n, 0}) \circ \ld_n \circ \kp_n
\andeqn
\nu_{n + 1, n}^{(1, 1)} (a)
 = \id_{C (S^1)} \otimes \mu_{n, 1}.
\]
Then define $\nu_{n + 1, n} \colon A_{n} \to A_{n + 1}$ by
\[
\nu_{n + 1, n} (a_0, a_1)
 = \bigl( \nu_{n + 1, n}^{(0, 0)} (a_0) + \nu_{n + 1, n}^{(0, 1)} (a_1),
     \, \nu_{n + 1, n}^{(1, 0)} (a_0) + \nu_{n + 1, n}^{(1, 1)} (a_1) \bigr)
\]
for $a_0 \in A_{n, 0}$ and $a_1 \in A_{n, 1}$.
\item\label{I_2114_OIA_nm}
For $m, n \in \Nz$ with $m \leq n$, define
\[
\nu_{n, m} = \nu_{n, \, n - 1} \circ \nu_{n - 1, \, n - 2}
          \circ \cdots \circ \nu_{m + 1, \, m}.
\]
\item\label{I_2114_OIA_Lim}
Define $A = \dirlim_n A_n$, using the maps $\nu_{n + 1, n}$,
and let $\af \colon S^1 \to \Aut (A)$ be the direct limit action.
\end{enumerate}
\end{cns}

\begin{lem}\label{L_2605_Cns_OK}
Adopt the notation of Construction~\ref{Cns_2114_OIA}.
\begin{enumerate}
\item\label{I_L_2605_Cns_OK_Algs}
For $m, n \in \Nz$ with $n \geq m$, the map
$\nu_{n + 1, n} \colon A_{n} \to A_{n + 1}$
is an equivariant unital \hm.
\item\label{I_L_2605_Cns_inj}
For $n \in \Nz$ and $j, k \in \{ 0, 1 \}$,
the map $\nu_{n + 1, n}^{(k, j)}$ in
Construction \ref{Cns_2114_OIA}(\ref{I_2114_OIA_Parts}) is injective.
\item\label{I_L_2605_Cns_OK_Lim}
The action $\af \colon S^1 \to \Aut (A)$ is a well defined \ct{} action.
\end{enumerate}
\end{lem}

\begin{proof}
We prove (\ref{I_L_2605_Cns_OK_Algs}).
We need only consider $\nu_{n + 1, n}$ for $n \in \Nz$.
That $\nu_{n + 1, n}$ is a unital \hm{} follows from the
computations (in which the summands are orthogonal)
\[
\nu_{n + 1, n}^{(0, 0)} (f_n) + \nu_{n + 1, n}^{(0, 1)} (1)
  = f_{n} + e_{n + 1, \, n + 1} \otimes p
  = f_{n + 1}
\]
and
\[
\nu_{n + 1, n}^{(1, 0)} (f_n) + \nu_{n + 1, n}^{(1, 1)} (1)
 = 1 \otimes \mu_{n, 0} (f_n) + 1 \otimes \mu_{n, 1} (1)
 = 1 \otimes q_0 + 1 \otimes q_1
 = 1.
\]
For the actions to be well defined, the only point which needs to be
checked is that $f_n$ is invariant under $\id_K \otimes \bt_{\zt}$.
This follows from invariance of $p$ under $\bt$,
which is a consequence of the choices made using Lemma~\ref{L_2521_Exist}.

For equivariance, it is enough to check equivariance of
the maps $\nu_{n + 1, \, n}^{(k, j)}$ in
Construction \ref{Cns_2114_OIA}(\ref{I_2114_OIA_Parts}).
This is immediate for $\nu_{n + 1, \, n}^{(0, 0)}$
and $\nu_{n + 1, \, n}^{(1, 1)}$,
and by the choices made using Lemma~\ref{L_2521_Exist}
for $\nu_{n + 1, \, n}^{(0, 1)}$.
The map $\nu_{n + 1, \, n}^{(1, 0)}$ is the composition
(writing pairs consisting of an algebra and an action)
\[
\begin{split}
\bigl( A_{n, 0}, \, \af^{(n, 0)} \bigr)
& \stackrel{\kp_n}{\longrightarrow}
\bigl( C (S^1) \otimes A_{n, 0}, \, \Lt \otimes \af^{(n, 0)} \bigr)
\\
& \stackrel{\ld_n}{\longrightarrow}
\bigl( C (S^1) \otimes A_{n, 0}, \, \Lt \otimes \id_{A_{n, 0}} \bigr)
\stackrel{\id_{C (S^1)} \otimes \mu_{n, 0}}{\longrightarrow}
\bigl( C (S^1) \otimes \OT, \, \Lt \otimes \id_{\OT} \bigr).
\end{split}
\]
The maps $\kp_n$ and $\id_{C (S^1)} \otimes \mu_{n, 0}$ are
obviously equivariant, and $\ld_n$ is equivariant by construction.

For~(\ref{I_L_2605_Cns_inj}), in the compositions defining these maps,
$\io$ is injective by the choices made using Lemma~\ref{L_2521_Exist}
and $\ld_n$ is injective by construction.
Injectivity of everything else which appears is immediate.

Part~(\ref{I_L_2605_Cns_OK_Lim})
is immediate from part~(\ref{I_L_2605_Cns_OK_Algs}).
\end{proof}

The following lemma is known, but we have not found a reference.

\begin{lem}\label{L_2521_PI_bound}
Let $A$ be a unital purely infinite simple \ca, let $\rh > 0$,
and let $x \in A$ satisfy $\| x \| > \rh$.
Then there are $a, b \in A$ such that $a x b = 1$,
$\| a \| < \rh^{-1}$, and $ \| b \| \leq 1$.
\end{lem}

\begin{proof}
Choose $\rh_0$ such that $(\rh \| x \|)^{1/2} < \rh_0 < \| x \|$.

Define \cfn{s} $f, f_0 \colon [0, \I) \to [0, \I)$ by
\[
f (\ld)
 = \begin{cases}
   \rh_0^{-2} \ld & \hspace*{1em} 0 \leq \ld \leq \rh_0^2
        \\
   1 & \hspace*{1em} \rh_0^2 < \ld
\end{cases}
\andeqn
f_0 (\ld)
 = \begin{cases}
   0           & \hspace*{1em} 0 \leq \ld \leq \rh_0^2
        \\
   \ld - \rh_0^2 & \hspace*{1em} \rh_0^2 < \ld.
\end{cases}
\]
Then $f_0 (x^* x) \neq 0$ since $\| x^* x \| > \rh_0^2$.
Therefore there are a \nzp{} $p \in {\ov{f_0 (x^* x) A f_0 (x^* x)}}$,
and, by pure infiniteness,
a partial isometry $s \in A$ such that $s^* s = 1$ and $s s^* \leq p$.
We have $f (x^* x) p = p$ and $p s = s$.
Therefore
\[
1 = s^* p s
  = s^* f (x^* x) p s
  = s^* f (x^* x) s
  \leq \rh_0^{-2} s^* x^* x s.
\]
Hence $s^* x^* x s$ is invertible, with
$\| (s^* x^* x s)^{-1} \| \leq \rh_0^{-2}$.
Take $b = s$ and $a = (s^* x^* x s)^{-1} s^* x^*$,
noting that $\| a \| \leq \rh_0^{-2} \| x \| < \rh^{-1}$.
\end{proof}

\begin{lem}\label{L_2521_SLim}
Let $(B_n)_{n \in \Nz}$ be a direct system of unital \ca{s},
with unital maps $\mu_{n, m} \colon B_m \to B_{n}$.
Suppose that for all $n \in \Nz$ we are given a direct sum decomposition
$B_n = B_n^{(0)} \oplus B_n^{(1)}$, in which both summands are nonzero.
For $m, n \in \Nz$ and $j, k \in \{ 0, 1 \}$ let
$\mu_{n, m}^{(k, j)} \colon B_m^{(j)} \to B_n^{(k)}$
be the corresponding partial map.
Assume the following:
\begin{enumerate}
\item\label{I__2521_SLim_pi}
$B_n^{(1)}$ is purely infinite and simple for all $n \in \Nz$.
\item\label{I__2521_SLim_inj}
$\mu_{n + 1, n}^{(k, j)}$ is injective for all $n \in \Nz$
and all $j, k \in \{ 0, 1 \}$.
\item\label{I__2521_SLim_full}
For every $n \in \Nz$ there is $t \in B_{n + 1}^{(0)}$
such that $\| t \| = 1$
and $t^* \mu_{n + 1, n}^{(0, 1)} (1_{B_n^{(1)}}) t$
is the identity of $B_{n + 1}^{(0)}$.
\end{enumerate}
Then $\dirlim_n B_n$ is purely infinite and simple.
\end{lem}

\begin{proof}
Set $B = \dirlim_n B_n$,
and for $m \in \Nz$ let $\mu_{\I, n} \colon B_n \to B$
be the standard map associated with the direct limit.
Also, for $n \in \Nz$ and $j \in \{ 0, 1 \}$ let $p_n^{(j)}$
be the identity of $B_n^{(j)}$.

Let $x \in B \SM \{ 0 \}$.
We need to find $a, b \in B$ such that $a x b = 1$.
\Wolog{} $\| x \| = 1$.
Choose $n \in \Nz$ and $y \in B_n$ such that
$\| \mu_{\I, n} (y) - x \| < \frac{1}{3}$.
Then $\| y \| > \frac{2}{3}$.
Set $z = \mu_{n + 1, n} (y)$.
Write $y = (y_0, y_1)$ and $z = (z_0, z_1)$ with $y_j \in B_n^{(j)}$
and $z_j \in B_{n + 1}^{(j)}$ for $j \in \{ 0, 1 \}$.
We have $\| y_j \|= \| y \|$ for some $j$,
so injectivity of $\mu_{n + 1, n}^{(1, j)}$
implies that $\| z_1 \| > \frac{2}{3}$.

Lemma~\ref{L_2521_PI_bound} provides $r_0, s_0 \in B_{n + 1}^{(1)}$
such that
\[
\| r_0 \| < \frac{3}{2},
\qquad
\| s_0 \| \leq 1,
\andeqn
r_0 z_1 s_0 = p_{n + 1}^{(1)}.
\]
Condition~(\ref{I__2521_SLim_full})
provides $t \in B_{n + 2}^{(0)}$
such that $\| t \| = 1$
and
$t^* \mu_{n + 2, \, n + 1}^{(0, 1)} \bigl( p_{n + 1}^{(1)} \bigr) t
  = p_{n + 2}^{(0)}$.
Define elements of $B_{n + 1}^{(0)}$ by
\[
c_0 = t^* \mu_{n + 2, \, n + 1}^{(0, 1)} (r_0)
\andeqn
d_0 = \mu_{n + 2, \, n + 1}^{(0, 1)} (s_0) t.
\]
Then
\begin{equation}\label{Eq_2605_Zero}
\| c_0 \| < \frac{3}{2},
\qquad
\| d_0 \| \leq 1,
\andeqn
t^* \mu_{n + 2, \, n + 1} \bigl( p_{n + 1}^{(0)} \bigr) t = 0.
\end{equation}
Using the last part of~(\ref{Eq_2605_Zero}),
and regarding everything in the first expression as being in $B_{n + 1}^{(0)}$,
we get
\begin{equation}\label{Eq_2605_3St}
c_0 \mu_{n + 2, \, n + 1} (z) d_0
 = c_0 \mu_{n + 2, \, n + 1}^{(0, 1)} (z_1) d_0
 = t^* \mu_{n + 2, \, n + 1}^{(0, 1)} (r_0 z_1 s_0) t
 = p_{n + 2}^{(0)}.
\end{equation}

Since $\mu_{n + 2, \, n + 1}^{(1, 1)}$ is injective
(by~(\ref{I__2521_SLim_inj}))
and $B_{n + 2}^{(1)}$ is purely infinite and simple,
Lemma~\ref{L_2521_PI_bound} provides $r_1, s_1 \in B_{n + 2}^{(1)}$
such that
\[
\| r_1 \| < \frac{3}{2},
\qquad
\| s_1 \| \leq 1,
\andeqn
r_1 \mu_{n + 2, \, n + 1}^{(1, 1)} (z_1) s_1 = p_{n + 2}^{(1)}.
\]
Define elements of $B_{n + 1}^{(1)}$ by
\[
q = \mu_{n + 2, \, n + 1}^{(1, 1)} \bigl( p_{n + 1}^{(1)} \bigr),
\qquad
c_1 = r_1 q,
\andeqn
d_1 = q s_1.
\]
Then, analogously to~(\ref{Eq_2605_Zero}) and~(\ref{Eq_2605_3St}),
\[
\| c_1 \| < \frac{3}{2},
\qquad
\| d_1 \| \leq 1,
\andeqn
c_1 \mu_{n + 2, \, n + 1} (z) d_1
 = c_1 \mu_{n + 2, \, n + 1}^{(1, 1)} (z_1) d_1
 = p_{n + 2}^{(1)}.
\]

Taking $c = (c_0, c_1)$ and $d = (d_0, d_1)$,
we get
\[
\| c \| < \frac{3}{2},
\qquad
\| d \| \leq 1,
\andeqn
c \mu_{n + 2, \, n + 1} (z) d = 1.
\]
Setting $a = \mu_{\I, \, n + 2} (c)$ and $h = \mu_{\I, \, n + 2} (d)$,
we get $\| a \| < \frac{3}{2}$, $\| h \| \leq 1$, and
$a \mu_{\I, \, n} (y) h = 1$.
Therefore
\[
\| a x h - 1 \|
  \leq \| a \| \| x - \mu_{\I, \, n} (y) \| \| h \|
  < \left( \frac{3}{2} \right) \left( \frac{1}{3} \right)
  = \frac{1}{2}.
\]
So $a x h$ is invertible.
Setting $b = h (a x h)^{-1}$ gives $a x b = 1$, as desired.
\end{proof}

\begin{lem}\label{L_2605_Cns_OI}
The algebra $A$ in Construction \ref{Cns_2114_OIA}(\ref{I_2114_OIA_Lim})
is isomorphic to~$\OI$.
\end{lem}

\begin{proof}
We first claim that $A$ is purely infinite and simple.
We use Lemma~\ref{L_2521_SLim},
with $B_n^{(0)} = A_{n, 1}$, $B_n^{(1)} = A_{n, 0}$, and
$\mu_{n, m} = \nu_{n, m}$.
Thus $\mu_{n + 1, n}^{(k, j)} = \nu_{n + 1, n}^{(1 - k, \, 1 - j)}$
for $n \in \Nz$ and $j, k \in \{ 0, 1 \}$.
Condition~(\ref{I__2521_SLim_pi}) in Lemma~\ref{L_2521_SLim}
is immediate because $A_{n, 0}$ is a corner of $K \otimes \OI$.
The maps $\mu_{n + 1, n}^{(k, j)}$ are all injective by
Lemma \ref{L_2605_Cns_OK}(\ref{I_L_2605_Cns_inj}),
which is condition~(\ref{I__2521_SLim_inj}).
For condition~(\ref{I__2521_SLim_full}),
following the notation of
Construction \ref{Cns_2114_OIA}(\ref{I_2114_OIA_PreParts}),
we have $\mu_{n + 1, n}^{(0, 1)} (f_n) = 1 \otimes q_0$.
Since $q_0$ is a \nzp{} in $\OT$, there is $t_0 \in \OT$
such that $t_0^* t_0 = 1$ and $t_0 t_0^* = q_0$.
Then $t = 1 \otimes t_0$ satisfies
$t^* \mu_{n + 1, n}^{(1, 0)} (f_n) t = 1$.
The claim now follows from Lemma~\ref{L_2521_SLim}.

The algebra $A$ satisfies the Universal Coefficient Theorem because it is
a direct limit, with injective maps,
of algebras which satisfy the Universal Coefficient Theorem.

We have $K_1 (A_n) = 0$ for all $n \in \Nz$, so $K_1 (A) = 0$.
Following the notation of
Construction \ref{Cns_2114_OIA}(\ref{I_2114_OIA_A}),
we have $K_0 (A_{n, 1}) = 0$.
Also $K_0 (A_{n, 0}) \cong \Z$ and $[f_0]$ is a generator,
and, in $K_0 (A_{n, 0})$, $[p] = [\io_* (1)] = \io_* (0) = 0$.
Therefore $[f_n] = [f_0]$.
It follows that $K_0 (A_n) \cong \Z$ for all~$n$,
generated by $[1_{A_n}]$,
and $(\nu_{n + 1, n})_* ([1_{A_n}]) = [1_{A_{n + 1}}]$.
So $K_0 (A) \cong \Z$, generated by $[1_A]$.

The classification theorem for purely infinite simple C*-algebras,
Theorem 4.2.4 of~\cite{Ph_PICls},
now implies that $A \cong \OI$.
\end{proof}

\begin{lem}\label{L_2609_FixedPt}
Let $\af \colon S^1 \to \Aut (A)$ be as in
Construction \ref{Cns_2114_OIA}(\ref{I_2114_OIA_Lim}).
Then $A^{\af}$ is purely infinite and simple.
\end{lem}

\begin{proof}
Following the notation of Construction \ref{Cns_2114_OIA}, for $n \in \Nz$
in Lemma~\ref{L_2521_SLim} we take
\[
B_n = (A_n)^{\af^{(n)}},
\qquad
B_n^{(0)} = (A_{n, 0})^{\af^{(n, 0)}},
\andeqn
B_n^{(1)} = (A_{n, 1})^{\af^{(n, 1)}},
\]
and let
$\mu_{n, m}$ be the restriction of $\nu_{n, m}$ to the fixed point algebra.
Then
$\mu_{n + 1, n}^{(k, j)}
 = \nu_{n + 1, n}^{(k, j)} |_{(A_{n, j})^{\af^{(n, j)}}}$
for $n \in \Nz$ and $j, k \in \{ 0, 1 \}$.

In Lemma~\ref{L_2521_SLim},
condition~(\ref{I__2521_SLim_pi})
follows because
$B_n^{(1)} = (C (S^1) \otimes \OT)^{\Lt \otimes \id_{\OT}} \cong \OT$,
and condition~(\ref{I__2521_SLim_inj})
follows from Lemma \ref{L_2605_Cns_OK}(\ref{I_L_2605_Cns_inj})
by restriction.
Condition~(\ref{I__2521_SLim_full}) is a consequence of the
choices in Construction \ref{Cns_2114_OIA}(\ref{I_2114_OIA_A})
made using Lemma~\ref{L_2521_Exist}, and the fact that
$e_{n + 1, \, n + 1} \otimes p$ is \mvnt{} to $e_{0, 0} \otimes p$
in $(K \otimes \OI)^{\id_K \otimes \bt}$.
The conclusion thus follows from Lemma~\ref{L_2521_SLim}.
\end{proof}

\begin{thm}\label{2610_OI_TRP}
Let $\af \colon S^1 \to \Aut (A)$ be as in
Construction \ref{Cns_2114_OIA}(\ref{I_2114_OIA_Lim}).
Then $\af$ has the tracial Rokhlin property with comparison.
\end{thm}

\begin{proof}
We verify the conditions of Definition~\ref{traR}.
Let $F \subseteq A$ and $S \subseteq C (G)$ be finite,
let $\varepsilon > 0$, let $x \in A_{+}$ satisfy $\| x \| = 1$,
and let $y \in (A^{\alpha})_{+} \setminus \{ 0 \}$.
Choose $n \in \Nz$ so large that there is a finite subset
$E \subseteq A_n$
with $\dist (a, \, \nu_{\I, n} (E) ) < \frac{\ep}{2}$
for all $a \in F$ and there is $c \in A_n$
with $\| x - \nu_{\I, n} (c) \| < \frac{\ep}{2}$.
Define $q = (0, 1) \in A_{n + 1}$.
The formula $\ps (f) = f \otimes 1_{\OT}$ defines
an equivariant unital \hm{} $\ps \colon C (S^1) \to A_{n + 1, 1}$,
which we identify with a unital \hm{}
$\ps \colon C (S^1) \to q A_{n + 1} q$.
Now define $p = \nu_{\I, n} (q)$ and
$\ph = \nu_{\I, n + 1} \circ \ps \colon C (S^1) \to A$.
Then $p$ is an $\af$-invariant \pj{} and $\ph$ is
an equivariant unital \hm{} from $C (S^1)$ to~$p A p$.

We claim that for all $a \in F$ and $f \in C (S^1)$ we have
$\| a \ph (f) - \ph (f) a \| < \ep$.
This will verify condition~(\ref{Item_893_FS_equi_cen_multi_approx})
of Definition~\ref{traR}.
Let $a \in F$.
Choose $b \in E$ such that
$\| \nu_{\I, n} (b) - a \| < \frac{\ep}{2}$.
Then $\nu_{n + 1, \, n} (b) \ps (f) = \ps (f) \nu_{n + 1, \, n} (b)$,
so, applying $\nu_{\I, n + 1}$,
\[
\| a \ph (f) - \ph (f) a \|
 \leq 2 \| a - \nu_{\I, n} (b) \| < \ep.
\]
The claim is proved.

We have $1 - p \precsim_A x$ because $x \neq 0$
and $A$ is purely infinite and simple.
Similarly, using Lemma~\ref{L_2609_FixedPt},
$1 - p \precsim_{A^{\af}} y$ and $1 - p \precsim_{A^{\af}} p$.

It remains only to verify condition~(\ref{Item_902_pxp_TRP})
of Definition~\ref{traR}.
Let $y = \nu_{n + 1, \, n} (c) \in A_{n + 1}$ and
write $y = (y_0, y_1)$ with $y_j \in A_{n + 1, \, j}$ for $j = 0, 1$.
The map $A_n \to A_{n + 1, \, 1}$ is injective by
Lemma \ref{L_2605_Cns_OK}(\ref{I_L_2605_Cns_inj}).
Using this at the third step, we have
\[
\begin{split}
\| p x p \|
& > \| p \nu_{\I, \, n} (c) p \| - \frac{\ep}{2}
  = \| q \nu_{n + 1, \, n} (c) q \| - \frac{\ep}{2}
\\
& = \| y_1 \| - \frac{\ep}{2}
  = \| c \| - \frac{\ep}{2}
  > \| x \| - \ep.
\end{split}
\]
This completes the proof.
\end{proof}

\begin{cor}\label{C_2611_EOI}
There exists an action of $S^1$ on $\OI$
which has the tracial Rokhlin property with comparison.
\end{cor}

\begin{proof}
By Lemma~\ref{L_2605_Cns_OI}, the algebra $A$ in Theorem~\ref{2610_OI_TRP}
is isomorphic to~$\OI$.
\end{proof}

\begin{lem}\label{2610_No_fin_RD}
There is no action of $S^1$
on $\OI$ which has finite Rokhlin dimension with commuting towers.
\end{lem}

\begin{proof}
This is part of Corollary 4.23 of \cite{Gar_rokhlin_2017}.
\end{proof}

For $n \in \{ 3, 4, \ldots \}$,
a construction similar to Construction~\ref{Cns_2114_OIA}
presumably gives an action of $S^1$ on ${\mathcal{O}}_{n}$ which has
the tracial Rokhlin property with comparison.
The only changes needed are in the construction of $u$ in the proof
of Lemma~\ref{L_2521_Exist}.
Corollary 4.23 of \cite{Gar_rokhlin_2017} also implies that
there is no action of $S^1$ on ${\mathcal{O}}_{n}$
which has finite Rokhlin dimension with commuting towers.

\begin{prp}\label{P_5327_Tensor}
Let $A_1$ and $A_2$ be simple separable \uca{s},
let $G$ be a second countable compact group,
and let $\bt \colon G \to \Aut (A_1)$ be an action that has
the restricted tracial Rokhlin property with comparison.
Then the action $\af \colon G \to \Aut (A_1 \otimes_{\min} A_2)$,
defined by $\af_g = \bt_g \otimes \id_{A_2}$,
has the restricted tracial Rokhlin property with comparison.
\end{prp}

\begin{proof}
For simplicity of notation, set $A = A_1 \otimes_{\min} A_2$.
Since $A_1$ and $A_2$ are simple, so is~$A$.
It follows from Theorem 3.1 of \cite{MhkPh1}
that $(A_1)^{\bt}$ is simple,
so $(A_1)^{\bt} \otimes_{\min} A_2$ is also simple.
Lemma~\ref{L_5420_Fin_fix} identifies $(A_1)^{\bt} \otimes_{\min} A_2$
with $A^{\af} \subseteq A$,
and we assume this identification throughout the proof.

We verify the conditions of Definition~\ref{traR}.
Let $F \subseteq A$ and $S \subseteq C (G)$ be finite,
let $\varepsilon > 0$, let $x \in A_{+} \setminus \{ 0 \}$,
and let $y \in (A^{\af})_{+} \setminus \{ 0 \}$.
By approximation and algebra, we may assume that there are finite sets
$F_1 \subseteq A_1$ and $F_2 \subseteq A_2$,
contained in the closed unit balls
of these algebras, such that
\[
F = \bigl\{ a_1 \otimes a_2 \colon
   {\mbox{$a_1 \in F_1$ and $a_2 \in F_2$}} \bigr\}.
\]

Lemma~\ref{L_5501_Comp_1} provides $c \in (A_1)_{+} \setminus \{ 0 \}$
and $d \in ((A_1)^{\bt})_{+} \setminus \{ 0 \}$ such that
\begin{equation}\label{Eq_5501_C_cx_dy}
c \otimes 1_{A_2} \precsim_{A} x
\andeqn
d \otimes 1_{A_2} \precsim_{A^{\af}} y.
\end{equation}
Choose $p_0 \in A_1$ and $\ph_0 \colon C (G) \to A_1$
as in Definition~\ref{traR},
with $F_1$ in place of $F$,
with $S$ as given, and with $x = c$ and $y = d$.
Define $p = p_0 \otimes 1_{A_2}$ and define
$\ph \colon C (G) \to p A p$
by $\ph (f) = \ph_0 (f) \otimes 1_{A_2}$ for $f \in C (G)$.
It is easily checked that $\varphi$
is an $(F, S, \varepsilon)$-approximately equivariant
central multiplicative map.
Using (\ref{Eq_5501_C_cx_dy}), $1 - p_0 \precsim_{A_1} c$,
and $1 - p_0 \precsim_{(A_1)^{\bt}} d$, we get
\[
1 - p \precsim_{A} c \otimes 1_{A_2} \precsim_{A} x
\andeqn
1 - p \precsim_{A^{\af}} d \otimes 1_{A_2} \precsim_{A^{\af}} y.
\]
Tensor the relation $1 - p_0 \precsim_{(A_1)^{\bt}} p_0$ with $1_{A_2}$
to get $1 - p \precsim_{A^{\af}} p$.
\end{proof}

\begin{thm}\label{T_2610_Kbg}
Let $\af \colon S^1 \to \Aut (\OI)$
be as in Construction \ref{Cns_2114_OIA}(\ref{I_2114_OIA_Lim}),
using Lemma~\ref{L_2605_Cns_OI} to identify $A$
in Construction \ref{Cns_2114_OIA}(\ref{I_2114_OIA_Lim})
with~$\OI$.
Let $B$ be a unital purely infinite simple separable nuclear \ca.
Then the action
$\bt = \af \otimes \id_B \colon S^1 \to \Aut (\OI \otimes B)$
is conjugate to an action of $S^1$ on~$B$
with the restricted tracial Rokhlin property with comparison.
\end{thm}

\begin{proof}
The action has the restricted tracial Rokhlin property with comparison
by Proposition~\ref{P_5327_Tensor},
and $\OI \otimes B \cong B$ by Theorem~3.15 of~\cite{KP1}.
\end{proof}

\begin{rmk}\label{R_5328_AT_t_OI}
It also follows from Proposition~\ref{P_5327_Tensor}
that if $\af \colon S^1 \to \Aut (A)$ is any of the actions
of Theorem~\ref{T_1919_Ex},
then the action $\zt \mapsto \af_{\zt} \otimes \id_{\OI}$
has the restricted tracial Rokhlin property with comparison.
This gives examples of actions
with the restricted tracial Rokhlin property with comparison
on certain unital purely infinite simple separable nuclear \ca{s},
with special K-groups.
For example, both $K_0$ and $K_1$ must be torsion free,
and can't be finitely generated.

Since tensoring with the trivial action on $\OI$ doesn't change
the equivariant K-theory as an $R (S^1)$-module,
the nonisomorphism result of Theorem~\ref{T_5328_Compare}
still holds for these actions.
It is, however, much easier for the underlying algebras to be isomorphic:
assuming the Universal Coefficient Theorem,
by the classification theorem for Kirchberg algebras
(Theorem 4.2.4 of~\cite{Ph_PICls}; \cite{Kbg_3}),
only the K-theory as abelian groups and $[1_A]$ matter.

By Theorem~\ref{T_1919_Ex}, none of these actions
has finite Rokhlin dimension with commuting towers.
\end{rmk}

We now show that none of the examples in Remark~\ref{R_5328_AT_t_OI}
is equivariantly isomorphic to
any of the examples in proof of Theorem~\ref{T_2610_Kbg}.
This requires the completion of the equivariant K-theory at the
augmentation ideal in the representation ring.
We give a brief summary, referring back to the discussion
of equivariant K-theory before Lemma~\ref{L_1918_K1}.
For more, see the discussions in several parts of~\cite{Ph89a} and
the references there.
We let $G$ be a general compact Lie group, not necessarily connected.
Some of this makes sense more generally.
The only case we use is $G = S^1$,
so that $R (G) = \Z [ \sm, \sm^{-1} ]$ and $I (G)$ is the ideal
generated by $\sm - 1$, as discussed before Lemma~\ref{L_1918_K1}.

For any unital ring $R$ and ideal $I \S R$, for $n \in \N$ we let $I^n$
be the set of all sums of products $x_1 x_2 \cdots x_n$
with $x_1, x_2, \ldots, x_n \in I$.
This is an ideal, and $I^m I^n = I^{m + n}$.
Similarly, if $M$ is an $R$-module and $J \S R$ is an ideal,
then $J M$ is submodule given as the set of all sums of products $x m$
with $x \in J$ and $m \in M$.
For any $R$-module~$M$, the $I$-adic topology on~$M$
has a neighborhood base at $0$
consisting of the sets $M, I M, I^2 M, \ldots$,
and ${\widehat{M}}$ is the Hausdorff completion of $M$ in this topology,
which can also be realized as the inverse limit $\invlim_n M / I^n M$.
Obviously ${\widehat{M}}$ is an $R$-module.
Although we won't explicitly use this fact,
${\widehat{R}}$ is a unital ring
and ${\widehat{M}}$ is an ${\widehat{R}}$-module.

Here, we always take $R = R (G)$ and $I = I (G)$.
Thus, the notation ${\widehat{M}}$ will be unambiguous.

\begin{lem}\label{P_5328_Tfr}
Let $E$ be a torsion free abelian group.
Make $R (S^1) \otimes E$ into an $R (S^1)$-module in the standard way.
Let $x \in [R (S^1) \otimes E]^{\wedge}$,
and suppose that,
with $\sm$ as in the discussion before Lemma~\ref{L_1918_K1},
we have $(\sm - 1) x = 0$.
Then $x = 0$.
\end{lem}

\begin{proof}
We first prove this when $E = \Q$.
In this case, $R (S^1) \otimes E \cong \Q [ \sm, \sm^{-1} ]$,
$\Q [ \sm, \sm^{-1} ]$ is a unital ring, and
$R (S^1) \otimes E$ is in fact a $\Q [ \sm, \sm^{-1} ]$-module
in the standard way.
Also, $I (S^1)^n [R (S^1) \otimes E]$ is the ideal
$\langle (\sm - 1)^n \rangle$ in $\Q [ \sm, \sm^{-1} ]$
generated by $(\sm - 1)^n$.

Recall that if $(M_n)_{n \in \Nz}$ is an inverse system of modules,
with maps $f_n \colon M_{n} \to M_{n - 1}$,
then $\invlim_n M_n$ can be realized as the set of all sequences
$(x_n)_{n \in \Nz} \in \prod_{n = 0}^{\I} M_n$ such that for all $n \in \N$
we have
\begin{equation}\label{Eq_5328_Cons}
f_n (x_{n}) = x_{n - 1}.
\end{equation}
Represent $x \in [R (S^1) \otimes E]^{\wedge}$ as such a sequence
\[
(x_n)_{n \in \Nz}
\in \prod_{n = 0}^{\I} \Q [ \sm, \sm^{-1} ] / \langle (\sm - 1)^n \rangle.
\]
Choose $y_n \in \Q [ \sm, \sm^{-1} ]$
whose image in $\Q [ \sm, \sm^{-1} ] / \langle (\sm - 1)^n \rangle$ is~$x_n$.
The consistency condition~(\ref{Eq_5328_Cons}) becomes
\begin{equation}\label{Eq_5328_Diff}
y_{n + 1} - y_{n} \in \langle (\sm - 1)^{n} \rangle
\end{equation}
for all $n \in \Nz$.
The hypothesis $(\sm - 1) x = 0$ becomes
$(\sm - 1) y_{n + 1} \in \langle (\sm - 1)^{n + 1} \rangle$
for all $n \in \Nz$.
This last relation says there is $z \in \Q [ \sm, \sm^{-1} ]$
such that $(\sm - 1) y_{n + 1} = (\sm - 1)^{n + 1} z$.
Since $\Q [ \sm, \sm^{-1} ]$ is an integral domain,
we get $y_{n + 1} = (\sm - 1)^{n} z$, so
$y_{n + 1} \in \langle (\sm - 1)^{n} \rangle$.
Now $y_{n} \in \langle (\sm - 1)^{n} \rangle$ by~(\ref{Eq_5328_Diff}).
This is true for all $n \in \Nz$, so $x = 0$.
The case $E = \Q$ is done.

Next, suppose that $E$ is a vector space over~$\Q$,
so that $E$ is a dirct sum $E \cong \bigoplus_{j \in J} E_j$,
with $E_j \cong \Q$ for all $j \in J$.
Write $x = (x_n)_{n \in \Nz}$ with $x_n \in E / (\sm - 1)^n E$
for all $n \in \Nz$.
In turn, write $x_n = (x_{n, j})_{j \in J}$
with $x_{n, j} \in E_j / (\sm - 1)^n E_j$ for all $j \in J$
(and with $x_{n, j} = 0$ for all but finitely many $j \in J$).
For $j \in J$, the sequence $z_j = (x_{n, j})_{n \in \Nz}$
is in $[R (S^1) \otimes \Q]^{\wedge}$ and $(\sm - 1) z_j = 0$.
Therefore $z_j = 0$ by the case $E = \Q$.
Since this is true for all $j \in J$, we conclude $x = 0$,
completing the proof when $E$ is a vector space over~$\Q$.

Now consider the general case.

We first claim that, as an abelian group, we have
$R (S^1) / I (S^1)^n \cong \Z^n$.
To see this, define $h_0 \colon \Z^n \to R (S^1)$ by
\[
h_0 (m_0, m_1, \ldots, m_{n - 1}) = \sum_{j = 0}^{n - 1} m_j \sm^j
\]
for $(m_0, m_1, \ldots, m_{n - 1}) \in \Z^n$.
Further let $p \colon R (S^1) \to R (S^1) / I (S^1)^n$ be the quotient map,
and set $h = p \circ h_0$.

The map $h$ is injective because no integer combination of
$1, \sm, \sm^2, \ldots, \sm^{n - 1}$ is a multiple in $\Z [ \sm, \sm^{-1} ]$
of $(\sm - 1)^n$.
For surjectivity, observe that $1 + I (S^1)^n$ is in the range $\mathrm{Ran} (h)$.
We will also show that $\mathrm{Ran} (h)$ is closed under multiplication by
$\sm$ and $\sm^{-1}$.
This will prove the claim.
For multiplication by~$\sm$, since
\[
\sm \cdot h_0 (m_0, m_1, \ldots, m_{n - 1})
 = m_{n - 1} \sm^n + \sum_{j = 0}^{n - 1} m_{j - 1} \sm^j,
\]
we need only show that $\sm^n + I (S^1)^n \in \mathrm{Ran} (h)$,
which follows from $\sm^n - (\sm - 1)^n \in \mathrm{Ran} (h_0)$.
For multiplication by~$\sm^{-1}$, it is similarly enough to show
that $\sm^{-1} + I (S^1)^n \in \mathrm{Ran} (h)$.
This follows from $\sm^{-1} - \sm^{-1} (\sm - 1)^n \in \mathrm{Ran} (h_0)$.
The claim is proved.

We next claim that the map
\[
[R (S^1) \otimes E]^{\wedge} \to [R (S^1) \otimes E \otimes \Q]^{\wedge}
\]
is injective.
Since inverse limits preserve injectivity,
it is enough to prove that, for all $n \in \Nz$, the map
\begin{equation}\label{Eq_5328_Inj_n}
[R (S^1) \otimes E] / I (S^1)^n [R (S^1) \otimes E]
 \to [R (S^1) \otimes E \otimes \Q]
      / I (S^1)^n [R (S^1) \otimes E \otimes \Q]
\end{equation}
is injective.
Rewrite~(\ref{Eq_5328_Inj_n}) as
\[
[R (S^1) / I (S^1)^n] \otimes E
 \to [R (S^1) / I (S^1)^n] \otimes E \otimes \Q,
\]
which by the previous claim is
\begin{equation}\label{Eq_5428_StSt}
\Z^n \otimes E \to \Z^n \otimes E \otimes \Q.
\end{equation}
The map $E \to E \otimes \Q$ is injective since $E$ is torsion free,
so (\ref{Eq_5428_StSt}) is injective, and the claim follows.

Given the claim, let $x \in [R (S^1) \otimes E]^{\wedge}$,
and suppose that $(\sm - 1) x = 0$.
Let $y \in [R (S^1) \otimes E \otimes \Q]^{\wedge}$ be the image of~$x$.
Then $(\sm - 1) y = 0$.
Since $E \otimes \Q$ is a rational vector space,
the previous case implies $y = 0$.
The last claim now shows that $x = 0$.
\end{proof}

For the statement of the next result, let representable K-theory
for $\sm$-\ca{s} be as in~\cite{PhRKTh},
let $E S^1$ be a model for the classifying space of $S^1$
as described at the beginning of Section~2 of~\cite{Ph89a},
let $C (E S^1)$ be the $\sm$-\ca{} of all \cfn{s} on $E S^1$
(not necessarily bounded), and let $\gm \colon S^1 \to \Aut (C (E S^1))$
be the action coming from the action of $S^1$ on $E S^1$.

\begin{lem}\label{L_5328_OI_NTor}
Let $B$ be a unital purely infinite simple separable nuclear \ca.
Assume that $K_0 (B)$ is torsion free.
Let $\bt \colon S^1 \to \Aut (B)$
be the action of Theorem~\ref{T_2610_Kbg}.
Using the notation above,
let $x \in RK_0 \bigl( (B \otimes C (E S^1))^{\bt \otimes \gm} \bigr)$,
and suppose that,
with $\sm$ as in the discussion before Lemma~\ref{L_1918_K1},
we have $(\sm - 1) x = 0$.
Then $x = 0$.
\end{lem}

\begin{proof}
In the notation of Theorem~\ref{T_2610_Kbg},
and using $\OI \otimes B \cong B$, we can rewrite $\bt$ as
\[
\zt \mapsto \af_{\zt} \otimes \id_{\OI} \otimes \id_B
  \in \Aut (\OI \otimes \OI \otimes B).
\]
Then Proposition~4.5 of~\cite{HrsPh1} implies that $\bt$ is homotopic to
the trivial action.
Therefore Corollary~4.2 of~\cite{Ph89a} implies that
$RK_0 \bigl( (B \otimes C (E S^1))^{\bt \otimes \gm} \bigr)$
is unchanged if $\bt$ is replaced by the trivial action~$\io$.
For the trivial action,
$K_0^{S^1} (B) \cong R (S^1) \otimes K_0 (B)$,
with the $R (S^1)$-module structure coming from the first tensor factor.
In Theorem~2.4 of~\cite{Ph89a},
take the finite generating set $F$ of ${\widehat{S^1}}$ to be $\{ \sm \}$.
Then the hypotheses that theorem hold,
because the modules which appear in (**) there are all zero.
Therefore, for the action~$\io$, Theorem~2.4 of~\cite{Ph89a} implies
\[
RK_0 \bigl( (B \otimes C (E S^1))^{\io \otimes \gm} \bigr)
  \cong [R (S^1) \otimes K_0 (B)]^{\wedge}.
\]
For this group, the conclusion holds by Lemma~\ref{P_5328_Tfr}.
\end{proof}

\begin{prp}\label{P_5503_OIB_NoFRDC}
Under the hypotheses of Lemma~\ref{L_5328_OI_NTor},
and assuming that $K_* (B) \neq 0$, the action $\bt$ 
does not have finite Rokhlin dimension with commuting towers.
\end{prp}

\begin{proof}
Suppose $\bt$ has finite Rokhlin dimension with commuting towers.
By Corollary~4.5 of~\cite{Gar_rokhlin_2017},
there is $n \in \N$ such that $I (S^1)^n K_*^{S^1} (B) = 0$.
In (**) in Theorem~2.4 of~\cite{Ph89a},
take the finite generating set $F$ of ${\widehat{S^1}}$ to be $\{ \sm \}$.
The increasing sequence of submodules there stabilizes at
$K_*^{S^1} (B)$ itself at stage~$n$.
Therefore Theorem~2.4 of~\cite{Ph89a} implies that
\[
K_*^{S^1} (B)^{\wedge}
 \cong RK_0 \bigl( (B \otimes C (E S^1))^{\bt \otimes \gm} \bigr).
\]
The relation $I (S^1)^n K_*^{S^1} (B) = 0$ implies that
$K_*^{S^1} (B)^{\wedge} = K_*^{S^1} (B)$, so
\[
I (S^1)^n RK_0 \bigl( (B \otimes C (E S^1))^{\bt \otimes \gm} \bigr) = 0,
\]
contradicting Lemma~\ref{L_5328_OI_NTor}.
\end{proof}

The actions of Theorem~\ref{T_2610_Kbg} will be shown to be distinct
from those of Remark~\ref{R_5328_AT_t_OI} by proving that
the actions of Remark~\ref{R_5328_AT_t_OI}
don't satisfy the conclusion of Lemma~\ref{L_5328_OI_NTor}.

\begin{lem}\label{L_5328_Stab_main}
Let $N \in \N$.
For any $x \in \Z [ \sm, \sm^{-1} ]$ and any $n \in \N$,
we have $(\sm - 1)^n x \in \langle \sm^N - 1 \rangle$
\ifo{} $(\sm - 1) x \in \langle \sm^N - 1 \rangle$.
\end{lem}

\begin{proof}
We claim that if $n \in \N$
then $(\sm - 1)^n x \in \langle \sm^N - 1 \rangle$
\ifo{} there is $z \in \Z [ \sm, \sm^{-1} ]$ such that
$x = (1 + \sm + \cdots + \sm^{N - 1}) z$.
Since the second condition does not depend on~$n$,
the lemma will follow.

We prove the claim.
If $z$ exists, then clearly $(\sm - 1) x \in \langle \sm^N - 1 \rangle$,
whence $(\sm - 1)^n x \in \langle \sm^N - 1 \rangle$.
So assume $(\sm - 1)^n x \in \langle \sm^N - 1 \rangle$.
Choose $y \in \Z [ \sm, \sm^{-1} ]$
such that $(\sm - 1)^n x = (\sm^N - 1) y$.
Choose $r \geq 0$ such that $\sm^r x, \sm^r y \in \Z [\sm]$.
Then, in $\Z [\sm]$, we have
$(\sm - 1)^n \sm^r x = (\sm^N - 1) \sm^r y$.
Since $\Z [\sm]$ is an integral domain,
it follows that
\begin{equation}\label{Eq_5503_smn}
(\sm - 1)^{n - 1} \sm^r x = (1 + \sm + \cdots + \sm^{N - 1}) \sm^r y.
\end{equation}
Now $\sm - 1$ is not a factor of $1 + \sm + \cdots + \sm^{N - 1}$,
and $\Z [\sm]$ is a unique factorization domain,
so there is $w \in \Z [\sm]$ such that
$\sm^r y = (\sm - 1)^{n - 1} w$.
Then, using (\ref{Eq_5503_smn}) at the first step,
\[
\begin{split}
(\sm - 1)^{n} \sm^r x
& = (1 + \sm + \cdots + \sm^{N - 1}) (\sm - 1) \sm^r y
\\
& = (1 + \sm + \cdots + \sm^{N - 1}) (\sm - 1)^{n} w.
\end{split}
\]
Since $\Z [ \sm, \sm^{-1} ]$ is an integral domain,
it follows that
\[
\sm^r x = (1 + \sm + \cdots + \sm^{N - 1}) w.
\]
Therefore $x = (1 + \sm + \cdots + \sm^{N - 1}) \sm^{- r} w$,
which is the claim with $z = \sm^{- r} w$.
\end{proof}

\begin{lem}\label{L_5328_Stab_KG}
Let $\af \colon S^1 \to \Aut (A)$ be any of the actions of
Theorem~\ref{T_1919_Ex}.
Then for every $n \in \N$, we have
\[
\bigl\{ x \in K_0^{S^1} (A) \colon (1 - \sm)^n x = 0 \bigr\}
 = \bigl\{ x \in K_0^{S^1} (A) \colon (1 - \sm) x = 0 \bigr\}.
\]
\end{lem}

\begin{proof}
By Lemma \ref{L_1918_MapRN}(\ref{I_5328_MapRN_K}),
following the notation there,
it suffices to prove this for $K_0^{S^1} (A_n)$ for all $n \in \Nz$
instead of for $K_0^{S^1} (A)$.
Moreover, it suffices to prove this for $K_0^{S^1} (A_{n, 0})$
and $K_0^{S^1} (A_{n, 1})$ separately.
Identifying $R (S^1)$ with $\Z [ \sm, \sm^{-1} ]$,
we have
\[
K_0^{S^1} (A_{n, 0})
 \cong \Z [ \sm, \sm^{-1} ] / \langle \sm^N - 1 \rangle
\andeqn
K_0^{S^1} (A_{n, 1})
 \cong \Z [ \sm, \sm^{-1} ] / \langle \sm - 1 \rangle.
\]
For both of these, the result follows from Lemma~\ref{L_5328_Stab_main},
taking $N = 1$ for $K_0^{S^1} (A_{n, 1})$.
\end{proof}

\begin{cor}\label{C_5328_AT_Cmpl}
Let $\af \colon S^1 \to \Aut (A)$ be one of the actions of
Theorem~\ref{T_1919_Ex}.
Then, referring to the notation before Lemma~\ref{L_5328_OI_NTor},
we have
\[
RK_0 \bigl( (A \otimes \OI \otimes C (E S^1)
    )^{\af \otimes \id_{\OI} \otimes \gm} \bigr)
  \cong K_*^{S^1} (A)^{\wedge}.
\]
\end{cor}

\begin{proof}
We use Theorem~2.4 of~\cite{Ph89a}.
Using the generating set $\{ \sm \}$ for ${\widehat{S^1}}$,
the hypothesis (**) in Theorem~2.4 of~\cite{Ph89a} follows for
$RK_0 \bigl( (A \otimes C (E S^1))^{\af \otimes \gm} \bigr)$
from Lemma~\ref{L_5328_Stab_KG}.
Tensoring everything with the trivial action on $\OI$
does not change the equivariant K-theory,
so this hypothesis holds for the action in the statement as well.
Therefore Theorem~2.4 of~\cite{Ph89a} gives the first
isomorphism in the calculation
\[
RK_0 \bigl( (A \otimes \OI \otimes C (E S^1)
    )^{\af \otimes \id_{\OI} \otimes \gm} \bigr)
  \cong K_*^{S^1} (A \otimes \OI)^{\wedge}
  \cong K_*^{S^1} (A)^{\wedge}.
\]
This completes the proof.
\end{proof}

\begin{thm}\label{T_5328_AT_not_OI}
Let $\af \colon S^1 \to \Aut (A)$ be any of the actions of
Remark~\ref{R_5328_AT_t_OI}, and let
$\bt \colon S^1 \to \Aut (B)$ be any of the actions of
Theorem~\ref{T_2610_Kbg}.
Then $A$ is not equivariantly isomorphic to~$B$.
\end{thm}

All the actions in this theorem
have the restricted tracial Rokhlin property with comparison,
by Theorem~\ref{T_1919_Ex} and Theorem~\ref{T_2610_Kbg},
and (for~$\bt$, assuming $K_* (B)$ is nonzero and torsion free)
none of them has finite Rokhlin dimension with commuting towers,
by Theorem~\ref{T_1919_Ex} and Proposition~\ref{P_5503_OIB_NoFRDC}.

\begin{proof}[Proof of Theorem~\ref{T_5328_AT_not_OI}]
For the proof,
we want to let $A$ be as in Construction~\ref{Cn_1824_S1An}.
Thus, $A$ in the statement of the theorem is now $A \otimes \OI$,
and $\af$ is now $\zt \mapsto \af_{\zt} \otimes \id_{\OI}$.

Suppose there is an equivariant isomorphism.
Then, referring to the notation before Lemma~\ref{L_5328_OI_NTor},
\[
RK_0 \bigl( (A \otimes  \OI \otimes C (E S^1)
     )^{\af \otimes \id_{\OI} \otimes \gm} \bigr)
  \cong RK_0 \bigl( (B \otimes C (E S^1))^{\bt \otimes \gm} \bigr).
\]
By Corollary~\ref{C_5328_AT_Cmpl}, this implies
\[
K_0^{S^1} (A)^{\wedge}
  \cong RK_0 \bigl( (B \otimes C (E S^1))^{\bt \otimes \gm} \bigr).
\]
We now exhibit a nonzero element $x \in K_0^{S^1} (A)^{\wedge}$ such
that $(\sm - 1) x = 0$.
Since $K_0 (A)$ is nonzero and torsion free,
the same is true of $K_0 (B)$,
and this will thus contradict Lemma~\ref{L_5328_OI_NTor},
and show that $A \otimes \OI$ is not equivariantly isomorphic to~$B$.

Following the notation
of Construction \ref{Cn_1824_S1An}(\ref{Item_1824_S1_An}),
recall that
$A_0 = A_{0, 0} \oplus A_{0, 1}$ with
\[
A_{0, 0} = M_{r_0 (n)} (R)
\andeqn
A_{0, 1} = M_{r_1 (0)} \bigl( C (S^1) \bigr).
\]
We have $I (S^1) K_0^{S^1} (A_{0, 1}) = 0$.
Obviously the inclusion $A_{0, 1} \to A_0$ induces an injective map
\[
K_0^{S^1} (A_{0, 1}) / I (S^1) K_0^{S^1} (A_{0, 1})
  \to K_0^{S^1} (A_0) / I (S^1) K_0^{S^1} (A_0).
\]
So Lemma \ref{L_1918_MapRN}(\ref{I_5328_MapRN_ModI})
implies that there is a submodule
of $K_0^{S^1} (A) / I (S^1) K_0^{S^1} (A)$
which is isomorphic to $K_0^{S^1} \bigl( C (S^1) \bigr) \cong \Z$.
Let $y \in K_0^{S^1} (A)$ be the image there of a nonzero
element of $K_0^{S^1} \bigl( C (S^1) \bigr)$.
Then $(\sm - 1) y = 0$ and
the image of~$y$ in $K_0^{S^1} (A) / I (S^1) K_0^{S^1} (A)$ is nonzero.
The inverse limit description of the completion gives a standard
map $K_0^{S^1} (A)^{\wedge} \to K_0^{S^1} (A) / I (S^1) K_0^{S^1} (A)$,
and the map $K_0^{S^1} (A) \to K_0^{S^1} (A) / I (S^1) K_0^{S^1} (A)$
factors through this map.
Let $x \in K_0^{S^1} (A)^{\wedge}$ be the image there of~$y$.
Then $(\sm - 1) x = 0$ but $x \neq 0$ because
its image in $K_0^{S^1} (A) / I (S^1) K_0^{S^1} (A)$ is nonzero.
\end{proof}

\section{Nonexistence}\label{Sec_1919_NonE}

We prove that if $A$ is an in\fd{} simple unital AF~algebra,
there is no direct limit action of $S^1$ on~$A$ that
even has the naive tracial Rokhlin property.
The main ingredient is Proposition~\ref{P_1920_S1DLim},
in which we show that,
for actions of $S^1$, one can require in all variants of the
tracial Rokhlin property
that the $(F, S, \ep)$-equivariant central multiplicative map
$\ph$ be exactly a \hm{} and exactly equivariant.
For direct limit actions of $S^1$, one can also require that it
take values in some algebra in the direct system.
We state a slightly more general form
(also covering finite abelian groups)
in the following proposition.
The statement about general actions is covered, using the trivial
direct system in which all the algebras are~$A$.

\begin{prp}\label{P_1920_S1DLim}
Let $G$ be a (not necessarily connected) compact abelian Lie group such
that $\dim (G) \leq 1$.
Let $\bigl( (A_n)_{n \in \Nz}, \, (\nu_{n, m})_{m \leq n} \bigr)$
be a direct system of \uca{s} with unital injective maps
$\nu_{n, m} \colon A_m \to A_n$.
Set $A = \dirlim_n A_n$, with maps $\nu_{\I, m} \colon A_m \to A$.
Assume we are given actions $\af^{(n)} \colon G \to \Aut (A_n)$
such that the maps $\nu_{n, m}$ are equivariant,
and let $\af$ be the direct limit action $\af = \dirlim \af^{(n)}$.
\begin{enumerate}
\item\label{Item_1920_S1DL_R}
The action $\af$ has the Rokhlin property
\ifo{} for every $N \in \Nz$, every finite set $F \subseteq A_{N}$,
every finite set $S \subseteq C (G)$, and every $\ep > 0$,
there exists $n \geq N$ and a unital equivariant \hm{}
$\ph \colon C (G) \to A_n$ such that
$\| \ph (f) \nu_{n, N} (a) - \nu_{n, N} (a) \ph (f) \| < \ep$
for all $f \in S$ and $a \in F$.
\item\label{Item_1920_S1DL_TRPC}
Suppose $A$ is simple.
Then $\af$ has the tracial Rokhlin property with comparison
\ifo{} for every $N \in \Nz$, every finite set $F \subseteq A_{N}$,
every finite set $S \subseteq C (G)$, every $\ep > 0$,
every $x \in A_{+}$ with $\| x \| = 1$,
and every $y \in (A^{\alpha})_{+} \SM \{ 0 \}$,
there exist $n \geq N$, a projection $p \in (A_n)^{\alpha^{(n)}}$,
and a unital equivariant \hm{} $\ph \colon C (G) \to p A_n p$
such that the following hold.
\begin{enumerate}
\item\label{Item_1920_S1DL_TRPC_Comm}
$\| \ph (f) \nu_{n, N} (a) - \nu_{n, N} (a) \ph (f) \| < \ep$
for all $f \in S$ and $a \in F$.
\item\label{Item_1920_S1DL_TRPC_Sub}
$1 - \nu_{\I, n} (p) \precsim_A x$,
$1 - \nu_{\I, n} (p) \precsim_{A^{\alpha}} y$,
and $1 - p \precsim_{(A_n)^{\alpha^{(n)}}} p$.
\item\label{Item_1920_S1DL_TRPC_Nm}
$\| \nu_{\I, n} (p) x \nu_{\I, n} (p) \| > 1 - \ep$.
\end{enumerate}
\item\label{Item_1920_S1DL_NTRP}
Suppose $A$ is simple.
Then $\af$ has the naive tracial Rokhlin property
\ifo{} for every $N \in \Nz$, every finite set $F \subseteq A_{N}$,
every finite set $S \subseteq C (G)$, every $\ep > 0$,
and every $x \in A_{+}$ with $\| x \| = 1$,
there exist $n \geq N$, a projection $p \in (A_n)^{\alpha^{(n)}}$,
and a unital equivariant \hm{} $\ph \colon C (G) \to p A_n p$
such that the following hold.
\begin{enumerate}
\item\label{Item_1920_S1DL_NTRP_Comm}
$\| \ph (f) \nu_{n, N} (a) - \nu_{n, N} (a) \ph (f) \| < \ep$
for all $f \in S$ and $a \in F$.
\item\label{Item_1920_S1DL_NTRP_Sub}
$1 - \nu_{\I, n} (p) \precsim_A x$.
\item\label{Item_1920_S1DL_NTRP_Nm}
$\| \nu_{\I, n} (p) x \nu_{\I, n} (p) \| > 1 - \ep$.
\end{enumerate}
\setcounter{TmpEnumi}{\value{enumi}}
\end{enumerate}
\end{prp}

\begin{proof}
In all three parts, the fact that the condition implies the
appropriate property follows from the fact that any finite subset of~$A$
can be approximated arbitrarily well by a finite subset
of $\bigcup_{n = 0}^{\I} \nu_{\I, n} (A_n)$.

The reverse directions are deduced from equivariant semiprojectivity
of $C (G)$,
which is Theorem~4.4 of~\cite{Gdla}.
The proofs are similar.
We only do~(\ref{Item_1920_S1DL_TRPC}), which has the most steps.
We give a full proof, since the steps must be done in the right order.

To simplify notation, we assume that $A_k \S A$ for all $k \in \Nz$,
and that the maps $\nu_{l, k}$ and $\nu_{\I, k}$ are all inclusions.
Thus, $A = {\overline{\bigcup_{l = 1}^{\I} A_l}}$.
Moreover, $\af^{(l)}_g = \af_g |_{A_l}$
for all $g \in G$ and $l \in \N$, and we just write~$\af_g$.

Let $N \in \Nz$, let $F \S A_{N}$ be finite,
let $S \S C (G)$ be finite,
let $\ep > 0$, let $x \in A_{+}$ satisfy $\| x \| = 1$,
and let $y \in (A^{\af})_{+} \SM \{ 0 \}$.
Since $\af$ has the \trpc, Proposition 3.10 of \cite{MhkPh1} provides
a \pj{} $e \in (A_{\I, \af} \cap A')^{\af_{\I}}$
and an equivariant unital \hm{}
$\ps \colon C (G) \to e (A_{\I, \af} \cap A') e$
such that the following hold.
\begin{enumerate}
\setcounter{enumi}{\value{TmpEnumi}}
\item\label{It_CSq_2115_1mp_af_sm}
$1 - e$ is $\af$-small in $A_{\I, \af}$.
\item\label{It_CSq_2115_1mp_sm_Aaf}
$1 - e$ is small in $(A^{\af})_{\I}$.
\item\label{It_CSq_2115_1mp_p}
$1 - e \precsim_{(A^{\af})_{\I}} e$.
\item\label{It_CSq_2115_Norm}
Identifying $A$ with its image in $A_{\I, \af}$,
we have $\| e x e\| = 1$.
\setcounter{TmpEnumi}{\value{enumi}}
\end{enumerate}

Choose $c^{(1)}, c^{(2)}, \ldots \in C (G)$ such that
$\bigl\{ c^{(1)}, c^{(2)}, \ldots  \bigr\}$ is dense in $C (G)$.
For $j \in \N$ choose
$d^{(j)} = \bigl( d^{(j)}_m \bigr)_{m \in \N} \in l^{\I}_{\af} (\N, A)$
such that $\pi_A \bigl( d^{(j)} \bigr) = \ps \bigl( c^{(j)} \bigr)$.
Use Lemma 2.9 of \cite{MhkPh1} to lift $e$
to an $\af^{\I}$-invariant \pj{}
$r = (r_m)_{m \in \N} \in l^{\I}_{\af} (\N, A)$.
Since $\af$ is a direct limit action,
$A^{\af} = {\overline{\bigcup_{l = 1}^{\I} (A_l)^{\af^{(l)}} }}$.
Therefore every \pj{} in $A^{\af}$ is a norm limit of \pj{s}
in $\bigcup_{l = 1}^{\I} (A_l)^{\af^{(l)}}$.
For $m \in \N$ choose $l (m) \in \N$ such that
there is a \pj{} $q_m \in (A_{l (m)})^{\af^{(l (m))}}$ satisfying
$\| q_m - r_m \| < \frac{1}{m}$,
and also so large that for $j = 1, 2, \ldots, m$ we have
$\dist \bigl( d^{(j)}_m, \, A_{l (m)} \bigr) < \frac{1}{m}$.
We may assume $l (1) \leq l (2) \leq \cdots$.
Set $q = (q_m)_{m \in \N} \in l^{\I}_{\af} (\N, A)$.
Then $\pi_A (q) = e$.
Let $B \S l^{\I}_{\af} (\N, A)$ be the closed subalgebra consisting
of all sequences $b = (b_m)_{m \in \N} \in l^{\I}_{\af} (\N, A)$
such that for all $m \in \N$ we have $b_m \in q_m A_{l (m)} q_m$.
Let $D = \pi_A (B) \S e A_{\I, \af} e \S A_{\I, \af}$.

We claim that $\ps (C (G)) \S D$.
It suffices to let $j \in \N$
and prove that $\ps \bigl( c^{(j)} \bigr) \in D$.
By construction, for all $m \geq j$ there is $b_m \in A_{l (m)}$
such that $\bigl\| b_m - d^{(j)}_m \bigr\| < \frac{1}{m}$.
Setting $b = (b_m)_{m \in \N} \in l^{\I} (\N, A)$,
we get $\pi_A (b) = \ps \bigl( c^{(j)} \bigr)$.
Since $c_0 (\N, A) \S l^{\I}_{\af} (\N, A)$
and $d^{(j)} \in l^{\I}_{\af} (\N, A)$, this implies that
$b \in l^{\I}_{\af} (\N, A)$.
Therefore also
\[
\ps \bigl( c^{(j)} \bigr)
 = e \ps \bigl( c^{(j)} \bigr) e
 = \pi_A (q) \pi_A (b) \pi_A (q)
 = \pi_A (q b q)
 \in D.
\]
The claim is proved.

For $k \in \N$ let $J_k \S B$ be the ideal consisting of all
sequences $(a_m)_{m \in \N} \in B$ such that $a_m = 0$ for all $m > k$.
Then
\[
J_1 \S J_2 \S \cdots
\andeqn
{\overline{\bigcup_{k = 1}^{\I} J_k}}
 = B \cap c_0 (\N, A) = {\operatorname{Ker}} (\pi_A |_B).
\]
Let $\kp_k \colon B \to B / J_k$ be the quotient map.
Use equivariant semiprojectivity of $C (G)$
(Theorem~4.4 of~\cite{Gdla}) to find $m_0 \in \N$ such that
$\ps \colon C (G) \to D = B / {\operatorname{Ker}} (\pi_A |_B)$
lifts to an equivariant unital \hm{} $\rh \colon C (G) \to B / J_{m_0}$,
that is, $\kp_{m_0} \circ \rh = \ps$.
We may require $m_0 \geq N$.
There is an obvious identification of $B / J_{m_0}$ with the \ca{}
of all bounded sequences $(a_m)_{m > m_0}$
such that $a_m \in q_m A_{l (m)} q_m$ for all $m > m_0$
and the map $g \mapsto (\af_g (a_m))_{m > m_0}$ is \ct.
Under this identification, there are equivariant unital \hm{s}
$\rh_m \colon C (G) \to q_m A_{l (m)} q_m$ such that
$\rh (f) = (\rh_m (f))_{m > m_0}$ for all $f \in C (G)$.

Since $\kp_{m_0} \circ \rh = \ps$, which has range contained in
$A_{\I, \af} \cap A'$,
for all $a \in A$ and $f \in C (G)$ we have
$\lim_{m \to \I} \| \rh_m (f) a - a \rh_m (f) \| = 0$.
In particular, there is $m_1 \geq m_0$ such that for all $m \geq m_1$,
all $f \in S$, and all $a \in F$,
we have $\| \rh_m (f) a - a \rh_m (f) \| < \ep$.
Since $\| e x e \| = 1$ and $\pi_A (q) = e$,
there is $m_2 \in \N$ such that for all $m \geq m_2$,
we have $\| q_m x q_m \| > 1 - \ep$.
Since $1 - e$ is $\af$-small and small in $(A^{\af})_{\I}$,
there is $m_3 \in \N$ such that for all $m \geq m_3$,
we have $1 - q_m \precsim_A x$ and $1 - q_m \precsim_{A^{\af}} y$.
Since $1 - e \precsim_{(A^{\af})_{\I}} e$,
there is $v \in (A^{\af})_{\I}$ such that $v^* v = 1 - e$ and
$v v^* \leq e$.
Then $1 - e = v^* e v$.
Choose $w = (w_{m})_{m \in \N} \in l^{\I} (\N, A^{\af})$
such that $\pi_A (w) = v$.
Then $\lim_{m \to \infty} \| w_m^* q_m w_m - (1 - q_m) \| = 0$.
Therefore there is $m_4 \in \N$ such that for every $m \geq m_4$
we have $\| w_m^* q_m w_m - (1 - q_m) \| < 1$.
Set $m = \max (m_1, m_2, m_3, m_4)$,
take the number $n$ in the statement to be $\max (N, l (m))$,
and set $p = q_m$ and $\ph = \rh_m$.
Lemma 2.11 of \cite{MhkPh1} implies that
$1 - q_m \precsim_{A^{\alpha}} q_m$,
and the rest of the conclusion is clear.
\end{proof}

Let $A$ be a unital \ca.
It is known (in~\cite{garduhfabs} see Theorem 2.17, Example 3.22, and
Example 3.23)
that the existence of an action of $S^1$ on~$A$
with the Rokhlin property implies severe restrictions on~$A$.
One can in fact rule out at least direct limit actions
on a simple unital AF~algebra
with even the naive tracial Rokhlin property.

\begin{prp}\label{P_1921_NoAction}
Let $A$ be an in\fd{} simple unital AF~algebra
and let $G$ be a one dimensional compact abelian Lie group.
There is no direct limit action of $G$ on~$A$,
with respect to any realization of $A$ as a direct limit
of \fd{} \ca{s}, which has the naive tracial Rokhlin property
(Definition~\ref{D_1920_NTRP}).
\end{prp}

\begin{proof}
Let $\bigl( (A_n)_{n \in \Nz}, \, (\nu_{n, m})_{m \leq n} \bigr)$
be an equivariant direct system of \fd{} \ca{s}
with actions $\af^{(n)} \colon G \to \Aut (A)$
and such that the direct limit action $\af = \dirlim \af^{(n)}$
has the naive tracial Rokhlin property.
Choose any \pj{} $x \in A \SM \{ 0, 1 \}$.
Apply Proposition \ref{P_1920_S1DLim}(\ref{Item_1920_S1DL_NTRP})
with $m = 0$, $F = \{ 1 \}$, $S = \{ 1 \}$, $\ep = \frac{1}{2}$,
and $x$ as given.
We get $n \in \Nz$, a nonzero \pj{} $p \in (A_n)^{\alpha^{(n)}}$,
and a unital equivariant \hm{} $\ph \colon C (G) \to p A_n p$.
Since $C (G)$ is $G$-simple and $\ph$ is equivariant,
$\ph$ must be injective.
This is a contradiction because $G$ is infinite and $A_n$ is \fd.
\end{proof}

\begin{cor}\label{C_1921_NoTRPC}
Let $A$ be an in\fd{} simple unital AF~algebra
and let $G$ be a one dimensional compact abelian Lie group.
There is no direct limit action of $G$ on~$A$,
with respect to any realization of $A$ as a direct limit
of \fd{} \ca{s}, which has the tracial Rokhlin property with comparison.
\end{cor}

\begin{proof}
This is immediate from Proposition~\ref{P_1921_NoAction}, because
the tracial Rokhlin property with comparison
implies the naive tracial Rokhlin property.
\end{proof}

\section{Open problems}\label{Sec_5506_Open}

In this section, we collect for easy reference some open problems.

We begin with problems related to the choice of definitions.
We state only two, but there are other related questions.
We expect that there is an example as in the first problem,
but it seems hard to find.

\begin{pbm}\label{Pb_5506_NTRP}
Is there a simple separable \uca~$A$, a compact group~$G$,
and an action $\af \colon G \to \Aut (A)$ that has the
naive tracial Rokhlin property (Definition~\ref{D_1920_NTRP})
but not the restricted tracial Rokhlin property with comparison
(Definition~\ref{traR})?
\end{pbm}

\begin{pbm}\label{Pb_5506_Restr}
Is there a simple separable \uca~$A$, a compact group~$G$,
and as action $\af \colon G \to \Aut (A)$ that has the
restricted tracial Rokhlin property with comparison
but not the tracial Rokhlin property with comparison
(both in Definition~\ref{traR})?
\end{pbm}

Proposition~\ref{tensortrpacts}, on tensor products of actions with the
restricted tracial Rokhlin property with comparison,
suggests consideration of a possible converse.
For example, in a more basic case,
let $A$ and $B$ be simple \uca{s}, let $G$ and $H$ be finite groups,
let $\af \colon G \to \Aut (A)$ and $\bt \colon H \to \Aut (B)$
be actions, and let $\gm \colon G \times H \to \Aut (A \otimes_{\min} B)$
be the tensor product action.
Suppose that $\gm$ has the tracial Rokhlin property,
or even the Rokhlin property.
What does this imply about $\af$ and $\bt$?
The question has an obvious analog for compact $G$ and~$H$,
and the restricted tracial Rokhlin property with comparison
or the Rokhlin property.
It also has an analog for finite Rokhlin dimension
with commuting towers.

One might hope, for example, that if $\gm$ has the Rokhlin property,
then so do $\af$ and~$\bt$.
This is false, even for finite groups,
by Example~4.3 of~\cite{IP_LgRkDim}.
In that example, $B = {\mathcal{O}}_2$, and one expects worse
behavior in the purely infinite case.
The following problem remains open.

\begin{pbm}\label{Pb_5421_From_tens}
Let $A$ and $B$ be simple \uca{s}, let $G$ and $H$ be finite groups,
let $\af \colon G \to \Aut (A)$ and $\bt \colon H \to \Aut (B)$
be actions, and let $\gm \colon G \times H \to \Aut (A \otimes_{\min} B)$
be the tensor product action.
Assume that $A \otimes_{\min} B$ is stably finite.
\begin{enumerate}
\item\label{I_5421_From_tens_R}
If $\gm$ has the Rokhlin property, does it follow that
$\af$ and $\bt$ have the Rokhlin property?
\item\label{I_5421_From_tens_TRP}
If $\gm$ has the tracial Rokhlin property, does it follow that
$\af$ and $\bt$ have the tracial Rokhlin property?
\item\label{I_5501_From_tens_Rdim}
If $\gm$ has finite Rokhlin dimension with commuting towers,
does it follow that
$\af$ and $\bt$ have finite Rokhlin dimension with commuting towers?
\end{enumerate}
\end{pbm}

Motivated by Theorem 3.13 of~\cite{Gdla}, according to which
the Rokhlin property implies that the equivariant K-theory $K_*^G (A)$
is annihilated by the augmentation ideal $I (G)$,
that is, $I (G) K_*^G (A) = 0$,
and by Corollary 4.15 of~\cite{Gar_rokhlin_2017},
according to which for compact Lie groups,
finite Rokhlin dimension with commuting towers implies the existence
of~$n$ such that $I (G)^n K_*^G (A) = 0$,
one might hope to use equivariant K-theory to get partial results.
(See the discussion before Lemma~\ref{L_1918_K1} for equivariant K-theory,
the representation ring, and its augmentation ideal.)
For example, one might hope to get partial results
for~(\ref{I_5421_From_tens_R})
by first showing that $I (G) K_*^G (A) = 0$ and $I (H) K_*^H (B) = 0$.
One would need conditions under which $I (G) K_*^G (A) = 0$
implies the Rokhlin property.
This is of course false for the trivial action of $G$ on~${\mathcal{O}}_2$.
But there are probably some (strong) conditions
on an action $\af \colon G \to \Aut (A)$ for a compact Lie group~$G$,
necessarily including stable finiteness of~$A$ and probably
including special structural conditions on~$\af$,
under which this is true.
See Corollary 4.25 of~\cite{Gar_rokhlin_2017} for such a result
for the very special case of locally representable AF~actions.

The proof of Lemma~\ref{L_5328_OI_NTor}
suggests that $K_*^{S^1} (B)$ is of interest.

\begin{pbm}\label{Pb_5506_KStS1OI}
Let $B$ be a unital purely infinite simple separable nuclear \ca.
Let $\bt \colon S^1 \to \Aut (B)$
be the action of Theorem~\ref{T_2610_Kbg}.
What is $K_*^{S^1} (B)$?
\end{pbm}

We don't know this group even for the case $B = \OI$,
the case in Construction~\ref{Cns_2114_OIA}
(see Lemma~\ref{L_2605_Cns_OI}),
although the naive conjecture is that it is just $R (S^1)$.
We don't even know $K_*^{S^1} (\OI)^{\wedge}$, since we don't
know whether $K_*^{S^1} (\OI)$ satisfies the hypothesis (**) in
Theorem~2.4 of~\cite{Ph89a}.

The hypothesis that $K_* (B)$ be torsion free in
Proposition~\ref{P_5503_OIB_NoFRDC} is gross overkill.
It is chosen because the proof is essentially immediate from the
work already done, and because it holds for the actions to which
we apply the result.
Presumably $K_* (B)$ having an element of infinite order,
or perhaps even less, is good enough.
On the other hand, the proof does not work for $B = {\mathcal{O}}_2$,
suggesting the following problem.

\begin{pbm}\label{Pb_5303_Tens_O2}
In Proposition~\ref{P_5503_OIB_NoFRDC}, take $B = {\mathcal{O}}_2$.
Does the resulting action have
finite Rokhlin dimension with commuting towers?
\end{pbm}

We know of no example in the stably finite case that is
K-theoretically similar to that of Construction~\ref{Cns_2114_OIA},
say in the sense of Lemma~\ref{L_5328_OI_NTor}.
In view of the proof of Lemma~\ref{L_5328_OI_NTor}, perhaps the
appropriate question is as follows.

\begin{pbm}\label{Pb_5506_StFinHty}
Are there a stably finite simple separable unital \ca~$A$
and an action $\af \colon S^1 \to \Aut (A)$
that has the restricted tracial Rokhlin property with comparison
but is not homotopic to the trivial action?
\end{pbm}

As mentioned in the introduction to Section~\ref{Sec_2114_OI},
it should also be not too hard to generalize Construction \ref{Cns_2114_OIA}
to actions of $(S^1)^m$ for $m \in \{ 2, 3, 4, \ldots, \I \}$.
It seems harder to deal with nonabelian groups.

\begin{pbm}\label{Pb_2114_NC}
Find an action of a nonabelian connected
compact Lie group (such as ${\operatorname{SU}} (2)$) on $\OI$
that has the restricted tracial Rokhlin property with comparison.
\end{pbm}

No such action can have finite Rokhlin dimension with commuting towers,
by Theorem~4.6 of~\cite{HrsPh1}.

\begin{pbm}\label{Pb_5503_NC2}
Find some simple separable \uca~$A$
and an action of a nonabelian connected compact Lie group on $A$
that has the restricted tracial Rokhlin property with comparison,
but does not have finite Rokhlin dimension with commuting towers.
\end{pbm}

On the other hand, there should be no difficulty in constructing
an analog of Example~\ref{Cn_1X17_Z2Inf} with a nonabelian group.

\end{document}